\documentclass[11pt,a4paper]{article}

\usepackage[utf8]{inputenc}
\usepackage{amsmath,amsfonts,amssymb}
\usepackage{amsthm,upref}
\usepackage[a4paper]{geometry}
\usepackage{color} 
\usepackage[dvipsnames]{xcolor}
\usepackage{graphicx}
\usepackage{subfigure}
\usepackage{tabularx,float}
\usepackage[hidelinks]{hyperref}
\usepackage{comment} 
\usepackage{tabularx} 

\newtheorem{thm}{Theorem}[section]
\newtheorem{definition}[thm]{Definition}

\newtheorem{proposition}[thm]{Proposition}
\newtheorem{corollary}[thm]{Corollary}
\newtheorem{lemma}[thm]{Lemma}
\newtheorem{remark}[thm]{Remark}
\newtheorem{assumption}{Assumption}

\usepackage{todonotes}

\newcommand{\secref}[1]{Section~\ref{sec:#1}}

\newcommand{\seclab}[1]{\label{sec:#1}}
\newcommand{\eqlab}[1]{\label{eq:#1}}
\renewcommand{\eqref}[1]{(\ref{eq:#1})}

\newcommand{\figref}[1]{Fig.~\ref{fig:#1}}
\newcommand{\figlab}[1]{\label{fig:#1}}
\newcommand{\propref}[1]{Proposition~\ref{proposition:#1}}
\newcommand{\proplab}[1]{\label{proposition:#1}}

\newcommand{\lemmaref}[1]{Lemma~\ref{lemma:#1}}
\newcommand{\lemmalab}[1]{\label{lemma:#1}}
\newcommand{\remref}[1]{Remark~\ref{remark:#1}}
\newcommand{\remlab}[1]{\label{remark:#1}}
\newcommand{\thmref}[1]{Theorem~\ref{theorem:#1}}
\newcommand{\thmlab}[1]{\label{theorem:#1}}

\usepackage{enumitem}
\usepackage{tikz}

\newcommand\response[1]{{\color{black}{#1}}}
\newcommand\responsenew[1]{{\color{black}{#1}}}
\title{Blowup analysis of a hysteresis model based upon singular perturbations} 
\author {Kristiansen, K. U.} 
\date{%
Department of Applied Mathematics and Computer Science, \\
Technical University of Denmark, \\
2800 Kgs. Lyngby, \\
Denmark,\\
krkri@dtu.dk\\
\vspace{0.5cm}
\today
}
\begin{document}
\maketitle

%

\begin{abstract}
In this paper, we provide a geometric analysis of a new hysteresis model that is based upon singular perturbations. Here hysteresis refers to a type of regularization of piecewise smooth differential equations where the past of a trajectory, in a small neighborhood of the discontinuity set, determines the vector-field at present. In fact, in the limit where the neighborhood of the discontinuity  \response{vanishes}, hysteresis converges in an appropriate sense to Filippov's sliding vector-field. Recently \response{(2022)}, however, \response{Bonet and Seara} showed that hysteresis, in contrast to regularization through smoothing, leads to chaos in the regularization of grazing bifurcations, even in two dimensions. The hysteresis model we analyze in the present paper -- which was developed by \response{Bonet et al in a paper from 2017} as an attempt to unify different regularizations of piecewise smooth systems -- involves two singular perturbation parameters and includes a combination of slow-fast and nonsmooth effects. The description of this model is therefore -- from the perspective of singular perturbation theory -- challenging, even in two dimensions. Using blowup as our main technical tool, we prove existence of an invariant cylinder carrying fast dynamics in the azimuthal direction and a slow drift in the axial direction. We find that the slow drift is given by Filippov's sliding \responsenew{vector-field} to leading order. Moreover, in the case of grazing, we identify two important parameter regimes that relate the model to smoothing (through a saddle-node bifurcation of limit cycles) and hysteresis (through chaotic dynamics, due to a folded saddle and a novel return mechanism).
\end{abstract}
%


\tableofcontents
\section{Introduction}
In this paper, we consider piecewise smooth (PWS) systems of the following form:
\begin{align}
 \dot z &= \begin{cases}
            Z_+(z) & y>0\\
            Z_-(z) & y<0
           \end{cases},\eqlab{zPWS}
\end{align}
where $z=(x,y)\in \mathbb R^{n+1}$, $Z_\pm(z)=(X_\pm(z),Y_\pm(z))$. The set $\Sigma:\,y=0$ is called the discontinuity set or switching manifold. In a more general setting, one could define the switching manifold $\Sigma$ as a smooth hypersurface $h(z)=0$ for some regular function $h:\mathbb R^{n+1}\rightarrow \mathbb R$. Locally, however, we can always introduce coordinates $(x,y)$ so that $h(x,y)=y$. We will suppose that $Z_\pm$ are smooth vector-fields, each defined in a neighborhood of $\Sigma$. 

The basic problem of \eqref{zPWS} is how to define solutions of \eqref{zPWS} \response{on} $\Sigma$. The case when $Y_+(x,0)<0$ and $Y_-(x,0)>0$ is most interesting from a technical point of view, because in this case orbits of either system $\dot z = Z_\pm (z)$ reach $\Sigma$ in finite time, see \figref{pws}. This is known as (stable) sliding. To be able to define a forward flow, a vector-field must be assigned on $\Sigma$. The most common way to do this, is through the Filippov vector-field defined by
%
%
\begin{align}
 \response{X_{sl}(x):=X_+(x,0)p(x)+X_-(x,0)(1-p(x)),\quad p(x):= \frac{ Y_-(x,0)}{Y_-(x,0)- Y_+(x,0)}\in (0,1).\eqlab{Xsl}}
\end{align}
The PWS systems, where \eqref{Xsl} is assigned along the subset of the switching manifold with $Y_+(x,0)Y_-(x,0)<0$, are called Filippov systems. Filippov systems may also be viewed more abstractly in the sense of differential inclusions \cite{filippov1988differential}. They occur naturally in mechanics, e.g. in friction modelling \cite{bossolini2017a,kristiansen2021a}. However, even such mechanical models may suffer from nonuniqueness of solutions, and a meaningful forward flow may not be defined at all points \cite{bossolini2017a}.

\begin{figure}
\begin{center}
\includegraphics[width=.45\textwidth]{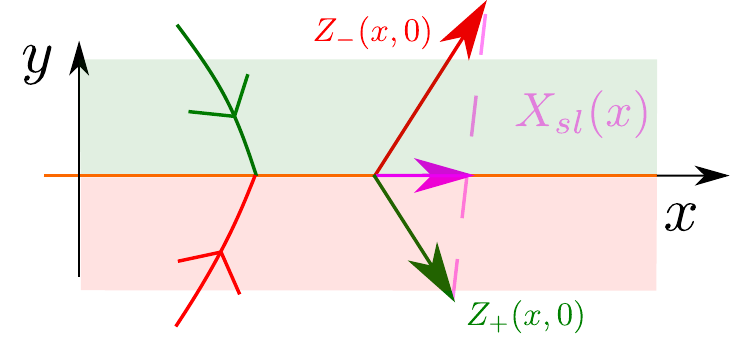}
\end{center}
 \caption{Illustration of Filippov's sliding vector-field $X_{sl}$ in the case of stable sliding.}
  \figlab{pws}
              \end{figure}

From a modelling perspective, nonuniqueness may be interpreted as an insufficient model where additional information or complexity has to be added in order to select a unique forward trajectory. From this point of view, it is therefore important to study regularizations of \eqref{zPWS}. There are two basic \responsenew{examples of} regularizations of \eqref{zPWS}, one is smoothing and another one is hysteresis. In this paper, we shall -- following Sotomayor and Teixeira \cite{Sotomayor96} -- define regularization by smoothing as replacing \eqref{zPWS} with an $\epsilon$-family of smooth systems:

\begin{equation}\eqlab{xysm}
\begin{aligned}
 \dot x &=X\left(z,\phi\left(\frac{y}{\epsilon}\right)\right),\\
 \dot y &=Y\left(z,\phi\left(\frac{y}{\epsilon}\right)\right),
\end{aligned}
\end{equation}
with $Z(z,p)=(X(z,p),Y(z,p))$ defined by:
\begin{align}
 Z(z,p):=Z_+(z) p + Z_-(z)(1-p),\eqlab{Zpaf}
\end{align}
for $Z_\pm =(X_\pm,Y_\pm)$. Regarding the function $\phi$ in \eqref{xysm} , we assume the following assumption, so that \eqref{xysm} approaches \eqref{zPWS} pointwise for $\epsilon\rightarrow 0$ for $y\ne 0$. 
\begin{assumption}
	\label{assumption:1}
	The smooth `regularization function' $\phi : \mathbb R \to \mathbb R$ satisfies the monotonicity condition
	\[
	\response{\phi'(s)} > 0 ,
	\]
	for all $s \in \mathbb R$ and, moreover,
	\begin{equation}
	\label{eq:mono}
	\phi(s) \to
	\begin{cases}
	1  & \text{for } s \to \infty , \\
	0  & \text{for } s \to -\infty .
	\end{cases}
	\end{equation}
\end{assumption}
On the other hand, in hysteresis, solutions of $\dot z=Z_+(z)$ are extended  to $y=-\alpha$ before switching to $\dot z=Z_-(z)$. Here $\alpha>0$ is some small parameter.  Solutions of $\dot z=Z_-(z)$ are similarly extended to $y=\alpha$ before switching occurs, see \figref{pws_hysteresis_0}.

In smoothing, we basically introduce a boundary layer of order $\mathcal O(\epsilon)$ around $y=0$ where $p=\phi(y\epsilon^{-1})$ changes by an $\mathcal O(1)$-amount. From this point of view, it is also useful to think of hysteresis as introducing a ``negative'' boundary layer around $y=0$ of size $2\alpha$. 
\begin{figure}
\begin{center}
\includegraphics[width=.45\textwidth]{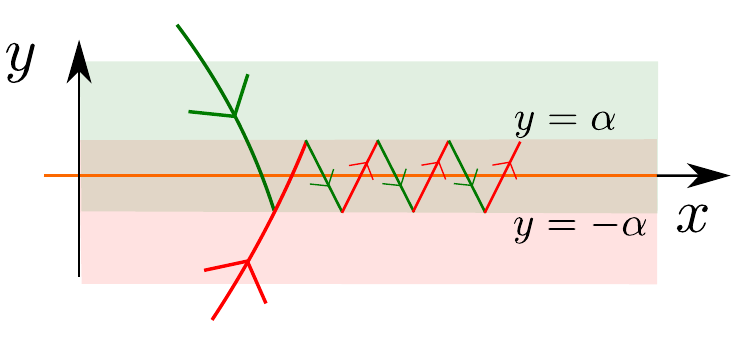} \end{center}
 \caption{Illustration of hysteresis in the case of stable sliding. In comparison with regularization by smoothing, hysteresis can be interpreted as introducing a \textit{negative} boundary layer, the size of which is given by $2\alpha$. }
 \figlab{pws_hysteresis_0}
              \end{figure}
              \begin{figure}
\begin{center}
\includegraphics[width=.65\textwidth]{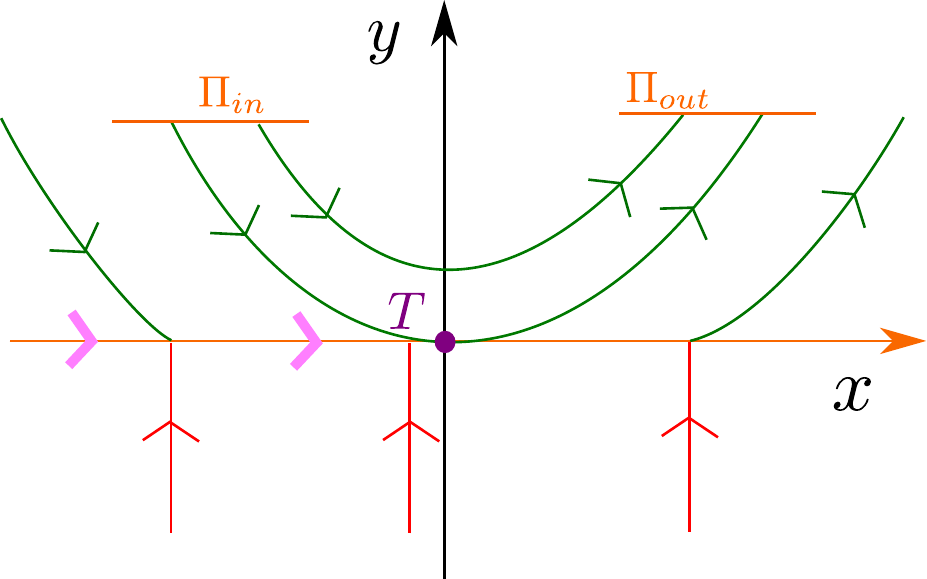}
 \end{center}
 \caption{Illustration of the planar visible fold, which is important for the grazing bifurcation. The visible fold separates the switching manifold $\Sigma$ into stable sliding ($x<0$) and crossing points ($x>0$).}
  \figlab{visible0}
              \end{figure}
              
In both types of regularizations, forward solutions can be uniquely defined for all $\epsilon>0,\alpha>0$, respectively, and in some cases, $\epsilon,\alpha\rightarrow 0$ can be analyzed. For example, in \cite{kristiansen2018a} the authors studied the regularization by smoothing of a visible-invisible two-fold in $\mathbb R^3$. The two-fold is a well-known singularity of Filippov systems that give rise to nonuniqueness of solutions. The results of \cite{kristiansen2018a} showed that the smooth system has a well-defined limit for $\epsilon\rightarrow 0$ which selects a distinguished forward trajectory through the two-fold. In this way, the nonuniqueness has (in a certain sense) been resolved.

The reference \cite{2019arXiv190806781U} also studied regularization by smoothing but considered the planar grazing bifurcation scenario, where a limit cycle of $Z_+$ grazes $\Sigma$ while $Z_-$ remains transverse and points towards $\Sigma$. The results showed, in line with \cite{reves_regularization_2014} and analysis based upon the associated Filippov system \cite{Kuznetsov2003}, that in the case of a repelling limit cycle, the smooth system has a locally unique saddle-node bifurcation of limit cycles. The analysis rested upon a careful description of the local dynamics, through a local transition map $\mathcal P_{loc}:\Pi_{in}\rightarrow \Pi_{out}$, near the grazing point which is given by a visible fold, see illustration in \figref{visible0}. These results, together with \cite{jelbart2021a,jelbart2021c,jelbart2021b,uldall2021a} working on similar systems, were obtained by adapting methods from Geometric Singular Perturbation Theory \cite{fen3,jones_1995}. In particular, these references use a modification of the blowup method \cite{dumortier1996a,krupa_extending_2001} to gain smoothness of systems of the form \eqref{xysm}. 

Recently, in \cite{rev2021a} the authors performed a related study of the grazing bifurcation, but using regularization by hysteresis instead. Interestingly, the results are completely different in this case. In fact,  hysteresis leads to chaotic dynamics for any $0<\alpha\ll 1$ under the same assumptions. 

In this paper, we consider a new regularization of \eqref{zPWS} developed by \cite{bonet2017a}:
\begin{equation}\eqlab{xypmodel}
\begin{aligned}
 \dot x &= X(z,p),\\
 \dot y &=Y(z,p),\\
 \epsilon \vert \alpha\vert \dot p & =\phi\left(\frac{y+  \alpha p}{\epsilon \vert\alpha\vert }\right)-p,
\end{aligned}
\end{equation}
for $0<\epsilon,\vert\alpha\vert\ll 1$.\footnote{In contrast to \cite{bonet2017a}, we write their $\kappa$ as $\epsilon$.} Notice that the dimension of \eqref{xypmodel} is one greater than the dimension of \eqref{zPWS}. The connection between \eqref{xypmodel} and \eqref{zPWS} at the pointwise level is as follows: By assumption \ref{assumption:1}, \eqref{xypmodel} converges pointwise to 
\begin{equation}
\begin{aligned}
 \dot z &= Z(z,p),\\
           p &=\begin{cases}
                1 &y>0\\
                0 & y<0
               \end{cases},
\end{aligned}
\end{equation}
for $\epsilon,\alpha\rightarrow 0$ for $y\ne 0$,
which upon using \eqref{Zpaf} projects to \eqref{zPWS}. 
This model was introduced by \cite{bonet2017a}, in a general framework where $Z(z,p)$ depends nonlinearly on $p$, with the purpose to incorporate smoothing and hysteresis in one single unified framework. The authors present asymptotic results for both $\alpha<0$ and $\alpha>0$, connecting the dynamics of \eqref{xypmodel} in the latter case with Filippov's sliding vector-field. Since trajectories in hysteresis cross each-other in the ``negative boundary layer'', recall \figref{pws_hysteresis_0}, it makes sense that the smooth model \eqref{xypmodel} is defined in an extended space.  

In \cite{bonet2017a}, the authors consider functions $\phi$, see assumption \ref{assumption:1}, that reach $0$ and $1$ at finite values:
\begin{align}
 \response{\phi(s)=\begin{cases}
          1 & \text{for all $s\ge 1$}\\
        0 & \text{for all $s\ge -1$}
         \end{cases},\quad \phi'(s)>0\quad \text{for all $s\in (-1,1)$.}\eqlab{Sotomayor}}
\end{align}
Such functions have -- following \cite{Sotomayor96} -- been called \responsenew{Sotomayor-Teixeira regularization functions}. In this paper, also to exemplify the power of our approach, we will follow \cite{jelbart2021a,jelbart2021c,jelbart2021b,uldall2021a,2019arXiv190806781U} and consider general regularization functions that are truly asymptotic, like analytic ones, e.g.
\begin{align}
 \phi(s)= \frac12 +\frac{1}{\pi}\arctan(s).\eqlab{arctan}
\end{align}
For this purpose, we add the following technical assumption:
\begin{assumption}
	\label{assumption:2}
	The regularization function $\phi$ has algebraic decay as $s \to \pm\infty$, i.e. there exists a $k \in \mathbb N$ and smooth functions $\phi_\pm:[0,\infty \response{)} \to [0,\infty)$ such that
	\begin{equation}
	\label{eq:reg_asymptotics}
	\phi(s^{-1})=
	\begin{cases}
	1-\phi_+(s)s^k
	, &\qquad s >0\,,\\
	\phi_-(-s)(-s)^k , &\qquad  s<0\,,
	\end{cases}
	\end{equation}
and 
	\begin{equation}
	\eqlab{eq:beta}
	\beta_+ :=\phi_+(0)>0,\,\beta_- = \phi_-(0)>0.
	\end{equation}
\end{assumption}
There could be different $k$-values $k_\pm$ for $s\response{\rightarrow} \pm \infty$, respectively, but for simplicity we take these to be identical. In the following, $\beta_-$ will play little role so we will therefore for simplicity write $\beta_+$ as $\beta$. \response{For \eqref{arctan}, $k=1$ and $\beta=\frac{1}{\pi}$.}

\begin{lemma}\lemmalab{Zpreg}
Suppose that assumption \ref{assumption:2} holds and consider \eqref{xypmodel} on a compact domain $\mathcal U_+$ upon which $y >0$. This system has an attracting slow manifold $S_{\epsilon,\alpha}$ -- of the graph form $p=1+\mathcal O(\epsilon^k \vert \alpha\vert^k)$ -- which carries the reduced problem:
\begin{equation}\eqlab{xypmodelslowplus}
\begin{aligned}
 \dot z &=Z_+(z)+\mathcal O(\epsilon^k \vert \alpha\vert^k),
\end{aligned}
\end{equation}
This holds uniformly and smoothly on the compact subset $\mathcal U_+$ and on this set \eqref{xypmodelslowplus} is therefore a smooth $\mathcal O(\epsilon^k \vert \alpha\vert^k)$-perturbation of $\dot z=Z_+(z)$. 
\end{lemma}
\begin{proof}
 \responsenew{By assumption \ref{assumption:2}, we have the following on $\mathcal U_+$
 \begin{align*}
 x' &=\epsilon \vert \alpha\vert X(z,p),\\
 y' &=\epsilon \vert \alpha\vert  Y(z,p),\\
 p' &= 1-p-\left(\frac{\epsilon\vert \alpha\vert}{y+\alpha p}\right)^k \phi_+\left(\frac{\epsilon\vert \alpha\vert}{y+\alpha p}\right),
\end{align*}
in terms of the fast time defined by $()'=\epsilon \vert \alpha\vert \dot{()}$. Setting $\epsilon=0,\alpha=0$ on $y>0$ gives the layer problem
\begin{align*}
 x' &=0,\\
 y' &=0,\\
 p' &= 1-p,
\end{align*}
for which $S_0=\{(x,y,p)\in \mathcal U_+\,\vert\,p=1\}$ clearly is a normally hyperbolic and attracting critical manifold. The result then follows from Fenichel's theory \cite{fen3}, see also \cite[Theorem 2]{jones_1995}).}
\end{proof}

A similar result clearly holds within $y<0$. The objective of our analysis is to uncover what occurs near $y=0$.


It is possible to obtain some intuition on the dynamics of \eqref{xypmodel} by looking at the equation for the $p$-nullcline:
\begin{align}
\phi\left(\frac{y+  \alpha p}{\epsilon \vert\alpha\vert }\right)-p = 0,\eqlab{peqn}
\end{align}
see also \cite[Fig. 3]{bonet2017a}. Given that $p$ is a fast variable of \eqref{xypmodel}, it is tempting to think about the set defined by \eqref{peqn} as a critical manifold (ignoring for the moment that it depends on $\epsilon$ and $\alpha$ in a singular way). We can solve \eqref{peqn} for $y$ as a function of $p,\epsilon$ and $\alpha$ by using $\phi^{-1}$. This gives
\begin{align}
y  = F(p,\epsilon,\alpha):=\epsilon \vert \alpha \vert \phi^{-1}(p) -\alpha .\eqlab{Fgraph}
\end{align}
Now, the graph $y=F(p,\epsilon,\alpha)$, $p\in (0,1)$, of the function $F$ has fold points at $(y,p)=(y_f,p_f)$ whenever
\begin{align*}
 F'_p(p_f,\epsilon,\alpha) =\epsilon \vert \alpha \vert \frac{1}{\phi'\left( \phi^{-1}(p_f)\right)} -\alpha =0,\quad F''_{pp}(p_f,\epsilon,\alpha)\ne 0,
\end{align*}
\response{see  \figref{Fgraph}.}
Since, the former condition can be written as
\begin{align*}
 \phi'\left( \phi^{-1}(p_f)\right) =\epsilon \operatorname{sign}\alpha,
\end{align*}
we only have fold points (using assumption \ref{assumption:1}) for $\operatorname{sign}\alpha=1$. In this case, assuming that assumption \ref{assumption:2} holds, it is a simple calculation to show that there exist two fold points $(y_f^\pm,p_f^\pm)$ and that these have the following asymptotics
\begin{align*}
(y_f^-,p_f^-) &= (\mathcal O(\alpha \epsilon^{\frac{k}{k+1}}),\mathcal O(\epsilon^{\frac{k}{k+1}})),\\
(y_f^+,p_f^+) &= (\mathcal O(\alpha),1+\mathcal O(\epsilon^{\frac{k}{k+1}})),
\end{align*}
with respect to $\epsilon,\alpha\rightarrow 0$,
%
 near $p=0$ and $p=1$, respectively. The leading order terms can expressed in terms of $\beta$ and $k$, see also \eqref{fp} below.
 
 In this paper, we will focus on $\alpha>0$; the case $\alpha<0$ is simpler and can be handled by the same methods. 

   \response{The graph of $F$ has an $S$-shape, see \figref{Fgraph}}, but since $F$ converges pointwise to $0$ for $p\in (0,1)$ as $\epsilon,\alpha\rightarrow 0$, the folds (black disks) are only visible upon magnification/blowup of $y$.  \responsenew{Moreover, due to the singular nature it is apriori unclear whether this folded structure behaves like folds in slow-fast systems, see e.g. \cite{szmolyan_canards_2001,szmolyan2004a}}. Nevertheless, if we continue to think of the graph of $F$ as a critical manifold and $p$ as the fast variable, the $S$-shape structure hints at a hysteresis-like mechanism for fast transitions between $p=0$ and $p=1$ through the fold points.   The folded structure becomes more profound for larger values of $\alpha>0$.  Our blowup approach  \response{(see \secref{lem:P22})} will describe this in further details and motivate coordinates, including $(\nu_{213},p_{213},\rho_{213})$ defined by
  \begin{align}\eqlab{213}
   \begin{cases}
    y &= -\alpha (1+\rho_{213}^k p_{213})+\alpha \rho_{213}^k \nu_{213},\\
    p&=(1+\rho_{213}^k p_{213}),\\
    \epsilon &=\rho_{213}^{k+1},
   \end{cases}
  \end{align}
   that can be used to describe the dynamics in a rigorous way. In fact, \eqref{213} leads to the following equations: $\dot x=0$ and
\begin{equation}\nonumber
\begin{aligned}
 \dot \nu_{213} &=-\nu_{213} \left(\beta \nu_{213}^{-k}+p_{213}\right),\\
 \dot p_{213} &=-\nu_{213}\left(\beta \nu_{213}^{-k}+p_{213}\right),
\end{aligned}
\end{equation}
with $\beta=\phi_+(0)$, 
in the dual singular limit $\alpha,\rho_{213}\rightarrow 0$. This system has the set $R_{213}$ defined by $p_{213}=-\beta \nu_{213}^{-k}$, $\nu_{213}>0$, as a manifold of equilibria. $R_{213}$ is normally hyperbolic everywhere except at 
\begin{align}\eqlab{fp}
 p_{213,f}=-\beta \left(k\beta\right)^{-\frac{k}{k+1}}, \quad \nu_{213,f}=\left(k\beta\right)^{\frac{1}{k+1}},
\end{align}
which is a fold point,
see the left subfigure in \figref{pws_hysteresis_3} below for an illustration. In the case of sliding, we find that the fold point is a simple jump point, whereas in the case of grazing it becomes a canard point (folded saddle singularity, within a certain parameter regime). Notice that the location of the fold point \eqref{fp} is by \eqref{213} in agreement with the asymptotics for $(y_f^+,p_f^+)$ above. 
  
 
  \begin{figure}
\begin{center}
\includegraphics[width=.65\textwidth]{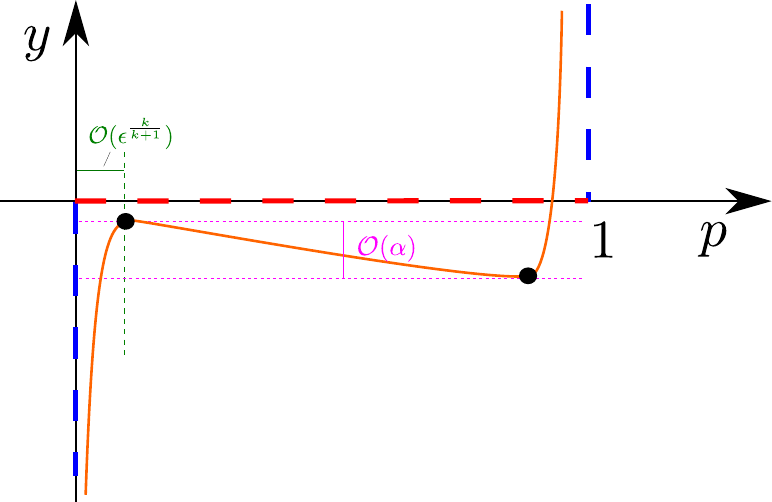}
 \end{center} \caption{Illustration of the graph of $F$ (in orange), see \eqref{Fgraph}. For $\epsilon,\alpha>0$ small enough, $F$ has two folds (black disks) located $\mathcal O(\epsilon^{\frac{k}{k+1}})$-close to $p=0$ and $p=1$ respectively. The parameter $\alpha$ measures the separation of the fold points in the $y$-direction. The pointwise limit for $\epsilon,\alpha\rightarrow 0$ is indicated in red and blue. In particular, the blue lines at $p=0$ and $p=1$ are asymptotes for the graph \eqref{Fgraph}.}
 \figlab{Fgraph}
              \end{figure}

We anticipate that our approach will have general interest. 
It is clear that \eqref{xypmodel} involves a combination of slow-fast and nonsmooth effects. The analysis of such system seems to be rare. The reference \cite{kristiansen2021a} offers an exception. This manuscript studied a model of a friction oscillator, also of the form \eqref{xysm} but with $\phi$ non-monotone, in the presence of a time scale separation. The combination of slow-fast and nonsmooth effects was shown to lead to chaotic dynamics through a horseshoe obtained through a folded saddle singularity \cite{szmolyan_canards_2001} and a novel return mechanism. We will obtain something similar for \eqref{xypmodel} in the case of the grazing bifurcation. However, the analysis of \eqref{xypmodel} is more involved, and as opposed to the system in \cite{szmolyan_canards_2001}, the slow-fast and nonsmooth effects are more combined. On top of that, \eqref{xypmodel} involves two small parameters $0<\epsilon,\vert\alpha\vert\ll 1$. 

We hope that our analysis of \eqref{xypmodel} will provide a template for the analysis of similar systems, with several singular parameters as well as a combination of slow-fast and nonsmooth phenomena. At the same time, it is our anticipation that our results, in particular on the grazing bifurcation and the unification of known results on smoothing and hysteresis, will stimulate further research on the model \eqref{xypmodel}.

\subsection{Overview}
The remainder of the paper is organized as follows. 
In \secref{blowup}, we present the blowup approach which will form the basis for our analysis of \eqref{xypmodel} with $\alpha>0$. This will include a review of this method in the context of regularization by smoothing. Although the results on smoothing are well-known to experts, we believe that the use of the blowup approach in this context provides a good platform to extend it to the analysis of \eqref{xypmodel}. In \secref{main1}, we then present the first main results, summarized in \thmref{main1}, on the dynamics of \eqref{xypmodel} for $\epsilon,\alpha>0$ both sufficiently small in the case of stable sliding. In proving this result, we also lay out the geometry of the dynamics using our blowup approach. In \secref{lem:P22}, we prove an important lemma on a return map (resting upon the description of a transition map near the blowup of the folds in \figref{Fgraph}) that is used to prove \thmref{main1}. Finally, in \secref{main2}, we turn our attention to the grazing bifurcation. 
The main results of this section are stated in \thmref{main2}.
In particular, for the grazing bifurcation,
we identify two separate parameter regimes in the $(\epsilon,\alpha)$-plane. In one regime, we obtain a locally unique saddle-node bifurcation, as in the regularization by smoothing \cite{2019arXiv190806781U}, while in the other regime, we obtain chaotic dynamics, consistent with the results in \cite{rev2021a} on regularization by hysteresis. The chaotic dynamics is obtained through a horseshoe and folded saddle singularities of the blowup of the folds in \figref{Fgraph}. 

\section{Blowup}\seclab{blowup}
The blowup approach \cite{dumortier1996a,krupa_extending_2001}, which in its original framework was developed as a method to deal with lack of hyperbolicity, has recently been adapted \cite{kristiansen2018a} to deal with smooth systems approaching nonsmooth ones. Within this framework, we gain smoothness rather than hyperbolicity by applying blowup. 

\subsection{A blowup approach for regularization by smoothing}\seclab{sec:xysm}
A particular emphasis in the development of blowup for smooth systems approaching nonsmooth ones, has been on systems of the form \eqref{xysm}. Within our context these systems correspond to regularization of the PWS system \eqref{zPWS} by smoothing and the blowup approach proceeds as follows: 

Firstly, we work in the extended space $(x,y,\epsilon)$ by adding the trivial equation $\dot \epsilon=0$. At the same time, to ensure that $\epsilon=0$ is well-defined, we consider this extended system in terms of a fast time:
\begin{equation}
\begin{aligned}
 x' &=\epsilon X\left(z,\phi\left(\frac{y}{\epsilon}\right)\right),\\
 y' &=\epsilon Y\left(z,\phi\left(\frac{y}{\epsilon}\right)\right),\\
 \epsilon'&=0.
\end{aligned} \eqlab{xyeps}
\end{equation}
Then $(x,y,0)$ is a set of equilibria, but $(x,0,0)$ is extra singular due to the lack of smoothness there. This set is therefore blown up through a cylindrical blowup transformation defined by  
\begin{align}
 (r,(\bar y,\bar \epsilon))\mapsto \begin{cases}
                                    y &=r\bar y,\\
                                    \epsilon &=r\bar \epsilon,
                                   \end{cases}\eqlab{blowup0}
\end{align}
for $r\ge 0$, $(\bar y,\bar \epsilon)\in S^1$, leaving $x$ fixed. Here $S^1$ is the unit circle in $\mathbb R^2$. Notice that $r=0,(\bar y,\bar \epsilon)\in S^1$ maps to $(y,\epsilon)=(0,0)$ and \response{the preimage of the set of points $(x,0,0)$ is a cylinder; it is in this sense that the set of point $(x,0,0)$ is blown up by \eqref{blowup0}}.  See \figref{pws_smoothing0}.

Under the assumption \ref{assumption:2}, \response{$\phi(y\epsilon^{-1})=\phi(\bar y\bar\epsilon^{-1})$ extends smoothly to $(\bar y,\bar \epsilon)\in S^1\cap \{\bar \epsilon\ge 0\}$. This leads to the following, see \cite{kristiansen2018a}.}
\response{\begin{lemma}\lemmalab{blowup0}
Let $V$ denote the vector-field associated with \eqref{xyeps} and let $\Phi:(x,r,(\bar y,\bar \epsilon))\mapsto (x,y,\epsilon)$ be the blowup transformation defined by \eqref{blowup0}. Moreover, let $\Phi^*(V)$ be the pull-back of $V$. Then
$$\widehat V:=\bar \epsilon^{-1}\Phi^*(V)\quad \text{defined on}\quad (x,r,(\bar y,\bar \epsilon))\in \mathbb R^n\times [0,\infty)\times S^1\cap \{\bar \epsilon>0\},$$ extends smoothly and nontrivially to $\bar \epsilon=0$.
\end{lemma}}

We suppose that the following holds.
\begin{assumption}
	\label{assumption:4}
	The PWS system \eqref{zPWS} has stable sliding along the discontinuity set $\Sigma$:
	\begin{align*}
	 Y_+(x,0)<0,\,Y_-(x,0)>0.
	\end{align*}
	We also assume that $\Sigma$ is a compact domain in $\mathbb R^n$. 
\end{assumption}
In this way, $Z_\pm$ are each transverse to $\Sigma$. This leads to $\widehat V$ having hyperbolic properties along $r=0,\bar \epsilon=0$, see \cite{kristiansen2018a}. This also holds true even if $Z(z,p)$ depends nonlinearly on $p$. However, for the purpose of this section, we suppose the following.
\begin{assumption}
\label{assumption:3}
	$Z(z,p)$ is affine with respect to $p$ as in \eqref{Zpaf}.
	\end{assumption}

Then we have the following.
\begin{proposition}\proplab{smZsl}\cite{kristiansen2018a,Llibre09,Sotomayor96}
Consider \eqref{xysm} and suppose that assumptions \ref{assumption:4} and \ref{assumption:3} both hold. Then $\widehat V$ has a normally hyperbolic critical manifold, carrying a reduced slow flow \response{defined by $\dot x=X_{sl}(x)$, where $X_{sl}$ is the Filippov sliding vector-field, see \eqref{Xsl}}.
 \end{proposition}

 \begin{figure}
\begin{center}
\includegraphics[width=.65\textwidth]{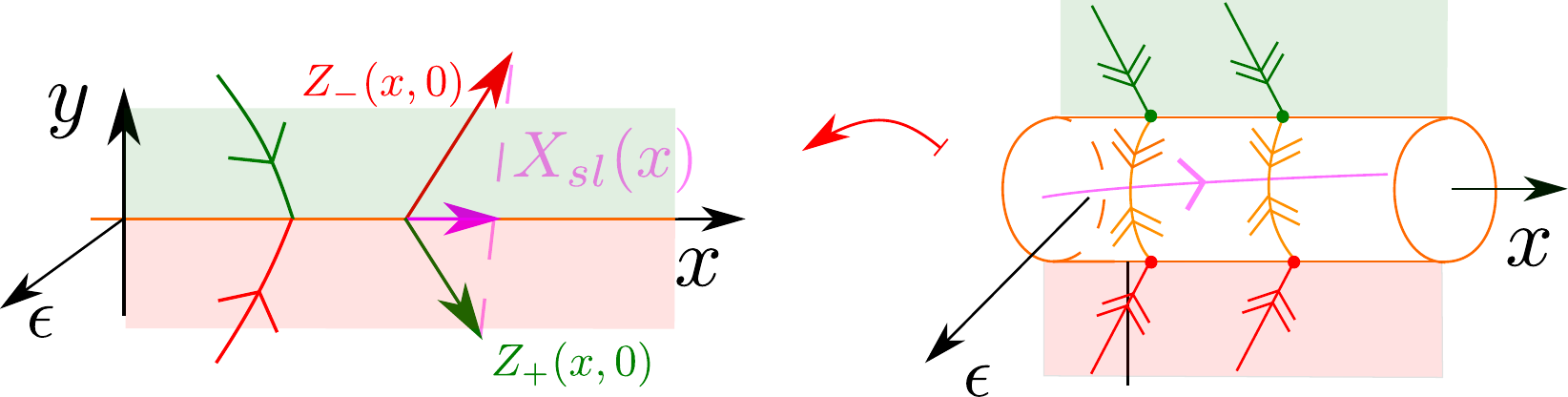}
 \end{center} \caption{Illustration of blowup in the case of regularization by smoothing. Upon blowup we gain smoothness and hyperbolicity along the edges of the cylinder (indicated by the double-headed arrows, see \secref{notation}) as well as a critical manifold (in pink) when the associated PWS system has stable sliding. The most fundamental result, see \propref{smZsl}, is then that the slow-flow on this critical manifold is given by Filippov's sliding vector-field. }
 \figlab{pws_smoothing0}
              \end{figure}
              The result is illustrated in \figref{pws_smoothing0}, see figure caption for further details, and has been known to experts for many years, see also \cite{Sotomayor96}.
              
              \subsection{Directional charts}\seclab{newDirectional}
 In practice, the analysis of $\widehat V$ is performed in directional charts. \response{Since we will use different directional charts in the sequel, 
 we now define these blowup-dependent charts (in some generality, following \cite[Definition 3.1]{szmolyan_canards_2001} and \cite{kristiansen2018a}) before we apply these concepts to \eqref{blowup0}. 

Consider 
 $x=(x_1,\ldots,x_n)\in \mathbb R^n$, $\kappa =(\kappa_1,\ldots,\kappa_n)\in \mathbb N^n$ and the following general, weighted (or quasihomogeneous \cite{kuehn2015}), blowup transformation:
 \begin{align}
  \Phi:\,[0,\infty)\times S^{n-1}\rightarrow \mathbb R^n:\, (\rho,\bar{ x})\mapsto  x,\quad (\rho,\bar{x})\mapsto (\rho^{\kappa_1}\bar x_1,\ldots,\rho^{\kappa_n}\bar x_n),\eqlab{generalBlowup}
 \end{align}
Here the pre-image of $x = 0$ is $\{0\}\times S^{n-1}$ where
\begin{align*}
 S^{n-1} = \left\{\bar{x}=(\bar x_1,\ldots,\bar x_n)\in \mathbb R^n \vert \sum_{i=1}^n \bar x_i^2=1 \right\},
\end{align*}
is the unit $(n-1)$-sphere. The positive integers $\kappa_i\in \mathbb N$ are called the {\it weights} of the blowup, see \cite{kuehn2015}. 

{\begin{definition}
  Let $j\in \{1,\ldots,n\}$ and write $\hat{x}^j=(\hat x_1,\ldots,\hat x_{j-1},\hat x_{j+1},\ldots,\hat x_n)\in \mathbb R^{n-1}$. 
Then the \textit{directional blowup} in the positive $j$-th direction is the {mapping}
 \begin{align*}
  \Psi^{j}:\,[0,\infty)\times \mathbb R^{n-1}\rightarrow \mathbb R^n, 
  \end{align*}
  obtained by setting $\bar x_j=1$ in \eqref{generalBlowup}:
\begin{align}
\Psi^{j}:\,(\hat \rho,\hat{x}^j)&\mapsto x = (\hat \rho^{\kappa_1} \hat x_1,\ldots, \hat \rho^{\kappa_{j-1}}\hat x_{j-1},\hat \rho^{\kappa_{j}}, \hat \rho^{\kappa_{j+1}} \hat x_{j+1},\ldots,\hat \rho^{\kappa_{n}}\hat x_n).\eqlab{directionPos}
  \end{align}
%
  {   The \textit{directional chart} $(\bar x_j=1)$ is then the coordinate chart 
   \begin{align*}
  \Xi^{j}:\, [0,\infty) \times S^{n-1}\rightarrow [0,\infty)\times \mathbb R^{n-1},
  \end{align*}
  such that 
  \begin{align*}
  \Phi=\Psi^{j} \circ \Xi^{j}.
  \end{align*} }
  
  \end{definition}}
  The directional blowup in the negative $j$-th direction and the associated directional chart $(\bar x_j=-1)$ are defined completely analogously (by setting $\bar x_j=-1$ in \eqref{generalBlowup}). 
  \begin{figure}[h!] 
\begin{center}
{\includegraphics[width=.69\textwidth]{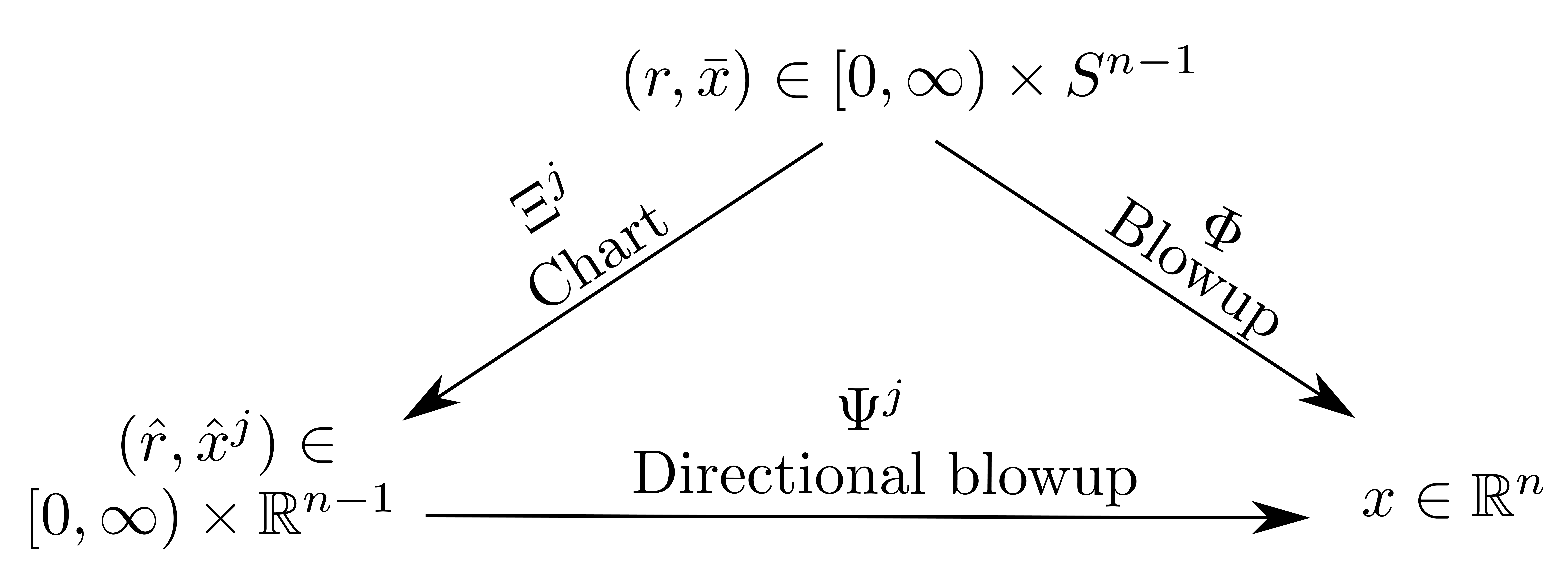}}
\end{center}
 \caption{{Given a blowup and an associated directional blowup (two edges in the diagram), we define the corresponding chart as the mapping (the final, third edge in the diagram) that makes the diagram commute  (on a subset of $[0,\infty)\times S^{n-1}$)}. }
\figlab{commute}
\end{figure}
%

We illustrate the concepts of a directional blowup and a directional chart in \figref{commute}. Notice that the directional blowup \eqref{directionPos} is a diffeomorphism for $\hat \rho>0$. But the preimage of $x=0$ is $\hat \rho=0,\hat{x} \in \mathbb R^{n-1}$. $\Xi^j$ exists and is unique, see \cite[Equation (5.5)]{kristiansen2018a}. The details are not important and therefore omitted.

%

With slight abuse of notation, we will, as is common in the literature, simply refer to \eqref{directionPos} as the (directional) charts $\bar x_j = 1$ (although they are actually the coordinate transformations in the local coordinates of the charts themselves,  see also \figref{commute}). 
Notice that the directional blowups are easy to compute: We just substitute $\bar x_j=1$ into \eqref{generalBlowup}, see \eqref{directionPos}.
 }

 \response{In the context of \eqref{blowup0}, we have three directional charts $(\bar y=\pm 1)$ and $(\bar \epsilon=1)$ so that \eqref{blowup0} takes the following local forms}:
 \begin{equation}\eqlab{y1eps1yN1}
\begin{aligned}
 (\bar y=1)_1:&\quad \begin{cases}
                     y&=r_1,\\
                     \epsilon &=r_1\epsilon_{1},
                    \end{cases}\\
(\bar \epsilon=1)_2:&\quad \begin{cases}
                     y&=r_2 y_2,\\
                     \epsilon &=r_2,
                    \end{cases}\\
 (\bar y=-1)_3:&\quad \begin{cases}
                     y&=-r_3,\\
                     \epsilon &=r_3\epsilon_3.
                    \end{cases}                   
\end{aligned}
\end{equation}
\response{(In the radial case of \eqref{blowup0}, they simply correspond to central projections onto the lines $\bar y=1$, $\bar \epsilon=1$ and $\bar y=-1$, respectively, see also \cite[Fig. 6]{kristiansen2018a}.)}
These charts cover the relevant part of the cylinder with $\bar \epsilon \ge 0$.  \response{As indicated, we refer to the three charts in \eqref{y1eps1yN1} by $(\bar y=1)_1$, $(\bar \epsilon=1)_2$, $(\bar y=-1)_3$ respectively, and the subscripts relate to the numbering used on the corresponding coordinates $(r_1,\epsilon_{1})$, $(r_2,y_2)$ and $(r_3,\epsilon_3)$, respectively.} The charts $(\bar y=1)_1$ and $(\bar \epsilon=1)_2$ overlap for $\bar y>0$ and the equations 
\begin{align}
r_1=r_2 y_2, \quad \epsilon_{1}= y_2^{-1}, \nonumber
\end{align}
define smooth change of coordinates there. \response{Similarly, $(\bar y=-1)_3$ and $(\bar \epsilon=1)_2$ overlap for $\bar y<0$ and the equations $r_3=r_2 y_2$, $\epsilon_3 = -y_2^{-1}$ define smooth change of coordinates there. (Obviously, $(\bar y=1)_1$ and $(\bar y=-1)_3$ do not overlap.)} Notice also that in $(\bar \epsilon=1)_2$ we have $y=\epsilon y_2$ upon eliminating $r_2$ and the blowup transformation therefore relates to this important scaling where 
\begin{align*}
 \phi\left(\frac{y}{\epsilon}\right)=\phi(y_2),
\end{align*}
changes by an $\mathcal O(1)$-amount. Moreover, in terms of $(x,y_2,r_2)$, $\widehat V$ becomes slow-fast:
\begin{equation}\eqlab{xy2}
\begin{aligned}
 \dot x &=\epsilon X(x,\epsilon y_2,\phi(y_2)),\\
 \dot y_2&=Y(x,\epsilon y_2,\phi(y_2)),
\end{aligned}
\end{equation}
with $r_2=\epsilon=\text{const}.$

Using assumptions \ref{assumption:4} and \ref{assumption:3}, it follows that \eqref{xy2} has a normally hyperbolic critical manifold for $\epsilon=0$, carrying reduced slow-flow given by \eqref{Xsl}. This essentially proves \propref{smZsl}. We illustrate the local dynamics in \figref{pws_smoothing2}. 

 \begin{figure}
\begin{center}
\includegraphics[width=.5\textwidth]{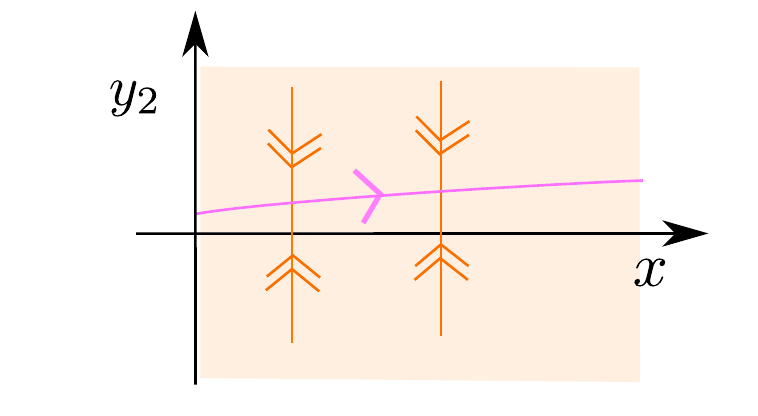}
 \end{center} \caption{Slow-fast dynamics in the $(\bar \epsilon=1)_2$-chart (for $\epsilon=0$) in the case of regularization by smoothing. The reduced problem is given by Filippov.}
 \figlab{pws_smoothing2}
              \end{figure}

              \subsection{A different version of blowup}\seclab{diffblowup}
              \response{We emphasize that, while the blowup \eqref{blowup0} -- following \lemmaref{blowup0} -- leads to gain of smooothness, blowup is traditionally associated with gain of hyperbolicity. In this version of blowup, the starting point is a vector-field $V$ having a fully nonhyperbolic equilibrium point (or a set of degenerate equilibria) with the linearization having only zero eigenvalues. Assuming that the equilibrium is at the origin, a blowup transformation $\Phi$ is then of the form \eqref{generalBlowup} with the weights $\kappa$ chosen such that 
              \begin{align*}
               \widehat V:=\rho^{-k} \Phi^*(V),
              \end{align*}
              on $(\rho,\bar x)\in (0,\rho_0]\times S^{n-1}$, 
extends smoothly and nontrivially to $\rho=0$ for some $k\in \mathbb N$. The most useful situation is when the division by $\rho^{-k}$ (desingularization) leads to hyperbolicity of equilibria within $\rho=0$, so that the usual hyperbolic methods (linearization, stable, unstable and center manifolds, etc) of dynamical systems theory, see e.g. \cite{wiggins2003a}, can be applied. See also \cite[Chapter 3.3]{dumortier2006a} for general results on blowup (including the use of Newton polygons to select the weights) for planar systems.

Blowup has been extremely succesful in the analysis of slow-fast systems, \cite{dumortier1996a,Gucwa2009783,krupa_extending_2001,szmolyan_canards_2001}, where loss of hyperbolicity occurs persistently in the layer problem. Here the weights $\kappa$ of the blowup transformation can often be directly related to the geometry of the problem. For example, for the planar fold jump point, see e.g. \cite[Equation 2.5]{krupa_extending_2001} where 
\begin{equation}\eqlab{foldplanar}
\begin{aligned}
 x' &\approx-y+x^2,\\
 y' &=0,
\end{aligned}
\end{equation}
for $\epsilon=0$,
we have a quadratic tangency between the critical manifold $y\approx x^2$ and the (degenerate) fiber $y=0$. In order to gain hyperbolicity, the weights $\kappa$ have to be so that this tangency is ``broken''. This can be achieved by $x=\rho \bar x,\,y=\rho^2\bar y$, $\rho\ge 0$, $(\bar x,\bar y)\in S^1$. Indeed, $y\approx x^2$ leads $\bar y\approx \bar x^2$, $(\bar x,\bar y)\in S^1$ ($\theta\approx \pm 0.67$ if $\bar x^{-1}\bar y=\tan \theta$), while $y=0$ leads to $\bar y=0$. For further details, we refer to \cite{krupa_extending_2001}.
}

\response{In this paper, we will combine these two different versions of blowup (gaining smooothness and gaining hyperbolicity) to study \eqref{xypmodel}. Similar combinations of blowup has been used to study bifurcations in systems of the form \eqref{xysm}, see e.g.  \cite{2019arXiv190806781U} for an analysis of the grazing bifurcation and \cite{jelbart2021c,jelbart2021b} for an analysis of  boundary equilibrium bifurcations (where equilibria of either $Z_\pm$ collide with $\Sigma$ upon parameter variation). 
              }


\subsection{A blowup approach for \eqref{xypmodel}}\seclab{sec:xypblowup}
To study \eqref{xypmodel} with $\alpha>0$, we now proceed as in \secref{sec:xysm}.  First, however, due to the time scale separation of \eqref{xypmodel}, we introduce a fast time and augment trivial equations for $\epsilon$ and $\alpha\ge 0$:
\begin{equation}\eqlab{xypmodelfast}
\begin{aligned}
 x' &=\epsilon \alpha X(z,p),\\
 y' &=\epsilon \alpha Y(z,p),\\
 p' & =\phi\left(\frac{y+  \alpha p}{\epsilon\alpha }\right)-p,\\
 \epsilon'&=0,\\
 \alpha'&=0.
\end{aligned}
\end{equation}
 Now, since \eqref{xypmodel} is PWS with respect to both $\epsilon\rightarrow \response{0}$ and $\alpha\rightarrow 0$, we anticipate that we will need to perform two blowup transformation. In light of this, \secref{sec:xysm} suggests that we should consider \eqref{xypmodelfast} with respect to an even faster time-scale, corresponding to multiplying the right hand side by $\epsilon \alpha$ again. But notice, despite the similarities, there is also a fundamental difference between \eqref{xypmodelfast} and \eqref{xyeps} insofar that the discontinuity set of \eqref{xypmodelfast} for $\epsilon,\alpha\rightarrow 0$ is $y=0$, $x\in \Sigma$, $p\in \mathbb R$, but the discontinuity only enters the $p$-equation. To avoid too many multiplications and subsequent divisions by the same quantities, we will therefore proceed more ad hoc in the following; in fact, the analysis will show that it is only necessary to multiply the right hand side of \eqref{xypmodelfast} by $\epsilon$ in order gain smoothness.

 Apriori it is not obvious how the two blowup transformations should be organized and whether the order is important, but leaving $\epsilon$ and $\alpha$ as independent small parameters, we will show that it is \response{convenient} to first blowup with respect to $\alpha$. (See the end of the section for a further discussion of this.)
 We therefore first apply the following blowup transformation
 \begin{align}
  (r,(\bar y,\bar \alpha))\mapsto \begin{cases}
                                     y &=-r\bar \alpha p+r\bar y,\\
                                     \alpha &= r\bar \alpha,
                                    \end{cases}\eqlab{blowupcyl1}
 \end{align}
 where $r\ge 0$, $(\bar y,\bar \alpha)\in S^1$,
 leaving all other variables $x$, $p$ and and $\epsilon$ untouched. In this way, we gain smoothness with respect to $\alpha\ge 0$ for any $\epsilon>0$. Indeed, the transformation \eqref{blowupcyl1} gives a smooth vector-field $\overline V$ for $\epsilon>0$ on $(x,p,\epsilon,(r,(\bar y,\bar \alpha))$, with $r\ge 0$, $(\bar y,\bar \alpha)\in S^1$, $\bar \alpha\ge 0$, by pull-back of \eqref{xypmodelfast} (without the need for further transformation of time, as the division by $\bar \epsilon$ in \lemmaref{blowup0}). Notice specifically, that using \eqref{blowupcyl1} the first term in the $p$-equation in \eqref{xypmodelfast} becomes:
 \begin{align}
  \phi\left(\frac{y+  \alpha p}{\epsilon \alpha }\right) = \phi\left(\frac{\bar y}{\epsilon \bar \alpha}\right),\eqlab{termp}
 \end{align}
 which for each $\epsilon>0$ is smooth on $(\bar y,\bar \alpha) \in S^1\cap \{\bar \alpha\ge 0\}$.  However, there is still a lack of smoothness along $(\bar \alpha,\bar y)=(1,0)$ as $\epsilon\rightarrow 0$. To deal with this, we perform a second blowup transformation:
 \begin{align}
  (\nu,(\bar{\bar y},\bar \epsilon))\mapsto \begin{cases}
                                     \bar \alpha^{-1} \bar y &=\nu \bar{\bar y},\\
                                     \epsilon &= \nu\bar \epsilon,
                                    \end{cases}\eqlab{blowupcyl2}
 \end{align}
 where $\nu\ge 0$, $(\bar{\bar y},\bar \epsilon)\in S^1$. Indeed, in this way, \eqref{termp} becomes regular
  \begin{align}
  \phi\left(\frac{y+  \alpha p}{\epsilon \vert\alpha\vert }\right) = \phi\left(\frac{\bar{\bar y}}{\bar \epsilon}\right),\nonumber
 \end{align}
 under assumption \ref{assumption:2}.
  We illustrate the blowup transformations in \figref{pws_blowup12}.

 \begin{figure}
\begin{center}
\includegraphics[width=.85\textwidth]{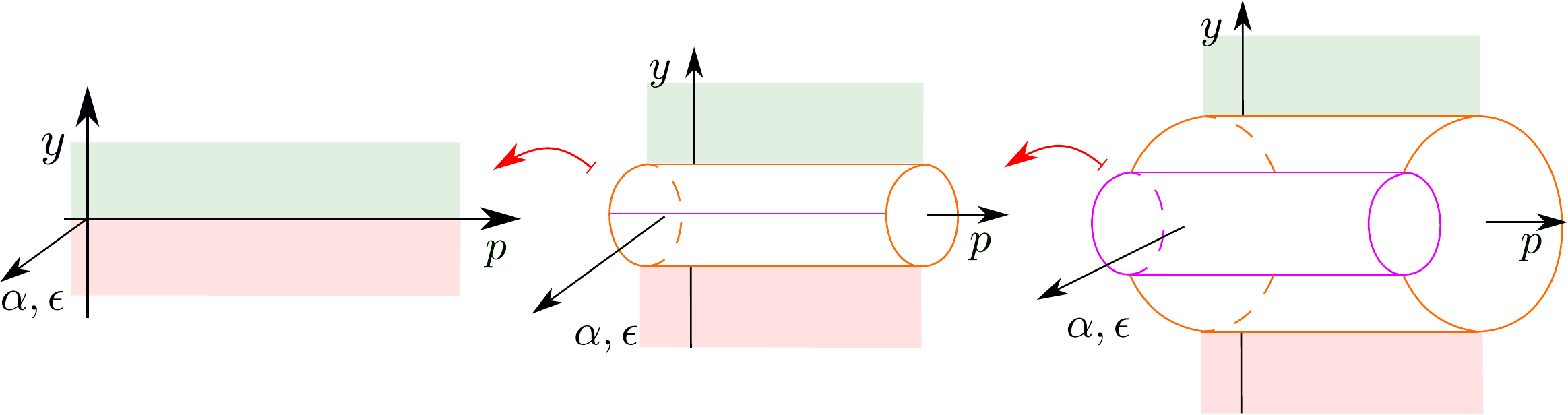}
 \end{center} 
\caption{Illustration of the two consecutive blowup transformations relating to \eqref{xypmodel}. The first cylinder corresponds to \eqref{blowupcyl1}. The second one corresponds to \eqref{blowupcyl2}. }
 \figlab{pws_blowup12}
              \end{figure}

              As described in \secref{sec:xysm} in the context of regularization by smoothing, we will also use different directional charts in the analysis of \eqref{xypmodel} to cover the two cylinders. In particular, to cover the first cylinder, defined by \eqref{blowupcyl1} and $(\bar y,\bar \alpha)\in S^1$, we (re-)consider the two charts defined by: 
\begin{align}
(\bar y=1)_1:&\quad \begin{cases}
                  y &=-r_1\alpha_1 p+r_1,\\
                  \alpha &= r_1 \alpha_1,
                 \end{cases}\eqlab{y1}\\
                 (\bar \alpha=1)_2:&\quad \begin{cases}
                  y &=-r_2 p+r_2 y_2,\\
                  \alpha &= r_2.
                 \end{cases}\eqlab{a1}
                 \end{align}
                                  We will refer to these charts by $(\bar y=1)_1$ and $(\bar \alpha=1)_2$, respectively, henceforth. 
                                  In principle, we will also need the chart $(\bar y=-1)_3$ that covers $\bar y<0$ of cylinder, but since the analysis there is identical to the analysis in the $(\bar y=1)_1$-chart we skip this. The change of coordinates between the charts $(\bar y=1)_1$ and $(\bar \alpha=1)_2$ are given by the expressions:
                                  \begin{align}
                                   r_1 = r_2 y_2,\quad \alpha_1=y_2^{-1}.\eqlab{cc12}
                                  \end{align}

                 
                 Subsequently, to cover the second cylinder due to \eqref{blowupcyl2}, we notice that in $(\bar \alpha=1)_2$, \eqref{blowupcyl2} becomes
                 \begin{align*}
                 y_2 &= \nu \bar{\bar y},\\
                 \epsilon &=\nu \bar \epsilon,
                 \end{align*}
for $\nu\ge 0$, $(\bar{\bar y},\bar\epsilon)\in S^1$. Therefore we define the following charts
                 \begin{align}
                 (\bar \alpha=1,\bar{\bar y}=1)_{21}:&\quad \begin{cases}
                  y_2 &=\nu_{21},\\
                                 \epsilon &= \nu_{21} \epsilon_{21},                                
                 \end{cases}\eqlab{y21}\\
                 (\bar \alpha=1,\bar \epsilon=1)_{22}:&\quad \begin{cases}
                  y_2 &=\nu_{22} y_{22},\\
                  \epsilon &= \nu_{22}.
                 \end{cases} \eqlab{eps22}
\end{align}
In both charts, we have $\alpha=r_2$. (The chart corresponding to $\bar{\bar y}=-1$ is again similar to $\bar{\bar y}=1$ and therefore left out.) 
The change of coordinates are given by the expressions
\begin{align}
 \nu_{21} = \nu_{22}y_{22},\quad \epsilon_{21} = y_{22}^{-1}, \eqlab{cc2122}
\end{align}
valid for $y_{22}>0$.

 The two blowup transformations relate to two important scalings. Firstly, in the $(\bar \alpha=1)_2$-chart, we have 
 \begin{equation}\eqlab{y2r2}
 \begin{aligned}
  y &= -\alpha p + \alpha y_2,
 \end{aligned}
 \end{equation}
 upon eliminating $r_2$ and consequently
\begin{align}
  \phi\left(\frac{y+  \alpha p}{\epsilon \vert\alpha\vert }\right) = \phi\left(\frac{y_2}{\epsilon}\right).\eqlab{term2}
 \end{align}
Through the coordinate $y_2$, we therefore zoom in on a $\mathcal O(\alpha)$-neighborhood of $y=0$. From \eqref{termp}, we understand that the resulting vector-field $V_2$ in terms of $(x,y_2,p,r_2,\epsilon)$ is itself PWS in the limit $\epsilon\rightarrow 0$. 
Consequently, following \secref{sec:xysm} and the results for gaining smoothness of \eqref{xyeps}, we see that through \eqref{blowupcyl2}, we obtain a smooth vector-field $\overline V_2$ on $(x,p,\alpha,\nu,(\bar{\bar y},\bar \epsilon))$, $\nu \ge 0$, $\bar \epsilon\ge 0$, by pullback of $\epsilon V_2$. This system has $\bar \epsilon$ as a common factor and it is therefore $\widehat V_2:=\bar \epsilon^{-1} \overline V_2$ that we will study \response{(please compare with \lemmaref{blowup0})}.

Next, we emphasize that in the $(\bar \alpha=1,\bar \epsilon=1)_{22}$-chart, we have $y_2=\epsilon y_{22}$ upon eliminating $\nu_{22}$ and consequently
\begin{align}
y&=-\alpha p+\alpha \epsilon y_{22}.\eqlab{y22scaling}
\end{align}
Therefore we also have that
\begin{align}
  \phi\left(\frac{y+  \alpha p}{\epsilon \vert\alpha\vert }\right) = \phi\left(y_{22}\right),\eqlab{termp2}
 \end{align}
 and we see that coordinate $y_{22}$ provides a zoom on a $\mathcal O(\alpha\epsilon)$-neighborhood of $y=-\alpha p$. 
 
 It is obvious that the scaling defined by \eqref{y22scaling} is important; this captures the region where the first term in the $p$-equation \eqref{xypmodelfast} changes by an $\mathcal O(1)$ amount with respect to $\epsilon,\alpha\rightarrow 0$. It also seems reasonable that the scaling \eqref{y2r2} is useful, but it not obvious why the scaling defined by 
 \begin{align}\eqlab{y1sc}
 y=-\alpha p+\epsilon y_{1},
 \end{align}seems to play no role. To see this we have to insert this expression into \eqref{xypmodelfast}. This gives
 \begin{align*}
  \epsilon \dot y_1 &= \alpha \left(\phi\left(\frac{y_1}{\alpha}\right)-p\right)+\epsilon \alpha Y(x,y,p).
 \end{align*}
Here we would like to divide by $\epsilon$ on the left hand side, but for this we will have to make assumptions on $\epsilon$ relative to $\alpha$ (i.e. whether $\epsilon^{-1}\alpha$ is small, moderate or large). If we insert \eqref{y2r2} instead, then we obtain 
\begin{align*}
  \alpha \dot y_2 &= \alpha \left(\phi\left(\frac{y_2}{\epsilon}\right)-p\right)+\epsilon \alpha Y(x,y,p).
 \end{align*}
Here $\alpha$ is a common factor on both sides which can therefore be divided out. This explains why \eqref{y2r2} and \eqref{y22scaling} are both important in our analysis and why \eqref{y1sc} will not be used.

\response{Finally, we emphasize that, while it might seem tempting to include $y,\,\epsilon$ and $\alpha$ in a single spherical blowup transformation, this only works well upon imposing specific order dependency on $\epsilon$ and $\alpha$. In contrast, our approach based on two separate blowup transformations allows us to consider the small parameters $0<\epsilon,\alpha\ll 1$ independently and thus cover a full neighborhood of $(\epsilon,\alpha)=(0,0)$. }
\subsection{Notation}\seclab{notation}
Throughout the paper we follow the convention that a set $S$ in the blowup space is given a subscript when viewed in a chart. I.e. the subset of a set $S$, which is visible in the chart $(\bar y=1)_1$, will be called $S_1$. Similarly, $S_2$ in the chart $(\bar \alpha=1)_2$. In the charts, $(\bar \alpha=1)_2$ and $(\bar \alpha=1,\bar \epsilon=1)_{22}$,  $r_2=\alpha$ and $\nu_{22}=\epsilon$ are constants, so when working in these charts, it is most \response{convenient} to eliminate $r_2$ and $\nu_{22}$, respectively, and return to treat $\epsilon$ and $\alpha$ as parameters. The only important thing to keep in mind in regards to this, is that when we change coordinates (e.g. through \eqref{cc12} and \eqref{cc2122}) then this has to be viewed in the appropriate space. For example, in the $(\bar \alpha=1)_2$-chart, we will obtain a slow manifold $S_{\epsilon,\alpha,2}$ in the $(x,y_{22},p)$-space. When writing this in the $(\bar y=1)_1$-chart, we first have to embed $S_{\epsilon,\alpha,2}$ in the extended $(x,y_{22},p,\epsilon,\alpha)$-space in the obvious way. We can then apply the change of coordinates \eqref{cc12} with $r_2=\alpha$ and obtain $S_{\epsilon,\alpha,1}$. We will henceforth perform similar change of coordinates without further explanation, moving back and forth between different spaces, treating $\epsilon$ and $\alpha$ as parameters whenever it is \response{convenient} to do so. 

\response{
Moreover, when illustrating phase space diagrams,
we follow the convention of using different arrows on
orbits to separate slow and fast directions. In particular, fast orbits are indicated by double-headed arrows, while slow orbits are 
indicated by single-headed ones. More generally, we adapt a similar notation to separate hyperbolic directions (double-headed arrows) from center/nonhyperbolic
directions (single-headed arrows).}

\section{Main results in the case of stable sliding }\seclab{main1}
In this section, we will use the blowup approach, outlined in the previous section, to describe the dynamics of \eqref{xypmodel} under the assumption \ref{assumption:4} of stable sliding. 
More specifically, we will provide a detailed study of the dynamics in each of the charts $(\bar y=1)_1$, $(\bar \alpha=1)_2$, $(\bar \alpha=1,\bar{\bar y}=1)_{21}$, $(\bar \alpha=1,\bar \epsilon=1)_{22}$. 
In summary, this analysis reveals the existence of two critical manifolds $C$ and $M$; these are essentially related to the blue and red dotted curves in \figref{Fgraph}. Whereas $C$ extends onto the first blowup cylinder, obtained by \eqref{blowupcyl1}, $M$ lies on the subsequent blowup cylinder, obtained by \eqref{blowupcyl2}. Moreover, $C$ is normally attracting and enables an extension of $S_{\epsilon,\alpha}$ in \lemmaref{Zpreg} up to $y=c\alpha$, for $c>0$ and $\epsilon,\alpha>0$ small enough. On the other hand, $M$ is normally repelling. Using the geometric representation used in \figref{pws_blowup12}, we illustrate the findings in \figref{pws_hysteresis_final0}. On both $C$ and $M$, we obtain a desingularized slow flow; the direction of this flow is also indicated in the figure but we emphasize that $x$ (not shown) is a constant for this reduced flow. This leads to a singular cycle $\Gamma_x$ for each $x\in \Sigma$, which we indicate  in \figref{pws_hysteresis_final0} using curves of increased thickness. Due to the desingularization along $C$, $\Gamma_x$ is akin to a relaxation cycle in slow-fast systems. In the full blowup space, the curves $\Gamma_x$ make up a singular cylinder $\Gamma=\{\Gamma_x\}_{x\in \Sigma}$ of dimension $n+1$. The first main result basically says that this singular cylinder persists for $0<\epsilon,\alpha\ll 1$ and that this manifold carries a reduced flow, which  can be related to the Filippov sliding vector-field \eqref{Xsl}.

\begin{figure}
\begin{center}
\includegraphics[width=.7\textwidth]{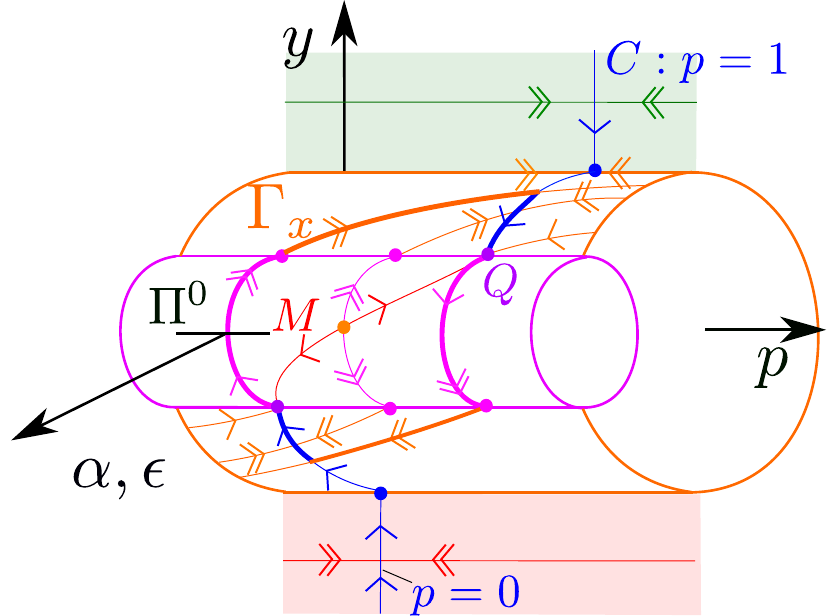}
 \end{center} \caption{Illustration of the dynamics on the blowup system. Our analysis reveals two normally hyperbolic critical manifolds $C$ and $M$. In case of stable sliding, the reduced slow flow on these invariant manifolds reveals closed a singular cycles $\Gamma_x$ (thick curves) for each $x\in \Sigma$. This cycle does not have completely desirable hyperbolicity properties due to the degeneracy at the point $Q$ \response{(indicated by the single-headed arrows, see \secref{notation})}. The slow flow on $M$ is given by $p'=-Y(x,0,p)$ and we illustrate the situation consistent with the assumption \ref{assumption:3}. In this case, we also have that $\dot x = X_{sl}(x)$ on the critical manifold defined by $Y(x,0,p)=0$, see \lemmaref{K22}. }
 \figlab{pws_hysteresis_final0}
              \end{figure}
              
\begin{thm}\thmlab{main1}
Suppose that assumptions \ref{assumption:1}, \ref{assumption:2} and \ref{assumption:4} all hold true and let $K>1$. Then there exists a $\delta>0$ such that for any  $0<\epsilon,\alpha<\delta$, \eqref{xypmodel} has an invariant cylinder $\mathcal C_{\epsilon,\alpha}$ of dimension ${n+1}$, contained within $y\in (-K\alpha,K\alpha)$. $C_{\epsilon,\alpha}$ is uniformly Lipschitz in the blowup space and converges to $\Gamma$ in the Hausdorff-distance as $\epsilon,\alpha\rightarrow 0$. 

Let $\Pi^0$ be a local section on $\{y=-\alpha p\}$ transverse to $\mathcal C_{\epsilon,\alpha}$ and define $x\mapsto x_+(x,\epsilon,\alpha)$ and $x\mapsto T(x,\epsilon,\alpha)$ to be the corresponding return map and the transition time, respectively. Then
 \begin{align}
  x_+(x,\epsilon,\alpha) &= x+\alpha \left[\vert Y_+(x,0)\vert^{-1} + \vert Y_-(x,0)\vert^{-1}\right]X_{sl}(x) + \mathcal O(\alpha^2,\epsilon^{\frac{k}{k+1}}\alpha),\eqlab{xpXsl}\\
  T_+(x,\epsilon,\alpha) &=\alpha \left[\vert Y_+(x,0)\vert^{-1} + \vert Y_-(x,0)\vert^{-1}\right] + \mathcal O(\alpha^2,\epsilon^{\frac{k}{k+1}}\alpha),
 \end{align}
 where the order of the remainder remains unchanged upon differentiation with respect to $x$.
Specifically, 
\begin{align*}
 \lim_{\epsilon,\alpha\rightarrow 0} \frac{x_+(x,\epsilon,\alpha)-x}{T_+(x,\epsilon,\alpha)} = X_{sl}(x).
\end{align*}

 

%

\end{thm}
\thmref{main1} generalizes \propref{smZsl} to the framework of \eqref{xypmodel} with $0<\epsilon,\alpha\ll 1$ without any order dependency on $\epsilon$ and $\alpha$.  

\response{We firmly believe that our approach can be modified to obtain a similar result for the Sotomayor-Teixeira regularization functions, see \eqref{Sotomayor}. Here the role of $k$ will be replaced by the order of smooothness of $\phi$ at $\pm 1$ (assuming finite smoothness), see \cite[p. 10]{reves_regularization_2014} (where $k$ is called $p$). In fact, as discussed in \cite[Section 3.1 and App. A]{kristiansen2017a}, the Sotomayor-Teixeira regularization functions are somewhat easier to handle in general as they do not require compactification. }

We prove \thmref{main1} in the following. In \secref{bary1}--\secref{sec:y21}, we first analyze the dynamics in each of the charts $(\bar y=1)_1$, $(\bar \alpha=1)_2$, $(\bar \alpha=1,\bar \epsilon=1)_{22}$, $(\bar \alpha=1,\bar{\bar y}=1)_{21}$, respectively. 
 In \secref{sec:finalbd}, we then collect the findings in the local charts into a global result, see \figref{pws_hysteresis_final0}. This includes a detailed description of $\Gamma_x$. Following this in \secref{sec:invcurve}, we first present a description of the return map defined on the section $\Pi^0_{22}:\,y_{22}=0$ transverse to $\Gamma$ in the $(\bar \alpha=1,\bar \epsilon=1)_{22}$-chart, see \lemmaref{lem:P22}. 
  The description of this mapping rests upon a subsequent blowup transformation of the degenerate point $Q$, which sits at the interface between $C$ and $M$, with the purpose of gaining hyperbolicity. The details of this blowup analysis of $Q$ and the proof of \lemmaref{lem:P22} are delayed to \secref{lem:P22}. \response{(The main idea of the proof of \thmref{main1} can be understood without this blowup)}. Before this in \secref{main1proof}, we show how \lemmaref{lem:P22} implies \thmref{main1}. Here we rely on a general result \cite[Theorem A.1]{szmolyan2004a} on the existence of an invariant curve for a return mapping. 
\subsection{Analysis in the $(\bar y=1)_1$-chart}\seclab{bary1}
In this chart, we insert \eqref{y1} into \eqref{xypmodelfast} and obtain
\begin{equation}\eqlab{y1eqns}
\begin{aligned}
 x' &=\epsilon r_1 \alpha_1 X_1(x,r_1,p,\alpha_1,\epsilon),\\
 r_1' &=r_1\alpha_1 \left(1-\phi_+\left(\epsilon\alpha_1\right)\epsilon^k\alpha_1^k-p+\epsilon Y_1(x,r_1,p,\alpha_1,\epsilon)\right),\\
 p' &=1-\phi_+\left(\epsilon\alpha_1\right)\epsilon^k\alpha_1^k-p,\\
 \alpha_1' &= -\alpha_1^2 \left(1-\phi_+\left(\epsilon\alpha_1\right)\epsilon^k\alpha_1^k-p+\epsilon Y_1(x,r_1,p,\alpha_1,\epsilon)\right),
\end{aligned}
\end{equation}
and $\epsilon'=0$, 
using assumption \ref{assumption:2}. This system is the local form of $\overline V$ in the $(\bar y=1)_1$-chart.  As already advertised above, we will henceforth treat $\epsilon$ as parameter in this chart. In \eqref{y1eqns},  we have defined 
\begin{align*}
 X_1(x,r_1,p,\alpha_1,\epsilon):= X(x,y,p),\quad Y_1(x,r_1,p,\alpha_1,\epsilon):= Y(x,y,p),
\end{align*}
with $y=-\alpha p +r_1$ and $\alpha=r_1\alpha_1$ on the right hand sides.
The system \eqref{y1eqns} is a slow-fast system in nonstandard form with respect to the small perturbation parameter $\epsilon$. Indeed for $\epsilon=0$, the set $C_1$ defined by $p=1$ is a critical manifold of the layer problem:
\begin{align*}
 x' &=0,\\
 r_1' &=r_1\alpha_1 \left(1-p\right),\\
 p' &=1-p,\\
 \alpha_1' &= -\alpha_1^2 \left(1-p\right),
\end{align*}
see illustration in \figref{pws_hysteresis_1}. The linearization around any point in $C_1$ produces $-1$ as the only nonzero eigenvalue. $C_1$ is therefore normally attracting.
\begin{lemma}\lemmalab{Seps1}
 Consider any compact submanifold $S_{0,1}$ of $C_1$, defined as the graph $p=1$ over a compact domain $D_1$ in the $(x,r_1,\alpha_1)$-space. Then for all $0<\epsilon\ll 1$, there exists a locally invariant slow manifold $S_{\epsilon,1}$, which is also a smooth graph over $D_1$:
 \begin{align*}
  p = P_1(x,r_1,\alpha_1,\epsilon), 
 \end{align*}
 where
 \begin{align*}
  P_1(x,r_1,\alpha_1,\epsilon)=1-\phi_+\left(\epsilon\alpha_1\right)\epsilon^k\alpha_1^k + \mathcal O(\epsilon^{k+1}  \alpha_1^{k+1}).
 \end{align*}

\end{lemma}
\begin{proof}
Direct calculation.
\end{proof}
For any $\alpha> 0$ small enough, we let $S_{\epsilon,\alpha,1}$ denote the constant $\alpha$-section, defined by $\alpha = r_1\epsilon_1$, of the center manifold $S_{\epsilon,1}$. The resulting invariant manifold $S_{\epsilon,\alpha,1}$ provides an extension of the slow manifold $S_{\epsilon,\alpha}$ in \lemmaref{Zpreg} into the $(\bar y=1)_1$-chart. 

\begin{figure}
\begin{center}
\includegraphics[width=.5\textwidth]{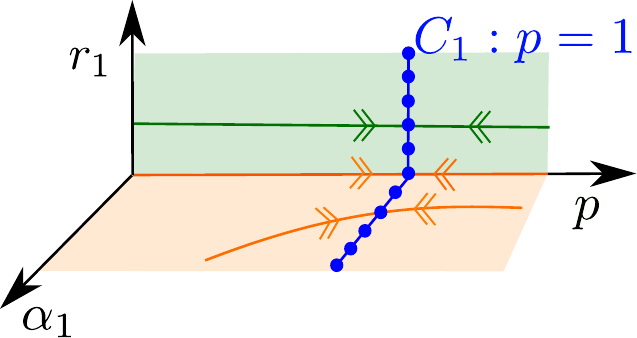}
  \end{center}\caption{Dynamics in the $(\bar y=1)_1$-chart. The manifold $C_1$ is normally hyperbolic. $C_1$ actually extends to any $\alpha=r_1\alpha_1$ but in this picture we illustrate the $\alpha=0$ limit.}
 \figlab{pws_hysteresis_1}
              \end{figure} 
On $S_{\epsilon,1}$, we have a reduced flow defined by 
\begin{equation}\eqlab{xr1alpha1}
\begin{aligned}
 \dot x &= r_1 X_1(x,r_1,P(x,r_1,\alpha_1,\epsilon),\alpha_1,\epsilon),\\
 \dot r_1 &= r_1 \left(Y_1(x,r_1,1-\beta\epsilon^k\alpha_1^k,\alpha_1,\epsilon)+ \mathcal O(\epsilon^{k+1}\alpha_1^{k+1} )\right),\\
 \dot \alpha_1 &= -\alpha_1 \left(Y_1(x,r_1,1-\beta\epsilon^k\alpha_1^k,\alpha_1,\epsilon)+ \mathcal O(\epsilon^{k+1}  \alpha_1^{k+1})\right),
\end{aligned}
\end{equation}
upon desingularization, corresponding division of the right hand side by $\epsilon\alpha_1$. 
\begin{lemma}
Consider \eqref{xr1alpha1}. Then $(x,0,0)$ defines a set of equilibria for all $\epsilon\ge 0$ and it is normally hyperbolic and of saddle type if $Y_+(x,0)\ne 0$.
\end{lemma}
The reduced problem is illustrated in \figref{pws_hysteresis_Reduced}. Notice it is identical to what is found by smoothing the PWS system, recall \eqref{xyeps} and \figref{pws_smoothing0}, near the edge of the blowup cylinder defined by \eqref{blowup0}.

              \begin{figure}
\begin{center}
\includegraphics[width=.56\textwidth]{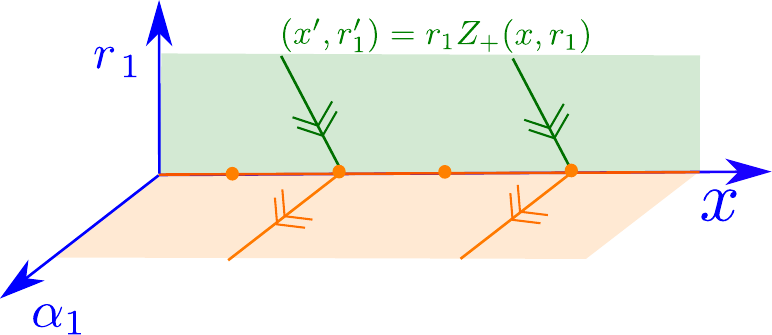}
 \end{center} \caption{Reduced dynamics on the normally hyperbolic critical manifold $C_1$ (blue in \figref{pws_hysteresis_1}) in the case when $Y_+(x,0)<0$. Within $\alpha_1=0$ the system is equivalent to $z'=Z_+$ upon time reparametrization for $r_1=y>0$. The line $r_1=\alpha_1=0$ is normally hyperbolic, each point having stable and unstable manifolds (green and orange, respectively) under the assumption $Y_+(x,0)<0$. In particular, the former invariant manifold lies within $r_1=0$, and along this set $x$ is a constant.}
 \figlab{pws_hysteresis_Reduced}
              \end{figure}
              
              \subsection{Analysis in the $(\bar \alpha=1)_2$-chart}
              In this chart, we insert \eqref{a1} into \eqref{xypmodelfast} and obtain the following equations
              \begin{equation}\eqlab{xypmodelfast2}
\begin{aligned}
 x' &=\epsilon \alpha X(x,-\alpha p+\alpha y_2,p),\\
 y_2' &=\phi\left(\frac{y_2}{\epsilon }\right)-p+ \epsilon Y(x,-\alpha p+\alpha y_2,p),\\
 p' & =\phi\left(\frac{y_2}{\epsilon}\right)-p,
\end{aligned}
\end{equation}
with $\epsilon'=\alpha'=0$. 
Within $\epsilon=0$, we re-discover the manifold of equilibria $C_1$ from the $(\bar y=1)_1$-chart, in the following form:
\begin{align*}
 C_2:\quad p = 1.
\end{align*}
Notice that the dependency on $\alpha$ is regular. In particular, note that $C_2$ is a critical manifold for any $\alpha\ge 0$. We will often view it within $\alpha=0$ (as in \figref{pws_hysteresis_1} since $\alpha=r_1\alpha_1$ in the $(\bar y=1)_1$-chart).

The manifold $C_2$ is also normally attracting for \eqref{xypmodelfast2} and carries the following reduced problem
\begin{align*}
 x' &=0,\\
 y_2'&=Y_+(x,0),
\end{align*}
upon passage to the slow time for $\epsilon=\alpha=0$.

In further details, let $S_{0,\alpha,2}\subset C_2$ be a compact submanifold contained within $y_2>0$ for any $\alpha\ge 0$. Then $S_{0,\alpha,2}$ perturbs to a slow manifold $S_{\epsilon,\alpha,2}$ by Fenichel's theory for $0<\epsilon\ll 1$ and an easy calculation shows that it takes the following graph form:
\begin{align*}
 S_{\epsilon,\alpha,2}:\quad p=P_{2}(y_2,\epsilon,\alpha),
\end{align*}
where
\begin{align*}
 P_{2}(y_2,\epsilon,\alpha)=1-\phi_+(y_2^{-1} \epsilon) y_2^{-k} \epsilon^k +\mathcal O(\epsilon^{k+1}).
\end{align*}
As a slow manifold, $S_{\epsilon,\alpha,2}$ is nonunique but may a copy such that it extends $S_{\epsilon,\alpha,1}$.
The reduced problem on $S_{\epsilon,\alpha,2}$ is given by
\begin{equation}\eqlab{redS2eps}
\begin{aligned}
 x' &= \alpha X(x,-\alpha P_2(y_2,\epsilon,\alpha)+\alpha y_2,P_2(y_2,\epsilon,\alpha)),\\
 y_2' &= Y(x,-\alpha+\alpha y_2,1)+\mathcal O(\epsilon^k)\\
 &=Y_+(x,0)+\mathcal O(\epsilon^k,\alpha).
\end{aligned}
\end{equation}


 \subsection{Analysis in the $(\bar \alpha =1,\bar \epsilon=1)_{22}$-chart}
 Consider \eqref{xypmodelfast} in terms of a faster time corresponding to multiplication of the right hand side by $\epsilon$. Then by inserting \eqref{eps22} into these equations, we obtain the following
\begin{equation}\eqlab{eps22eqns}
 \begin{aligned}
  \dot x &= \epsilon^2 \alpha X(x,-\alpha p+\epsilon \alpha y_{22},p),\\
   \dot y_{22}&=\phi(y_{22})-p+\epsilon Y(x,-\alpha p+\epsilon \alpha y_{22},p),\\
  \dot p &=\epsilon (\phi(y_{22})-p),
 \end{aligned}
 \end{equation}
  and $\epsilon'=\alpha'=0$. 
%
The system \eqref{eps22eqns} is now a slow-fast system with respect to $\epsilon\ge 0$ in standard form, $x$ and $p$ being slow while $y_{22}$ is fast. For $\epsilon=0$, we obtain the following layer problem:
\begin{equation}\eqlab{layer22}
\begin{aligned}
  \dot x &= 0,\\
   \dot y_{22}&=\phi(y_{22})-p,\\
  \dot p &=0,
 \end{aligned}
\end{equation}
 and consequently the set $M_{22}$ defined by $(x,y_{22},\phi(y_{22}))$ is a critical manifold, even for $\alpha>0$. As with $C$, we will often think of $M_{22}$ within $\alpha=0$.

 The manifold $M_{22}$ is normally repelling, since the linearization of \eqref{layer22} around any point $(x,y_{22},\phi(y_{22}))$ produces $\phi'(y_{22})>0$ as a single nonzero eigenvalue, see assumption \ref{assumption:1}. 
 \begin{lemma}
  Consider any compact submanifold $N_{0,\alpha,22}$ of $M_{22}$, defined as the graph $p=\phi(y_{22})$ over a compact domain $E_{22}$ in $(x,y_{22})$-space for any $\alpha\ge 0$. Then for all $0< \epsilon\ll 1$ there exists a locally invariant slow manifold $N_{\epsilon,\alpha,22}$ which is also a smooth graph over $E_{22}$:
  \begin{align*}
   p = P_{22}(x,y_{22},\epsilon,\alpha),
  \end{align*}
  where
  \begin{align*}
   P_{22}(x,y_{22},\epsilon,\alpha)=\phi(y_{22}) + \epsilon Y(x,0,\phi(y_{22})) + \mathcal O(\epsilon^2,\epsilon \alpha).
  \end{align*}
  The reduced problem on $N_{\epsilon,\alpha,22}$ is given by
  \begin{equation}
  \begin{aligned}
  x' &= \alpha X(x,-\alpha P_{22}(x,y_{22},\epsilon,\alpha)+\epsilon \alpha y_{22},P_{22}(x,y_{22},\epsilon,\alpha)),\\
  y_{22}'&=-\phi'(y_{22})^{-1} Y(x,0,\phi(y_{22}))+\mathcal O(\epsilon,\alpha),
  \end{aligned}\eqlab{reducedN22}
  \end{equation}
  in terms of a slow time (that corresponds to dividing the right hand side of \eqref{eps22eqns} by $\epsilon^2$). 
 \end{lemma}
\begin{proof}
 For the reduced problem, we first use that 
 \begin{align*}
  \dot p &= \epsilon \left(\phi(y_{22})-p\right) = -\epsilon^2 \left(Y(x,0,\phi(y_{22})) +\mathcal O(\epsilon,\alpha)\right),
 \end{align*}
on $N_{\epsilon,\alpha,22}$. Then upon realizing that $p=\phi(y_{22}) + \mathcal O(\epsilon)$, we obtain
 the desired result. 
\end{proof}
Notice that for $\epsilon=\alpha=0$, we can also write \eqref{reducedN22} as 
\begin{equation}\eqlab{reducedM22}
\begin{aligned}
x'&=0,\\
 p' &=- Y(x,0,p), 
\end{aligned}
\end{equation}
which is more \response{convenient}. 
\begin{lemma}\lemmalab{K22}
Suppose that assumptions \ref{assumption:4} and \ref{assumption:3} both hold. Then \eqref{reducedM22} has a critical manifold $K_{22}$ defined by 
\begin{align*}
 Y(x,0,p)=0,
\end{align*}
which is normally repelling. The reduced problem on $K_{22}$ is given by 
\begin{align}
 \dot x &=X_{sl}(x),\eqlab{Xsl2}
\end{align}
recall \eqref{Xsl}, with respect to the original (slow) time of \eqref{xypmodel} for $\epsilon=\alpha=0$.
\end{lemma}
\begin{proof}
 \response{Using \ref{assumption:3}, we obtain that $K_{22}$ is given by 
 \begin{align*}
  p= \frac{Y_-(x,0)}{Y_-(x,0)-Y_+(x,0)}.
 \end{align*}
 Inserting this into $\dot x = \lim_{\alpha,\epsilon\rightarrow 0} \alpha^{-1}x'$, where $x'$ is given as in \eqref{reducedN22} with $P_{22}(x,y_{22},0,0)=p$, produces \eqref{Xsl2}, see also \eqref{Xsl}. Finally, the stability of $K_{22}$ is determined by the linearization of \eqref{reducedM22}. We obtain $-Y_{+}+Y_->0$ (using assumption \ref{assumption:4}) as the single nontrivial eigenvalue. This completes the proof.
}
\end{proof}

In \figref{pws_hysteresis_22} we summarize the findings. 
\begin{remark}\remlab{lambda}
 Interestingly, the contraction and expansion rates along $S_{\epsilon,\alpha}$ and $N_{\epsilon,\alpha}$ are different with respect to $\epsilon,\alpha>0$ in the following sense: Suppose that $X_1\ne 0$. Then when $x_1$ changes by an order $\mathcal O(1)$-amount for the reduced flow on $S_{\epsilon,\alpha}$, then there is contraction along the stable fibers of the order $\mathcal O(e^{-c\epsilon^{-1}\alpha^{-1}})$, $c>0$. On the other hand, under the same assumptions on $N_{\epsilon,\alpha}$, see \eqref{eps22eqns}, if $x_1$ changes by an order $\mathcal O(\alpha)$-amount for the reduced problem on $N_{\epsilon,\alpha}$ in \textnormal{backward time} then there is a contraction along the (unstable) fibers of the order $\mathcal O(e^{-c\epsilon^{-2}})$, $c>0$.
%
%
\end{remark}
\begin{figure}
\begin{center}
\includegraphics[width=.5\textwidth]{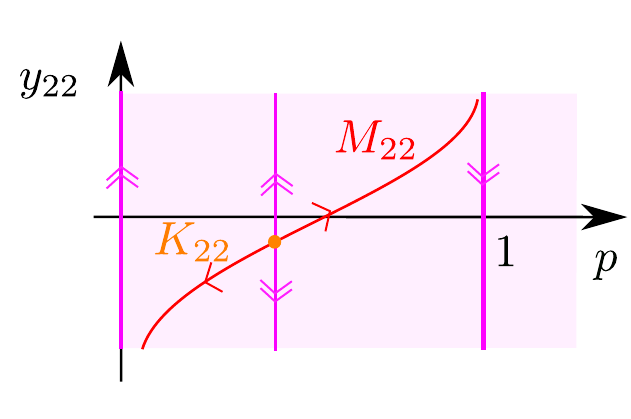}
 \end{center} \caption{Dynamics in the $(\bar \alpha=1,\bar \epsilon=1)_{22}$-chart. The critical manifold $M_{22}$ is normally hyperbolic and repelling and if assumptions \ref{assumption:3} and \ref{assumption:4} hold true then there exists an unstable critical set $K_{22}$ of the slow flow on $M_{22}$, carrying  Filippov's sliding flow as a reduced slow flow, see \lemmaref{K22}. }
 \figlab{pws_hysteresis_22}
 \end{figure}

\subsection{Analysis in the $(\bar \alpha=1,\bar{\bar y}=1)_{21}$-chart}\seclab{sec:y21}
 Consider again \eqref{xypmodelfast} in terms of a faster time corresponding to multiplication of the right hand side by $\epsilon$. Then by inserting \eqref{y21} into these equations, we obtain the following
\begin{equation}\eqlab{eqn21}
 \begin{aligned}
  \dot x &= \nu_{21}^2 \epsilon_{21} \alpha X_{21}(x,\nu_{21},p,\alpha),\\
  \dot \nu_{21} &=\nu_{21} \left[1-\phi_+(\epsilon_{21})\epsilon_{21}^k-p+\nu_{21} \epsilon_{21} Y_{21}(x,\nu_{21},p,\alpha)\right],\\
  \dot p&=\nu_{21} \left( 1-\phi_+(\epsilon_{21})\epsilon_{21}^k-p\right),\\
  \dot \epsilon_{21}&=-\epsilon_{21} \left[1-\phi_+(\epsilon_{21})\epsilon_{21}^k-p+\nu_{21} \epsilon_{21} Y_{21}(x,\nu_{21},p,\alpha)\right],
 \end{aligned}
 \end{equation}
 upon desingularization through division of the right hand side by $\epsilon_{21}$. Here we treat $\alpha$ as parameter and have introduced the following quantities
 \begin{align*}
  X_{21}(x,\nu_{21},p,\alpha):=X(x,-\alpha p+\alpha \nu_{21},p),\quad Y_{21}(x,\nu_{21},p,\alpha):=Y(x,-\alpha p+\alpha \nu_{21},p).
 \end{align*}
The set $B_{21}$ defined by $\nu_{21}=\epsilon_{21}=0$ is a set of equilibria for any $\alpha\ge 0$. The linearization about any point in this set has two nontrivial eigenvalues: $\pm (1-p)$. Consequently, the subset $Q_{21}\subset B_{21}$ defined by $p=1$ is fully nonhyperbolic, also for any $\alpha\ge 0$.


Let $\nu_{21}=0$ in \eqref{eqn21}. Then
\begin{align*}
 \dot x &=0,\\
 \dot p&=0,\\
 \dot \epsilon_{21} &=-\epsilon_{21} \left(1-\phi_+(\epsilon_{21})\epsilon_{21}^k -p\right).
\end{align*}
Besides $B_{21}$, we see that the set $M_{21}$, defined by 
\begin{align}\eqlab{graphM21} 
M_{21}:\quad p=1-\phi_+(\epsilon_{21})\epsilon_{21}^k,\quad \epsilon_{21}>0,
\end{align}
is a set of equilibria within $\nu_{21}=0$. $M_{21}$ corresponds to the subset of $M_{22}$ with $y_{22}>0$ by \eqref{cc2122}. The corresponding graph \eqref{graphM21} ends in $Q_{21}$ for $\epsilon_{21}=0$. 

There is obviously another critical set $C_{21}$, given by $\epsilon_{21}=0,p=1$, $\nu_{21}>0$, emanating from $Q_{21}$. It corresponds to $C_1$ from the $(\bar y=1)$-chart, see \secref{bary1}.

Both sets, $M_{21}$ and $C_{21}$ are normally hyperbolic, $M_{21}$ being repelling whereas $C_{21}$ is attracting. \response{The set $Q_{21}$ -- at the interface of these critical manifolds with different normal stability -- acts like a regular fold jump point of slow-fast systems, see \cite{krupa_extending_2001,szmolyan2004a}. In particular, there is only one mechanism (a fast jump, magenta in \figref{pws_hysteresis_21}) with which one can leave $Q_{21}$ (upon entering from either $C_{21}$ or $M_{21}$). (For further details, see \secref{Qblowup} below where $Q_{21}$ is blown up.) Notice that as in the case of the planar fold piont \eqref{foldplanar}, there is tangency between $p=1$ (the jump mechanism) and $M_{21}$ within $\nu_{21}=0$ for $k>1$, but the tangency is of  order $k$ in the present case.} 


Let $\epsilon_{21}=0$ in \eqref{eqn21}. Then
\begin{equation}\nonumber
 \begin{aligned}
  \dot x &= 0,\\
  \dot \nu_{21} &=\nu_{21} \left(1-p\right),\\
  \dot p&=\nu_{21} \left( 1-p\right).
 \end{aligned}
 \end{equation}
 It follows that each point on the critical set $(x,0,p,0)\in B_{21}$ with $p<1$ is connected by a heteroclinic orbit through the dynamics of \eqref{eqn21} to a point on $C_{21}$. In particular, we have the following result, which follows from a simple calculation.
 \begin{lemma}\lemmalab{het}
 Consider \eqref{eqn21}. Then for each $p<1$ and any $\alpha\ge 0$, there is a heteroclinic connection contained within $\epsilon_{21}=0$, having  $(x,0,p,0)\in B_{21}$ as the $\alpha$-limit set and $(x,1-p,1,0)\in C_{21}$ as the $\omega$-limit set. 
 \end{lemma}

 We illustrate our findings in the $(\bar \alpha=1,\bar{\bar y}=1)_{21}$-chart in \figref{pws_hysteresis_21}.
            \begin{figure}
\begin{center}
\includegraphics[width=.45\textwidth]{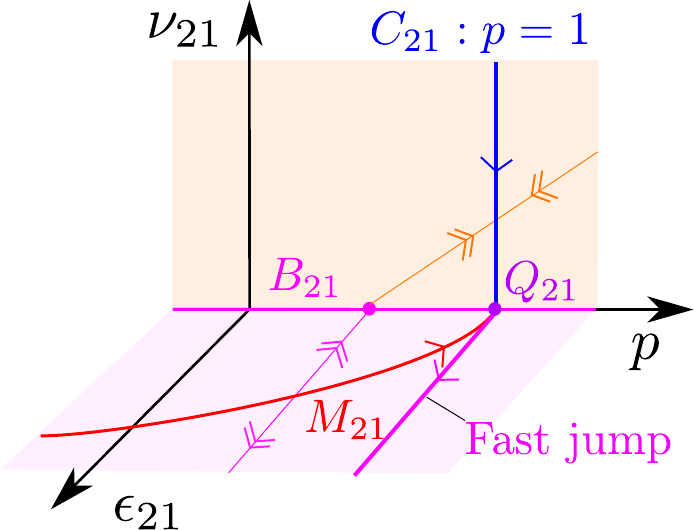}
 \end{center} \caption{Dynamics in the $(\bar \alpha=1,\bar{\bar y}=1)_{21}$-chart. The critical manifolds $M_{21}$, $B_{21}$ and $C_{21}$ are all normally hyperbolic away from the degenerate point $Q_{21}$ at $\nu_{21}=\epsilon_{21}=0,\,p=1$.}
 \figlab{pws_hysteresis_21}
 \end{figure}  

 
\subsection{Collecting the local results into a global picture}  \seclab{sec:finalbd}
 \figref{pws_hysteresis_final0} summarizes the findings in the local charts. Notice specifically, that while we have focused on the upper part of the cylinders, the analysis of the lower part is identical and therefore skipped. In conclusion, we obtain a singular cycle $\Gamma_x$ for each $x\in \Sigma_{sl}$, $x$ being a constant on the two cylinders. $\Gamma_x$ is the union of six pieces $\gamma_{xi}$, $i=1,\ldots, 6$ where:
 \begin{enumerate}
  \item $\gamma_{x1}$ is a heteroclinic connection on the first cylinder. It is described in the coordinates of the $(\bar \alpha=1,\bar{\bar y})_{21}$-chart in \lemmaref{het} (corresponding to $p=0$ in this result). In particular, $x$ is constant along $\gamma_{x1}$ and its $\alpha$-limit set is given by $(\nu_{21},p,\epsilon_{21})=(0,0,0)$ on $B_{21}$, whereas the $\omega$-limit set is given by $(\nu_{21},p,\epsilon_{21})=(1,1,0)$, belonging to the normally attracting set $C_{21}$.
  \item $\gamma_{x2}$ is an orbit segment of the desingularized system on the attracting manifold $C$. In the coordinates of the $(\bar \alpha=1)_{2}$-chart, $\gamma_{2x}$ takes the following form $p=1$, $y_2\in (0,1]$, $\epsilon=\alpha=0$. In the coordinates of the $(\bar \alpha=1,\bar{\bar y}=1)_{21}$-chart, it ends at $Q_{21}$. 
  \item $\gamma_{x3}$ is a heteroclinic connection on the second cylinder, connecting the degenerate point $Q$ with a partially hyperbolic point on the other side. In the coordinates of the $(\bar \alpha=1,\bar \epsilon)_{22}$-chart, $\gamma_{x3}$ is given by $p=1$, $y_{22}\in \mathbb R$, $\epsilon=\alpha=0$. 
 \end{enumerate}The remaining pieces $\gamma_{xi}$, $i=4,5,6$ are obtained in a similar way. 
 When $x$ ranges over the compact domain $\Sigma$, we obtain a compact cylinder $\Gamma:=\{\Gamma_x\}_{x\in \Sigma}$. 

 \subsection{A return map}\seclab{sec:invcurve}
Consider the $(\bar \alpha=1,\bar \epsilon=1)_{22}$-chart and define a local section $\Pi^0_{22}$ in  the $(x,y_{22},p)$-space at $y_{22}=0$ with $p\in I^0$ a small neighborhood of $p=0$, see \figref{pws_hysteresis_final0}, treating both $\epsilon\ge 0$ and $\alpha\ge 0$ as sufficiently small parameters. $\Gamma$ then intersects $\Pi^0_{22}$ in $p=0$ (for $\epsilon=\alpha=0$). For $\epsilon>0$, $\alpha>0$, sufficiently small, we will then have a well-defined return map $\mathcal P_{22}:\Pi^0_{22}\rightarrow \Pi^0_{22}$, $(x,p)\mapsto (x_+,p_+)$ with $(x_+,0,p_+)$ being the first return of $(x,0,p)$ to $\Pi^0_{22}$ upon following the forward flow. In particular, the following holds. 
\begin{lemma}\lemmalab{lem:P22}
 The mapping $\mathcal P_{22}$ is given by
\begin{equation}\eqlab{P22}
\begin{aligned}
 x_+(x,p,\epsilon,\alpha)&=  x+\alpha \left[ (1-p)\vert Y_+(x,0)\vert^{-1} X_+(x,0)+\vert Y_-(x,0)\vert^{-1}X_-(x,0)\right]\\
 &+\mathcal O(\alpha^2,\epsilon^{\frac{k}{k+1}} \alpha),\\
 p_+(x,p,\epsilon,\alpha)&=s_{22}(x,\epsilon,\alpha)+\mathcal O(e^{-c/\epsilon}),
\end{aligned}
\end{equation}
with $s_{22}(x,\epsilon,\alpha)=\mathcal O(\epsilon^{\frac{k}{k+1}})$ smooth. The remainder terms remain unchanged upon differentiation with respect to $x$ and $p$.
\end{lemma}
We prove \lemmaref{lem:P22} in \secref{lem:P22}.
\subsection{Completing the proof of \thmref{main1}}\seclab{main1proof}
We now show how \lemmaref{lem:P22} implies \thmref{main1}.
 For this, we first realize that the return map in \lemmaref{lem:P22} satifies the hypothesis of \cite[Theorem A.1]{szmolyan2004a} regarding the existence of an invariant curve.
\begin{proposition}\proplab{invcur}
 The mapping $\mathcal P_{22}$ has an invariant curve given by the graph
 \begin{align*}
  p = c_{22}(x,\epsilon,\alpha),
 \end{align*}
 with $c_{22}(x,\epsilon,\alpha)=\mathcal O(\epsilon^{\frac{k}{k+1}})$ smooth in $x$ and continuous in 
$\epsilon,\alpha\rightarrow 0$. 
\end{proposition}
\begin{proof}
 To apply \cite[Theorem A.1]{szmolyan2004a} we first write $\mathcal P_{22}$ in terms of $(x,\tilde p)$ 
where $\tilde p:=p-s_{22}(x,\epsilon,\alpha)$. We drop the tilde. Then following \lemmaref{lem:P22}, $\mathcal P_{22}$ for $\alpha,\epsilon\rightarrow 0$ is given by $(x,p)\mapsto (x,0)$. In comparison with \cite[Theorem A.1]{szmolyan2004a}, we therefore have $y=x$, $z=p$, $G_0(y)=y$ with $G_0'(y)\ne 0$ and $H_{2}(\epsilon) = \mathcal O(e^{-c/\epsilon})$. The conditions of \cite[Theorem A.1]{szmolyan2004a} are easily verified. 
\end{proof}
Upon applying the flow map to the invariant curve of $\mathcal P_{22}$ in \propref{invcur}, we obtain the desired invariant cylinder $\mathcal C_{\epsilon,\alpha}$ in \thmref{main1}. To finish the proof of \thmref{main1}, we just have to prove \eqref{xpXsl}. For this, we reduce the mapping $\mathcal P_{22}$ to the invariant manifold $\mathcal C_{\epsilon,\alpha}$. From the previous analysis, we obtain
\begin{align*}
x\mapsto x+\alpha \left[\vert Y_+(x,0)\vert^{-1}X_+(x,0) +\vert Y_-(x,0)\vert^{-1} X_-(x,0)\right] +\mathcal O(\alpha^2,\epsilon^{\frac{k}{k+1}} \alpha).
\end{align*}
Using \eqref{Xsl} we can write $\left[\cdots\right]$ as 
\begin{align*}
\left(\vert Y_+(x,0)\vert^{-1} +\vert Y_-(x,0)\vert^{-1} \right)X_{sl}(x).
\end{align*}
This completes the proof of the expression for $x_+$ in \eqref{xpXsl}. The expression for the transition time is similar; in fact, it can be obtained from the expression for $x$ by setting $X_+=X_-=1$ (since $\dot t=1$).

\section{Proof of \lemmaref{lem:P22}}\seclab{lem:P22}
To prove \lemmaref{lem:P22}, we will chop the return map $\mathcal P_{22}$ into several local pieces. However, to describe the local transition near the degenerate set $Q$, we have to perform an additional blowup step. In the following, we first analyze this blowup transformation and the associated dynamics in separate local charts. In this way, we obtain singular cycles $\Gamma_x$ with improved hyperbolicity properties. 
\subsection{Blowup of $Q$}\seclab{Qblowup}
 We work in the $(\bar \alpha=1,\bar{\bar y}=1)_{21}$-chart with the coordinates $(x,\nu_{21},p,\epsilon_{21})$, treating $\alpha$ as a parameter. Then $Q$ takes the local form $(x,0,1,0)$, $x\in\Sigma$, which is blown up by the following transformation 
\begin{align}
\rho\ge 0,\,(\bar \nu_{21},\bar p,\bar \epsilon_{21})\in S^2\mapsto \begin{cases}
                                                           \nu_{21}&=\rho^{k}\bar \nu_{21},\\
                                                           p &=1+\rho^k \bar p,\\
                                                           \epsilon_{21} &=\rho \bar \epsilon_{21},
                                                          \end{cases}      \eqlab{blowupsp}                                          
\end{align}
that leaves $x$ fixed. \response{Notice that the weights on $p$ and $\epsilon_{21}$ are so that the tangency between $p=1$ and $M_{21}$, see \eqref{graphM21}, is ``broken'' in the blown up space (recall the discussion around \eqref{foldplanar}).} This transformation induces a vector-field $\overline V_{21}$ by pull-back of \eqref{eqn21}, having $\rho^{k}$ as a common factor. It is therefore the desingularized vector-field $\widehat V= \rho^{-k}\overline V_{21}$ that we study in the following. 

Seeing that $\nu_{21}, \epsilon_{21}\ge 0$ we are only interested in the quarter sphere defined by $\bar \nu_{21},\bar \epsilon_{21}\ge 0$, see \figref{pws_hysteresis_3} and \figref{pws_hysteresis_final} below.
Consider the two directional charts, $\bar \nu=1$ and $\bar\epsilon=1$ with chart-specific coordinates defined by 
\begin{align*}
(\bar \alpha=1,\bar{\bar y}=1,\bar \nu_{21}=1)_{211}:&\quad \begin{cases}
                    \nu_{21} &=\rho_{211}^k \\
                    p &=1+\rho_{211}^k p_{211},\\
                    \epsilon_{21} &=\rho_{211} \epsilon_{211}.
                   \end{cases},\\
(\bar \alpha=1,\bar{\bar y}=1,\bar \epsilon_{21}=1)_{212}:&\quad \begin{cases}
                    \nu_{21} &=\rho_{212}^k\nu_{212} \\
                    p &=1+\rho_{212}^k p_{212},\\
                    \epsilon_{21} &=\rho_{212}.
                   \end{cases}       
\end{align*}
Although these charts cover the relevant part of the sphere (except for $\bar p=\pm 1$ but this part is trivial), we prefer to cover a compact subset of $\bar \nu_{21},\bar \epsilon_{21}>0$ using a separate chart. This chart, which we will refer to as $(\bar \alpha=1,\bar{\bar y}=1,\bar \nu_{21}\bar \epsilon_{21}=1)_{213}$, is defined by the coordinates $(\rho_{213},p_{213},\nu_{213})$ and the equations
\begin{align*}
(\bar \alpha=1,\bar{\bar y}=1,\bar \nu_{21}\bar \epsilon_{21}=1)_{213}:\quad \begin{cases} \nu_{21} &= \rho_{213}^k \nu_{213},\\
 p &= 1+\rho_{213}^{k} p_{213},\\
 \epsilon_{21} &=\rho_{213} \nu_{213}^{-1}.
 \end{cases}
\end{align*}
The advantage of working with this chart, is that in these coordinates 
\begin{align}\eqlab{rho3eps}
\epsilon=\nu_{21}\epsilon_{21} = \rho_{213}^{k+1},
\end{align} and $\rho_{213}$ is therefore conserved. In comparison, we have  
\begin{align}
 \epsilon = \nu_{21}\epsilon_{21} = \rho_{211}^{k+1} \epsilon_{211}=\rho_{212}^{k+1}\nu_{212},\eqlab{eps211}
\end{align}
in the other charts. Notice that we also have $\bar \nu_{21} \bar \epsilon_{21}^{-k} = \nu_{213}^{k+1}$, which is why we only use these coordinates to cover a compact subset of $\bar \nu_{21},\bar \epsilon_{21}>0$. 
The coordinate changes between the different charts are given by the following expressions:
\begin{align}
 \begin{cases}
  \rho_{211} &=\rho_{213} \nu_{213}^{\frac{1}{k}},\\
  p_{211} &=p_{213}\nu_{213}^{-1},\\
  \epsilon_{211}&=\nu_{213}^{-\frac{k+1}{k}},
 \end{cases}\quad \quad \begin{cases}
  \rho_{212} &=\rho_{213} \nu_{213}^{-1},\\
  p_{212} &=p_{213}\nu_{213}^{k},\\
  \nu_{212}&=\nu_{213}^{k+1}.
 \end{cases}
 \eqlab{cc211}
\end{align}

    
 \subsection{Entry chart $(\bar \alpha=1,\bar{\bar y}=1,\bar \nu_{21}=1)_{211}$}\seclab{entrychartQ}
              In this chart, we obtain the following equations:
\begin{equation}\eqlab{eqns211}
\begin{aligned}
\dot x &=\rho_{211}^{k+1} \epsilon_{211} \alpha X_{211}(x,\rho_{211},p_{211},\alpha),\\
 \dot \rho_{211} &=\frac{1}{k}\rho_{211} \left[-p_{211}-\phi_+(\rho_{211} \epsilon_{211}) \epsilon_{211}^k+\rho_{211} \epsilon_{211} Y_{211}(x,\rho_{211},p_{211},\alpha)\right],\\
 \dot p_{211} &= (1-p_{211})\left(-p_{211}-\phi_+(\rho_{211} \epsilon_{211}) \epsilon_{211}^k\right)-\rho_{211} \epsilon_{211} p_{211}  Y_{211}(x,\rho_{211},p_{211},\alpha),\\
 \dot \epsilon_{211} &=-\frac{k+1}{k}\epsilon_{211}\left[-p_{211}-\phi_+(\rho_{211} \epsilon_{211}) \epsilon_{211}^k+\rho_{211} \epsilon_{211} Y_{211}(x,\rho_{211},p_{211},\alpha)\right],
 \end{aligned}
 \end{equation}
where
\begin{align*}
 X_{211}(x,\rho_{211},p_{211},\alpha) &:= X_{21}(x,\rho_{211}^k,1+\rho_{211}^k p_{211},\alpha),\\
  Y_{211}(x,\rho_{211},p_{211},\alpha)&:=Y_{21}(x,\rho_{211}^k,1+\rho_{211}^k p_{211},\alpha).
\end{align*}
Setting $\rho_{211}=\epsilon_{211}=0$, we find $\dot x=0$ and 
\begin{align*}
 \dot p_{211} &= -p_{211}(1-p_{211}).
\end{align*}
Consequently, $(x,0,0,0)$ and $(x,0,1,0)$ are both partially hyperbolic. The former allows us to extend the critical manifold $C_{21}$ in chart $(\bar \alpha=1,\bar{\bar y}=1)_{21}$ onto the blowup sphere as a normally hyperbolic invariant manifold $C_{211}$. In fact, within $\rho_{211}=0$ we have that $p_{211}=-\beta \epsilon_{211}^k$ is a manifold of equilibria $R_{211}$ and $C_{211}$ will therefore include these points, at least locally. 
We will see the resulting slow-fast structure more clearly in the chart $(\bar \alpha=1,\bar{\bar y}=1,\bar \nu_{21}\bar \epsilon_{21}=1)_{213}$ which we analyze in the following section. The hyperbolicity of $C_{211}$ allows us to extend the slow manifold $S_{\epsilon,\alpha}$ as a constant $\epsilon$-section $S_{\epsilon,\alpha,211}$, defined by \eqref{eps211}, of a center manifold $S_{\alpha,211}$.  

\begin{lemma}
There exists an attracting center manifold $S_{\alpha,211}$ of $(x,0,0,0)$ for \eqref{eqns211} for all $0\le \alpha\ll 1$, which is a graph over a compact domain $D_{211}$ in the $(x,\rho_{211},\epsilon_{211})$-space:
 \begin{align*}
  p_{211}=P_{211}(x,\rho_{211},\epsilon_{211},\alpha),
 \end{align*}
 where 
 \begin{align*}
  P_{211}(x,\rho_{211},\epsilon_{211},\alpha)= -\phi_+(\rho_{211} \epsilon_{211})\epsilon_{211}^k\left(1 + k \rho_{211} \epsilon_{211} Y_+(x,\alpha \rho_{211}^k)+\mathcal O(\rho_{211} \epsilon_{211}^2) \right).
 \end{align*}
\end{lemma}
\begin{proof}
 Direct calculation. In the expression for $P_{211}$, we have used that $Y_{211}(x,\rho_{211},0,\alpha) = Y_+(x,\alpha \rho_{211}^k)$.
\end{proof}

The reduced problem on $S_{\alpha,211}$ is given by 
\begin{equation}\nonumber
\begin{aligned}
 \dot x &=\rho_{211}^k \alpha X_{211}(x,\rho_{211},P_{211}(x,\rho_{211},\epsilon_{211},\alpha),\alpha),\\
 \dot \rho_{211} &=\frac{1}{k}\rho_{211} \left[Y_{211}(x,\rho_{211},-\beta \epsilon_{211}^k,\alpha)+k\epsilon_{211}^kY_+(x,\alpha \rho_{211}^k) + \mathcal O(\epsilon_{211}^{k+1})\right],\\
 \dot \epsilon_{211} &=-\frac{k+1}{k}\epsilon_{211}\left[Y_{211}(x,\rho_{211},-\beta \epsilon_{211}^k,\alpha)+k\epsilon_{211}^kY_+(x,\alpha \rho_{211}^k) + \mathcal O(\epsilon_{211}^{k+1})\right],
\end{aligned}
\end{equation}
upon dividing the right hand side by $\rho_{211}\epsilon_{211}$. 
Whenever we have stable sliding, we have $Y_+(x,0)<0$ and we can therefore divide through by $-[\cdots]>0$:
\begin{equation}\eqlab{reducedS11}
\begin{aligned}
 \dot x &=\rho_{211}^k \alpha \left(-\frac{X_+(x,0)}{Y_+(x,0)} + \mathcal O(\epsilon_{211}^k,\alpha)\right),\\
 \dot \rho_{211} &=-\frac{1}{k}\rho_{211} ,\\
 \dot \epsilon_{211} &=\frac{k+1}{k}\epsilon_{211}.
\end{aligned}
\end{equation}
We will now describe a transition map $\mathcal P_{211}^4:\Pi_{11}^4\rightarrow \Pi_{11}^5$ where $\Pi_{11}^4:\, \rho_{211}=c_{in}$ to $\Pi_{11}^5:\,\epsilon_{211}=c_{out}$. We express this map in terms of $(x,\rho_{211},\tilde p_{211},\epsilon_{211})$ with $\tilde p_{211}$ defined by 
\begin{align*}
\tilde p_{211}=p_{211}-P_{211}(x,\rho_{211},\epsilon_{211},\alpha).
\end{align*}
and then restrict $\tilde p_{211}$ to a sufficiently small neighborhood of $0$. 
\begin{lemma}\lemmalab{P114}
 The transition map $\mathcal P_{211}^4$ from $\Pi_{11}^4$ to $\Pi_{11}^5$ takes the following form
 \begin{align*}
 \mathcal P_{211}^4(x,c_{in},\tilde p_{211}, \epsilon_{211},\alpha) = \begin{pmatrix}
                                                     x-c_{in}^k \alpha \frac{X_+(x,0)}{Y_+(x,0)} +\mathcal O(\alpha^2,\alpha \epsilon_{211}^{\frac{k}{k+1}})  \\
                                                     c_{in}(c_{out}^{-1}\epsilon_{211})^{\frac{1}{k+1}}\\
                                    \mathcal O(e^{-c/\epsilon_{211}})\\
                                                     c_{out}
                                                    \end{pmatrix},
 \end{align*}
 for some $c>0$. 
The order of the remainders remain unchanged upon differentiation with respect to $x$ and $\tilde p_{211}$. 
\end{lemma}
\begin{proof}
The proof is standard using Fenichel's theory and normal forms, see e.g. \cite{jones_1995}.
In particular, since $\tilde p_{211}=0$ is invariant, we have
\begin{align}
 \dot p_{211} &= \left(-1+\mathcal O(\epsilon_{211}\rho_{211},\epsilon_{211}^k)\right) p_{211},\eqlab{p211eqn}
\end{align}
upon dropping the tildes. Then upon invoking Fenichel's normal form \cite{jones_1995}, we straighten out the stable fibers of $S_{\alpha,211}$ by setting $\tilde x = x+\mathcal O(\rho_{211}^{k+1}\epsilon_{211}\alpha)$. Then the $(\tilde x,\rho_{211},\epsilon_{211})$-system is independent of $p_{211}$ and described by \eqref{reducedS11} upon dropping the tilde. We then simply integrate the $\rho_{211}$ and $\epsilon_{211}$-equations in \eqref{reducedS11}, insert the resulting expressions into the $x$-equation and estimate $x$. On the other hand, on the time scale of \eqref{reducedS11}, \eqref{p211eqn} becomes
\begin{align*}
 \dot p_{211} &= \rho_{211}^{-1}\epsilon_{211}^{-1} \vert Y_+(x,0)\vert^{-1} \left(-1+\mathcal O(\epsilon_{211}\rho_{211},\epsilon_{211}^k)\right) p_{211}.
\end{align*}
From here, using \eqref{eps211}, we then estimate $p_{211}=\mathcal O(e^{-c/\epsilon_{211}})$ uniformly on $\Pi_{11}^5$ for some $c>0$. The partial derivatives of $\mathcal P_{211}^4$ can be handled in a similar way. The expression for the $\rho_{211}$-component, $\rho_{211,out}$, follows from the conservation of $\epsilon$, recall \eqref{eps211}:
\begin{align*}
 c_{in}^{k+1}\epsilon_{211} = \rho_{211,out}^{k+1} c_{out}.
\end{align*}

\end{proof}

 \subsection{Analysis in the $(\bar \alpha=1,\bar{\bar y}=1,\bar \nu_{21}\bar \epsilon_{21}=1)_{213}$-chart}
 In this chart, we obtain the following equations:
 \begin{equation}\eqlab{eqs213}
 \begin{aligned}
  \dot x &=\rho_{213}^{k+1}\nu_{213} \alpha X_{213}(x,\nu_{213},p_{213},\rho_{213},\alpha) ,\\
  \dot \nu_{213} &= \nu_{213}\left(\rho_{213}  Y_{213}(x,\nu_{213},p_{213},\rho_{213},\alpha) -\phi_+(\rho_{213}\nu_{213}^{-1}) \nu_{213}^{-k}-p_{213}\right),\\
  \dot p_{213} &=-\nu_{213} \left(\phi_+(\rho_{213}\nu_{213}^{-1}) \nu_{213}^{-k}+p_{213}\right) ,
 \end{aligned}
 \end{equation}
 and $\dot \rho_{213}=0$. Notice that we restrict attention to a compact set with $\nu_{213}>0$, to avoid the singularity at $\nu_{213}=0$.
Here we have defined $X_{213}$ and $Y_{213}$ by
\begin{align*}
 X_{213}(x,\nu_{213},p_{213},\rho_{213},\alpha):&=X(x,-\alpha (1+\rho_{213}^k p_{213})+\alpha \rho_{213}^{2k+1} \nu_{213},1+\rho_{213}^kp_{213}),\\
 Y_{213}(x,\nu_{213},p_{213},\rho_{213},\alpha):&=Y(x,-\alpha (1+\rho_{213}^k p_{213})+\alpha \rho_{213}^{2k+1} \nu_{213},1+\rho_{213}^kp_{213}).
 \end{align*}
For $\rho_{213}=0$, which corresponds to $\epsilon=0$, we obtain the layer problem
\begin{align*}
 \dot x &=0,\\
 \dot \nu_{213}&=- \nu_{213}\left(\beta \nu_{213}^{-k}+p_{213}\right),\\
 \dot p_{213} &=-\nu_{213}\left(\beta \nu_{213}^{-k}+p_{213}\right),
\end{align*}
recall \eqref{eq:beta},
writing $\phi_+(0)$ as $\beta_+=\beta$ for simplicity. Consequently, the set $R_{213}$ defined by $p_{213} = -\beta \nu_{213}^{-k}$, $\nu_{213}>0$, $\rho_{213}=0$ is a manifold of equilibria; it coincides with $R_{211}$ from the $(\bar \alpha=1,\bar{\bar y}=1,\bar \nu_{21}=1)_{211}$-chart upon change of coordinates, see \eqref{cc211}. The linearization about any point in $R_{213}$ gives a single nonzero eigenvalue $k\beta\nu_{213}^{-k-1}-1$. This gives the following.
\begin{lemma}\lemmalab{R213}
Let 
\begin{align}
\nu_{213,f}:=\left({k\beta}\right)^{\frac{1}{k+1}}.\eqlab{nu213f}
\end{align}Then $R_{213}$ divides into a repelling part $R_{213,r}$ for $0<\nu_{213}<\nu_{213,f}$ and an attracting part $R_{213,a}$ for $\nu_{213}>\nu_{213,f}$. Moreover, if $Y_+(x,0)<0$ for all $x$ then the degenerate subset $J_{213}$ of $R_{213}$ defined by $\nu_{213}=\nu_{213,f}$ consists of regular jump points. 
\end{lemma}
\begin{proof}
The statement about the jump points follows from an analysis of the reduced problem on $R_{213}$:
\begin{equation}\eqlab{R213red}
\begin{aligned}
 x' &=0,\\
 \nu_{213}' &=Y_+(x,0)\frac{\nu_{213}^{2}}{\nu_{213}-k\beta\nu_{213}^{-k}}.
\end{aligned}
\end{equation}
This can be obtained from \cite{wechselberger2020a} or more directly by writing the slow manifold approximation as 
\begin{align}
 p_{213} &= -\phi_+(\rho_{213} \nu_{213}^{-1}) \nu_{213}^{-k} +\rho_{213} \frac{k\beta \nu_{213}^{-k}}{\nu_{213}-k\beta\nu_{213}^{-k}}Y_{213}(x,\nu_{213},-\beta \nu^{-k},0,0)+\mathcal O(\rho_{213}^2,\rho_{213}\alpha),\eqlab{p213slowman}
\end{align}
where $Y_{213}(x,\nu_{213},-\beta\nu^{-k},0,0)=Y_+(x,0)$, inserting the result into the $(x,\nu_{213})$-subsystem,
writing the system in terms of the slow time and then letting $\rho_{213}\rightarrow 0$.
 
\end{proof}
The dynamics of the layer problem and the reduced problem are illustrated in \figref{pws_hysteresis_3}. 

\begin{figure}
\begin{center}
\includegraphics[width=.7\textwidth]{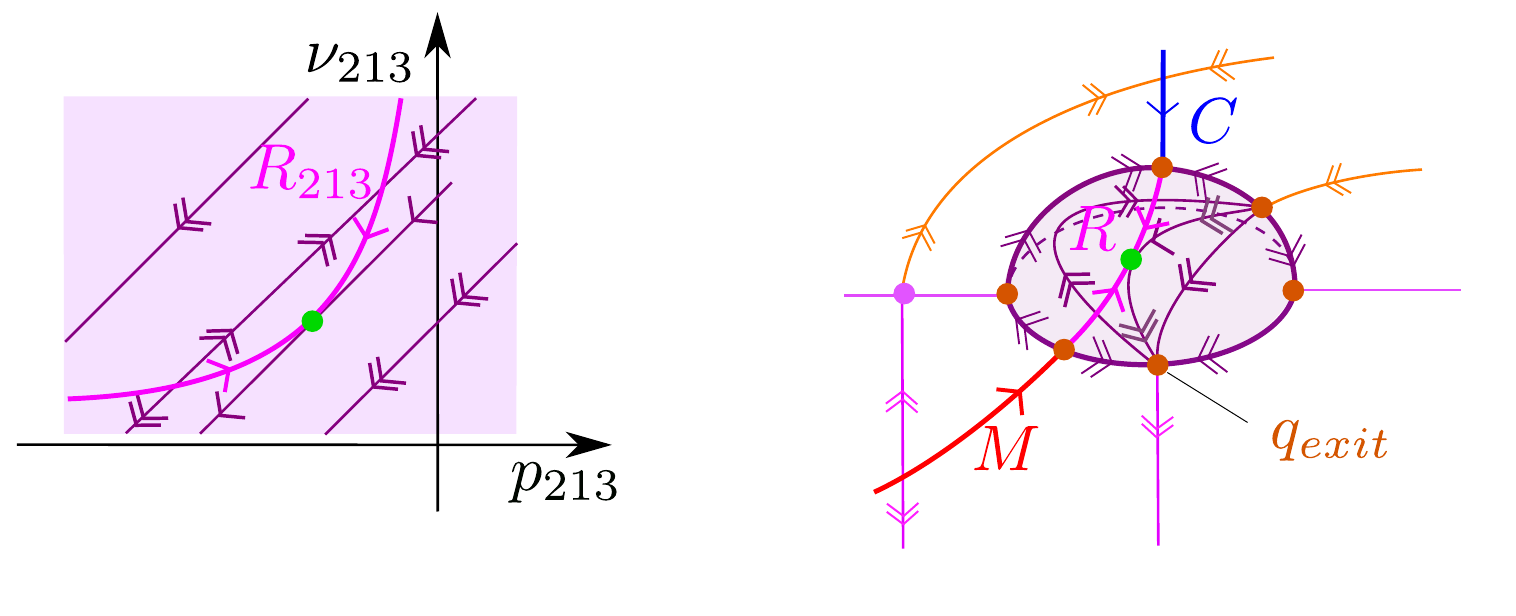}
 \end{center} \caption{To the left, we illustrate the dynamics in the $(\bar \alpha=1,\bar{\bar y}=1,\bar \nu_{21}\bar \epsilon_{21}=1)_{213}$-chart using a projection onto the $(p_{213},\nu_{213})$-coordinate plane. Here we find a critical manifold $R_{213}$ which has a regular fold jump point (in green). On the right, we summarize the findings on the blowup of $Q$. The local diagram on the left covers the subset of the sphere that is bounded away from the edges (purple).   }
 \figlab{pws_hysteresis_3}
              \end{figure}
              \subsection{Exit chart $(\bar \alpha=1,\bar{\bar y}=1,\bar \epsilon_{21}=1)_{212}$}
              In this chart, we obtain the following equations:
              \begin{equation}\eqlab{sphereexit}
              \begin{aligned}
               \dot x &= \rho_{212}^{k+1} \nu_{212}^2 \alpha X_{212}(x,\nu_{212},p_{212},\rho_{212},\alpha),\\
               \dot \nu_{212} &= -(1+k) \nu_{212}\left[-\rho_{212}  \nu_{212} Y_{212}(x,\nu_{212},p_{212},\rho_{212},\alpha) +\phi_+(\rho_{212})+p_{212}\right],\\
               \dot p_{212} &=-kp_{212}\left[-\rho_{212}  \nu_{212} Y_{212}(x,\nu_{212},p_{212},\rho_{212},\alpha) +\phi_+(\rho_{212})+p_{212}\right]-\nu_{212}(\phi_+(\rho_{212})+p_{212})\\
            \dot \rho_{212} &=\rho_{212} \left[-\rho_{212}  \nu_{212} Y_{212}(x,\nu_{212},p_{212},\rho_{212},\alpha) +\phi_+(\rho_{212})+p_{212}\right].
              \end{aligned}
              \end{equation}
              Here we have defined
              \begin{align*}
               X_{212}(x,\nu_{212},p_{212},\rho_{212},\alpha) &= X_{21}(x,\rho_{212}^k\nu_{212},1+\rho_{212}^kp_{212},\alpha),\\
               Y_{212}(x,\nu_{212},p_{212},\rho_{212},\alpha) &= Y_{21}(x,\rho_{212}^k\nu_{212},1+\rho_{212}^kp_{212},\alpha).
              \end{align*}
For $\alpha=0$, $\nu_{212}=0$, we re-discover the manifold of equilibria $M_{22}$ from the chart $(\bar \alpha=1,\bar \epsilon=1)_{22}$ in the following graph form
\begin{align*}
 M_{212}:\quad p_{212} = -\phi_+(\rho_{212}),\quad \rho_{212}>0.
\end{align*}
The graph ends at a partially hyperbolic point $p_{212}=-\phi_+(0)=-\beta<0$, recall \eqref{eq:beta}. On the other hand, consider $\alpha=0$ and the $(\nu_{212},p_{212},\rho_{212})$-subsystem with $x$ fixed. Then the point $q_{exit,212}:\,p_{212}=0$, $\nu_{212}=0$, $\rho_{212}=0$ is fully hyperbolic for the resulting $(\nu_{212},p_{212},\rho_{212})$-subsystem. Indeed the linearization of this system around $(0,0,0)$ produces the following eigenvalues
\begin{align*}
 -(1+k)\beta,\,-k\beta,\,\beta,
\end{align*}
independent of $x$.

For later \response{convenience}, we will now describe details of a transition map $\mathcal P_{212}^{7}:\Pi_{213}^7\rightarrow \Pi_{213}^8$ for all $\alpha\ge 0$ sufficiently near $p_{exit}$ and with $x\in \Pi_{sl}$, from $\Pi_{213}^7:\,\nu_{212}=c_{in}$ to $\Pi_{213}^8:\,\rho_{212}=c_{out}$, with $c_{in}>0$ and $c_{out}>0$ small enough. For this, we first divide the right hand side of \eqref{sphereexit} by the square bracket: $\left[-\rho_{212}  \nu_{212} Y_{212}(x,p_{212},\rho_{212},\alpha) +\phi_+(\rho_{212})+p_{212}\right]>0$, using that this quantity is $\approx \beta$ and therefore positive in a sufficiently small neighborhood of $q_{exit,212}$.
 This gives
 \begin{equation}\eqlab{sphereexit2}
              \begin{aligned}
               \dot x &= \rho_{212}^{k+1} \nu_{212}^2 \alpha \widetilde X_{212}(x,\nu_{212},p_{212},\rho_{212},\alpha),\\
               \dot \nu_{212} &= -(k+1) \nu_{212},\\
               \dot p_{212} &=-kp_{212}-\nu_{212}  + \nu_{212}^2 \rho_{212} Z_{212}(x,\nu_{212},p_{212},\rho_{212},\alpha)\\
            \dot \rho_{212} &=\rho_{212},
              \end{aligned}
              \end{equation}
              where $\widetilde X_{212} = X_{212}/\left[\cdots\right]$ and 
              \begin{align*}
               Z_{212}(x,\nu_{212},p_{212},\rho_{212},\alpha):=
               -\frac{Y_{212}(x,p_{212},\rho_{212},\alpha)}{-\rho_{212}  \nu_{212} Y_{212}(x,p_{212},\rho_{212},\alpha) +\phi_+(\rho_{212})+p_{212}}
              \end{align*}
              \begin{lemma}\lemmalab{P27}
               The transition map $\mathcal P_{212}^7$ for systems \eqref{sphereexit} from $\Pi_{213}^7$ to $\Pi_{213}^8$ is given by 
               \begin{align*}
                \mathcal P_{212}^7 (x,c_{in},p_{212},\rho_{212},\alpha) = \begin{pmatrix}
                \mathcal P_{212x}^7(x,p_{212},\rho_{212},\alpha)\\
                \left(\frac{\rho_{212}}{c_{out}}\right)^{k+1}c_{in}\\
                \mathcal P_{212p}^7(x,p_{212},\rho_{212},\alpha)\\
                c_{out}                                                        
                                                       \end{pmatrix}
               \end{align*}
               where $(x,p_{212})\mapsto \mathcal P_{212x}^7(x,p_{212},\rho_{212},\alpha),\mathcal P_{212p}^7(x,p_{212},\rho_{212},\alpha)$ are both smooth and continuous with respect to $\rho_{212}$ and $\alpha$, satisfying
               \begin{align*}
                \mathcal P_{212x}^7(x,p_{212},\rho_{212},\alpha) = \mathcal O(\rho_{2}^k\alpha),\quad \mathcal P_{212p}^7(x,p_{212},\rho_{212},\alpha)=(p_{212}-c_{in})c_{out}^{-k} \rho_{212}^k+\mathcal O(\rho_{212}^{k+1}).
               \end{align*}
The order of the remainder terms remain unchanged upon differentiation with respect to $x$ and $p_{212}$. 

              \end{lemma}
              \begin{proof}
               We solve \eqref{sphereexit2} for $\nu_{212}$ and $\rho_{212}$, so that $\nu_{212}(t) = c_{in}e^{-(k+1)t}$, $\rho_{212}(t)=e^{t} \rho_{2,in}$ and define $u_2(t)$ by
        $p_{212}(t) = c_{in}e^{-(k+1)t}+(u_2(t)-c_{in})e^{-kt}$. Inserting this into the $p_{212}$ equation gives
        \begin{align*}
         \dot u_2 &= e^{-(k+1)t} c_{in}^2 \rho_{2,in} Z_{212}(x,\nu_{212}(t),c_{in}e^{-(k+1)t}+(u_2(t)-c_{in})e^{-kt},\rho_{212}(t),\alpha),
        \end{align*}
together with
\begin{align*}
 \dot x &=e^{-(k+1)t} \rho_{2,in}^k\alpha \widetilde X_{212}(x,\nu_{212}(t),c_{in}e^{-(k+1)t}+(u_2(t)-c_{in})e^{-kt},\rho_{212}(t),\alpha).
\end{align*}
The transition time is $T=\log \left(c_{out} \rho_{2,in}^{-1}\right)$. Notice that quantities $Z_{212}(\cdots)$, $\widetilde X_{212}(\cdots)$ are uniformly bounded on this domain. By integrating the equations we therefore obtain
\begin{align*}
 u_2(T) = u_2(0) + \mathcal O(\rho_{2,in}),\quad x(T) = x(0) + \mathcal O(\rho_{2,in}).
\end{align*}
              \end{proof}
Recall that $\epsilon=\rho_{212}^{k+1}\nu_{212}$ in this chart.

\subsection{Completing the proof of \lemmaref{lem:P22}}
In \figref{pws_hysteresis_final}, we summarize the findings from our analysis of the two cylindrical blowups and the blowup of $Q$. In particular, the blowup of $Q$ gives rise to an improved singular cycle. 

In \figref{pws_hysteresis_final}, we also indicate different sections $\Pi^i$, $i=1,\ldots,8$, that are each transverse to $\Gamma$, that we use to decompose the return mapping $\mathcal P_{22}$ in \lemmaref{lem:P22}. (The sections $\Pi^{0,1}$ are defined in a neighborhood of $p=0$, whereas $\Pi^{4,7,8}$ are defined on $p=1$. $\Pi^{2,3}$ are defined in between $p=0$ and $p=1$, but sufficiently close to these values, respectively. The remaining sections $\Pi^{5,6}$ are defined on a blowup of $p=1$.)
We describe each of the local mappings $\Pi^{i-1}\rightarrow \Pi^i$, $i=1,\ldots,8$ in the following. We try to strike the balance between including a complete, rigorous and self-contained analysis while at the same time avoiding too many details, that can be found elsewhere in similar contexts. We provide appropriate references along the way. 
\subsubsection*{$\Pi^0\rightarrow \Pi^1$}
The transition from $\Pi^0$ and $\Pi^1$ is regular in the $(\bar \alpha=1,\bar \epsilon)_{22}$-chart. We therefore leave out further details.
\subsubsection*{$\Pi^1\rightarrow \Pi^2$}
On the other hand, the transition map from $\Pi^1$ to $\Pi^2$ is described in the coordinates $(x,\nu_{21},p,\epsilon_{21})$ of the $(\bar \alpha^{-1} \bar y=1,\bar \alpha=1)_{21}$-chart. We therefore consider \eqref{eqn21} and define the sections as follows $\Pi_{21}^1:\, \epsilon_{21}=c_{in},p\in I_{in}$ to $\Pi_{21}^2:\,\nu_{21}=c_{out},\,p\in I_{out}$, with $I_{in}$ and $I_{out}$ open neighborhoods of $p=0$. Notice, for these values of $p$, the set $B_{21}$ is normally hyperbolic, see \figref{pws_hysteresis_21}. 

To describe the mapping $\mathcal P_{21}^1:\Pi_{21}^1\rightarrow \Pi_{21}^2$, it is \response{convenient} to divide the right hand side of the equation \eqref{eqn21} by the square bracket $\left[\cdots\right]$, which is $\approx 1-p$ and therefore positive for all $\epsilon_{21},\nu_{21}\ge 0$ sufficiently small. Seeing that $\frac{dp}{d\nu_{21}}=1$ for $\epsilon_{21}=0$, $\nu_{21}>0$ it is also \response{convenient} to express the map in terms of $\tilde p:=p-\nu_{21}$. This gives
\begin{equation}\eqlab{eqn21mod}
 \begin{aligned}
  \dot x &= \nu_{21}^2 \epsilon_{21} \alpha \widetilde X_{21}(x,\nu_{21},\tilde p,\epsilon_{21},\alpha),\\
  \dot \nu_{21} &=\nu_{21},\\
  \dot{\tilde p}&=\nu_{21}\epsilon_{21} H_{21}(x,\nu_{21},\tilde p,\epsilon_{21},\alpha),\\
  \dot \epsilon_{21}&=-\epsilon_{21},
 \end{aligned}
 \end{equation}
for some new smooth functions $\widetilde X_{21}$ and $H_{21}$. We then have the following.
\begin{lemma}\lemmalab{P211}
 The transition map $\mathcal P_{21}^1$ for system \eqref{eqn21mod} from $\Pi_{21}^1$ to $\Pi_{21}^2$ is given by
 \begin{align*}
  \mathcal P_{21}^1(x,\nu_{21},\tilde p,c_{in},\alpha)=\begin{pmatrix}
                                              \mathcal P_{21x}^1(x,\nu_{21},\tilde p,\alpha)\\
                                              c_{out}\\
                                              \mathcal P_{21p}^1(x,\nu_{21},\tilde p,\alpha)\\
                                              \frac{\nu_{21} c_{in}}{c_{out}}
                                             \end{pmatrix},
 \end{align*}
where $x,\tilde p\mapsto \mathcal P_{21x}^1(x,\nu_{21},\tilde p,\alpha),\mathcal P_{21p}^1(x,\nu_{21},\tilde p,\alpha)$ are both smooth and satisfy
\begin{align*}
 \mathcal P_{21x}^1(x,\nu_{21},\tilde p,\alpha)=\mathcal O(\nu_{21} \alpha \log \nu_{21} ),\quad \mathcal P_{21p}^1(x,\nu_{21},\tilde p,\alpha) = \mathcal O(\nu_{21}\log \nu_{21}),
\end{align*}
with the order of the remainder unchanged upon differentiation with respect to $x$ and $\tilde p$. 
\end{lemma}
\begin{proof}
 The proof is standard, see e.g. \cite[Proposition 2.1]{de2016a}. 
\end{proof}

\subsubsection*{$\Pi^2\rightarrow \Pi^3$}
The transition map from $\Pi^2\rightarrow \Pi^3$ is regular in the $(\bar \alpha=1,\bar{\bar y}=1)_{21}$-chart and further details are therefore left out.
\subsubsection*{$\Pi^3\rightarrow \Pi^{4,5}$}
The transition map from $\Pi^3$ to $\Pi^4$ is obtained from Fenichel's theory near the normally attracting manifold $C$ e.g. by working in the $(\bar \alpha=1)_2$-chart. 
In fact, by working in chart $(\bar \alpha=1,\bar{\bar y}=1)_{21}$ and using the blowup transformation \eqref{blowupsp} this result can be extended all the way up to the section $\Pi^5$ on the blowup of $Q_{21}$. The details were given in \lemmaref{P114}. 
\subsubsection*{$\Pi^5\rightarrow \Pi^{6}$}
The transition map from $\Pi^5$ to $\Pi^6$ is best described in the chart $(\bar \alpha=1,\bar{\bar y}=1,\bar \nu_{21}\bar \epsilon_{21}=1)_{213}$ where the equations are slow-fast. The transition map is then given as a regular fold, jump set with $\rho_{213} = \epsilon$ as the small parameter. See e.g. \cite{szmolyan2004a} for further details. 
\subsubsection*{$\Pi^6\rightarrow \Pi^{7}$}
The exit from the blowup sphere, that we describe by a transition map from $\Pi^6$ to $\Pi^7$ is given by the transition near a resonance saddle. The details were given in \lemmaref{P27}. 

\subsubsection*{$\Pi^7\rightarrow \Pi^{8}$}
The transition map from $\Pi^7\rightarrow \Pi^8$ is regular in the $(\bar \alpha=1,\bar \epsilon=1)_{22}$-chart and further details are therefore left out.
\subsubsection*{Analyzing the half-map: $\Pi^0\rightarrow \Pi^8$}
First, we state a simple corollary of the analysis above.
\begin{corollary}
Upon extension by the forward flow, the slow manifold $S_{\epsilon,\alpha}$ intersects $\Pi^8_{22}$ in chart $(\bar \alpha=1,\bar \epsilon=1)_{22}$ in a curve defined by
\begin{align*}
 y_{22}=0,\,p=s_{22}(x,\epsilon,\alpha),
\end{align*}
where
\begin{align*}
 s_{22}(x,\epsilon,\alpha) =1+\mathcal O(\epsilon^{\frac{k}{k+1}}),
\end{align*}
with the order of the remainder being unchanged upon differentiation with 
respect to $x$. 
\end{corollary}
This essentially follows from \lemmaref{P27} with $\rho_{212}\response{\approx} \epsilon$.

Now, let $\Pi^8_{22}$ be defined by $y_{22}=0$, $p\in I^8$ a small neighborhood of $p=1$ and $x\in \Sigma$. 
From the proceeding analysis, the map $\mathcal Q_{22}:\,\Pi^0_{22}\rightarrow \Pi^8_{22}$, $(x,p)\mapsto (x_+,p_+)$ is well-defined for all $\epsilon,\alpha>0$ sufficiently small. In particular, we have
\begin{align*}
 x_+(x,p,\epsilon,\alpha) &=x+\alpha (1-p)\vert Y_+(x,0)\vert^{-1} X_+(x,0)+\mathcal O(\alpha^2,\epsilon^{\frac{k}{k+1}} \alpha),\\
 p_+(x,p,\epsilon,\alpha)&=s_{22}(x,\epsilon,\alpha)+\mathcal O(e^{-c/\epsilon}),
\end{align*}
with the order of the remainder unchanged under differentiation with respect to $x,p$. Here the leading order expression for $x_+$ follows from \lemmaref{P114} with $c_{in}=(1-p)$, recall also \lemmaref{het}. The expression for the map from $\Pi^8_{22}$ to $\Pi^0_{22}$ is similar; the leading order terms follow by replacing $+$ by $-$ and by replacing $1$ in the expression for $s_{22}$ by $0$. This completes the proof of \lemmaref{lem:P22} (upon redefining $s_{22}$). 

\begin{figure}
\begin{center}
\includegraphics[width=.75\textwidth]{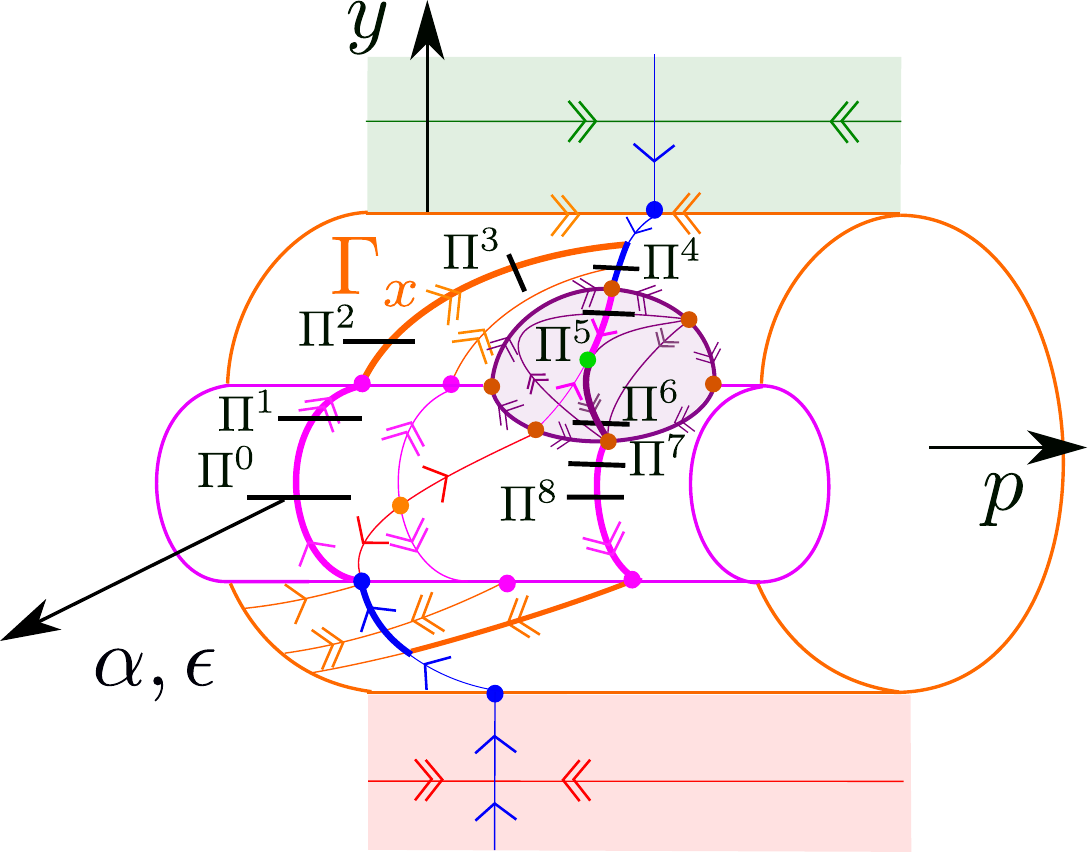}
 \end{center} \caption{The dynamics of the full desingularized system, including the spherical blowup of $Q$. \response{The cycle $\Gamma_x$ (thick curves) has improved hyperbolicity properties}. We also indicate the sections $\Pi^{0-8}$ used in the proof of \lemmaref{lem:P22}.  }
 \figlab{pws_hysteresis_final}
              \end{figure}
              
%
 \section{Main results in the case of grazing}\seclab{main2}
 In this section, we consider \eqref{xypmodel} under the following assumption (which replaces assumption \ref{assumption:4} henceforth):
 \begin{assumption}
	\label{assumption:5}
	The PWS system $Z_\pm$ is planar $z=(x,y)\in \mathbb R^2$ and each $Z_\pm$ depends smoothly on an unfolding parameter \response{$\mu\approx 0$ defined in a neighborhood of $0$}. In particular, for $\mu=0$, $Z_+$ has a hyperbolic and repelling limit cycle $\gamma_0$ that has a quadratic tangency with $\Sigma$ at $x=0$. $Z_-$, on the other hand, is assumed to be transverse to $\Sigma$. 
\end{assumption}

Consequently, for $\mu=0$ we have that $(x,y)=(0,0)$ is a visible fold point  \cite{jeffrey_geometry_2011,2019arXiv190806781U} of the piecewise smooth system $Z_\pm$, see $T$ in \figref{visible0}. In fact, by the implicit function theorem, $Z_+$ has visible fold point for each $\mu\response{\approx} 0$ and this point depends smoothly on $\mu$. Then upon using \response{\cite[Proposition 14]{reves_regularization_2014}, see also \cite{2019arXiv190806781U},} we \response{can} transform the PWS system $Z_\pm$ locally into 
\begin{align}
 Z_+(z,\mu) =\begin{pmatrix}
            1+f(z,\mu)\\
            2x+yg(z,\mu)
           \end{pmatrix},\quad Z_-(z,\mu) = \begin{pmatrix}
           0\\
           1
\end{pmatrix},\eqlab{Zpmvf}
\end{align}
by a $C^\infty$-diffeomorphism.
Here $f$ and $g$ are smooth functions with $f(0,\mu)=0$ for all $\mu\response{\approx} 0$; for \eqref{Zpmvf} the fold point is therefore fixed at $(x,y)=(0,0)$. 
This is the system that we will use to study the local dynamics near $(x,y)=(0,0)$. We will henceforth suppress the dependency of $f$ and $g$ on $\mu$ since this will play little role.

Since the limit cycle $\gamma_0$ in assumption \ref{assumption:5} is hyperbolic for $Z_+$, we have a repelling limit cycle $\gamma_\mu$ of $Z_+$ for every $\mu\response{\approx} 0$. Let $Y(\mu)=\min_t y(t)$ along $\gamma_\mu$ so that $Y(0)=0$.  We assume the following degeneracy condition.
\begin{assumption}
 \label{assumption:6}
 $Y'(0)> 0$. 
\end{assumption}
We illustrate the setting in \figref{grazing}. 

Under these assumptions, the reference \cite{2019arXiv190806781U} proved that the system obtained from regularization by smoothing \eqref{xysm}, has a locally unique saddle-node bifurcation of limit cycles at $\mu=o(1)$ with respect to $\epsilon\rightarrow 0$. On the other hand, the reference \cite{rev2021a} also showed that the system obtained from regularization by hysteresis has chaotic dynamics (through a Baker-like map) for all $\alpha>0$ sufficiently small provided $\mu\response{\approx} 0$ is sufficiently small. In this section, we try to bridge these two results by working on \eqref{xypmodel}, using (as in \cite{2019arXiv190806781U}) the normal form \eqref{Zpmvf} to perform the analysis near $(x,y)=(0,0)$.

\begin{figure}
\begin{center}
\includegraphics[width=.47\textwidth]{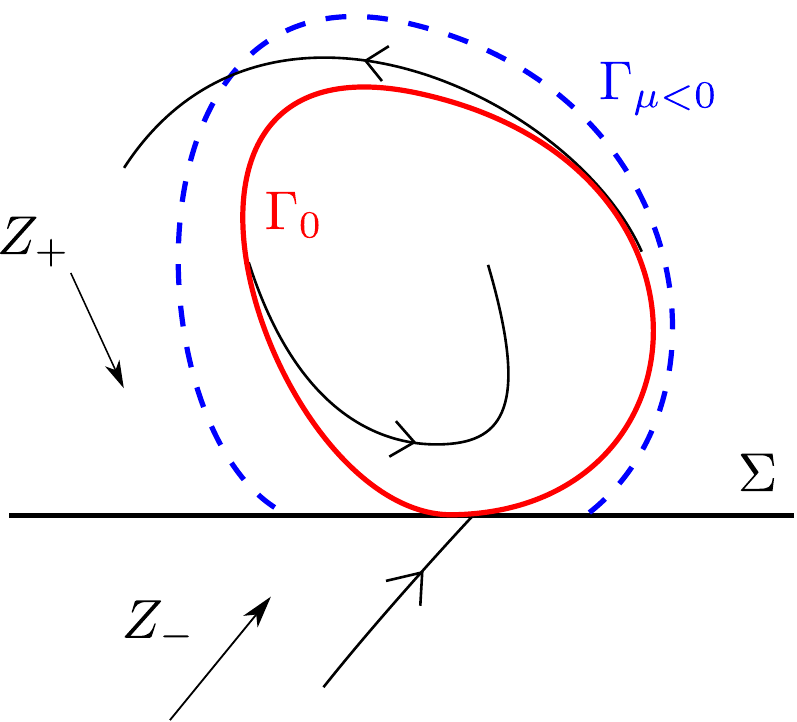}
 \end{center} \caption{The grazing bifurcation. We assume that the smooth vector-field $Z_+$ has a repelling limit $\Gamma_0$ for $\mu=0$ having a quadratic tangency with $\Sigma$. Under a further degeneracy condition, which ensures that the perturbation $\Gamma_\mu$ of $\Gamma_0$ as a limit cycle of $Z_+$ for $\mu\response{\approx} 0$ transverses $\Sigma$ with nonzero speed, see assumption \ref{assumption:6}, the reference \cite{rev2021a} have shown that, while  regularization by smoothing leads to a saddle-node bifurcation of limit cycle \cite{2019arXiv190806781U}, regularization by hysteresis leads to chaotic dynamics. \thmref{main2} is an attempt to bridge these two regimes by working on \eqref{xypmodel}.}
 \figlab{grazing}
              \end{figure}


To present the result, we define two wedge-shaped regions in the $(\epsilon,\alpha)$-plane. Firstly, for $\epsilon_0>0$, $\alpha_0>0$, let $W_1(\epsilon_0,\alpha_0)$ be the region defined by $0<\alpha\le \epsilon^{2k}\alpha_0$ for $0<\epsilon\le\epsilon_0$. On the other hand, let $W_2(\epsilon_0,\epsilon_1,\alpha_0)$ be the region defined by $$0<\alpha^{\frac{k+1}{k}}\epsilon_0<\epsilon\le \alpha^{\frac{k+1}{k}}\epsilon_1,$$ for $0<\alpha\le \alpha_0$ and $0<\epsilon_0<\epsilon_1$. We illustrate the two regions in \figref{W12}. These regions do not overlap for $\epsilon_1>\epsilon_0> 0$ and $\alpha_0\ge 0$ sufficiently small.

In the following, we will sometimes write $W_1(\epsilon_0,\alpha_0)$ and $W_2(\epsilon_0,\epsilon_1,\alpha_0)$ as $W_1$ and $W_2$ for simplicity. 
\begin{figure}
\begin{center}
\includegraphics[width=.46\textwidth]{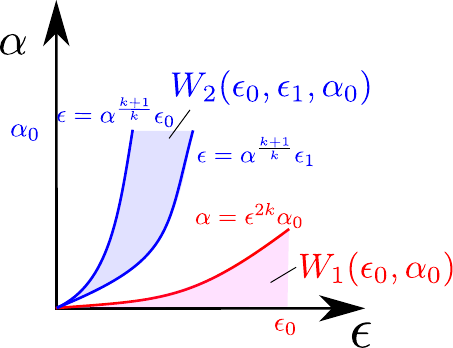}
 \end{center} \caption{The two regions in the $(\epsilon,\alpha)$-plane relevant for \thmref{main2}.}
 \figlab{W12}
              \end{figure}

\response{
For $N\in \mathbb N$, $N\ge 2$, let $\Sigma^N$ denote the space of all $N$-symbol sequences $s=\{\ldots,s_{-1},s_0,s_1,\ldots\}$, $s_i\in \{0,1,\ldots,N-1\}$ for all $i\in \mathbb Z$, equipped with the complete metric:
\begin{align*}
 d(s,\tilde s) = \sum_{i=-\infty}^\infty \frac{1}{2^{\vert i\vert}} \frac{\delta_{s_i,\tilde s_i}}{1+\delta_{s_i,\tilde s_i}},\quad \delta_{i,j}:=\begin{cases}
1& \text{for}\,i=j\\                                                                                                                                                0& \text{for}\,i\ne j,                                                                                                                                                \end{cases}
\end{align*}
see \cite[Chapter 24.1]{wiggins2003a}. 
\begin{proposition}\proplab{shift}
\cite[Proposition 24.2.2]{wiggins2003a}
 The full shift on $N$ symbols $\sigma:\Sigma^N\rightarrow \Sigma^N$, defined by 
\begin{align*}
 \sigma(s)=\{\ldots,s_{0},s_{1},s_2,\ldots\},\quad \mbox{that is}\quad (\sigma(s))_i=s_{i+1}\quad\mbox{for all}\quad i\in \mathbb Z,
\end{align*}
is continuous and chaotic in the following sense:
\begin{enumerate}
 \item There is a countable infinity of periodic orbits, consisting of orbits of all periods.
 \item There is an uncountable infinity of nonperiodic orbits.
 \item There is a dense orbit. 
\end{enumerate}
\end{proposition}
The case $N=2$ is the most familiar one, since this is the shift map relevant to the standard Smale's horseshoe. 
}

\begin{thm}\thmlab{main2}
 Consider \eqref{xypmodel} under the assumptions \ref{assumption:1}, \ref{assumption:2}, \ref{assumption:3}, \ref{assumption:5} and \ref{assumption:6} so that \eqref{Zpaf} holds with $Z_\pm$ given in a small neighborhood of $(x,y)$ by \eqref{Zpmvf}. Fix any $n\in \mathbb N$. Then for $\epsilon_1>\epsilon_0>0$ and $\alpha_0>0$ all sufficiently small, we have the following:
 \begin{enumerate}
  \item \label{item1} For any $(\epsilon,\alpha)\in W_1(\epsilon_0,\alpha_0)$, there exists a $\mu\response{\approx} 0$ such that the system \eqref{xypmodel} has a saddle-node bifurcation of limit cycles.
  \item \label{item2} For any $(\epsilon,\alpha)\in W_2(\epsilon_0,\epsilon_1,\alpha_0)$, there exists a $\mu\response{\approx} 0$ such that there is a return map defined by the system \eqref{xypmodel} having an invariant cantor set upon which the map is homeomorphic to the full shift $\sigma:\Sigma^N\rightarrow \Sigma^N$ on $N$ symbols. 
 \end{enumerate}
\end{thm}
The assumption \ref{assumption:3} ($p\mapsto Z(z,p)$ is affine) is mainly added for simplicity. In fact, it is not needed in item \ref{item1} and the statement of item \ref{item2} could also be generalized by including a milder assumption on $p\mapsto Z(z,p)$ at $x=0$, see \remref{ass3main2}. \response{We again expect that our approach can be modified to obtain a similar result for the Sotomayor-Teixeira regularization functions, see \eqref{Sotomayor}.} We leave these generalizations to the interested reader.


To prove the theorem, we have to describe the local transition near the grazing point with $\Sigma$. 
Before going into details, we first emphasize that $Z_\pm$ in \eqref{Zpmvf} has stable sliding for $x<0$ and crossing for $x>0$ along $\Sigma$. Therefore the blowup dynamics for $x<0$ in a compact interval is covered by \thmref{main1} and the blowup dynamics for $\epsilon=\alpha=0$ is therefore as in  \figref{pws_hysteresis_final} in this case. The blowup dynamics for $x>0$ on the other hand, where assumption \ref{assumption:4} is violated and crossing occurs, is shown in \figref{pws_hysteresis_final2}. This follows from the blowup analysis with $Y_+>0$. In each of the two diagrams, \figref{pws_hysteresis_final} and  \figref{pws_hysteresis_final2}, $x$ is constant on the cylinders and there is only slow flow in the $\bar y$-direction. In order to describe the details of the dynamics associated with the visible fold, we will need to zoom in on $x=0$ so that the dynamics in this direction for $0<\epsilon,\alpha\ll 1$ becomes comparable with the dynamics in the $\bar y$-direction. We achieve this zoom through blowup. In particular,  in \secref{grazingC1}, we first reduce to the slow manifold $S_{\epsilon,\alpha}$ obtained as a perturbation of the critical manifold $C$ in the $(\bar y=1)_1$-chart and then perform two separate blowup transformations. In the parameter regime $(\epsilon,\alpha)\in W_1$, this is sufficient to prove \thmref{main2} (\ref{item1}). Interestingly, we find that the details are similar to those in \cite{2019arXiv190806781U} covering the grazing bifurcation in the case of regularization by smoothing. 

\begin{figure}
\begin{center}
\includegraphics[width=.7\textwidth]{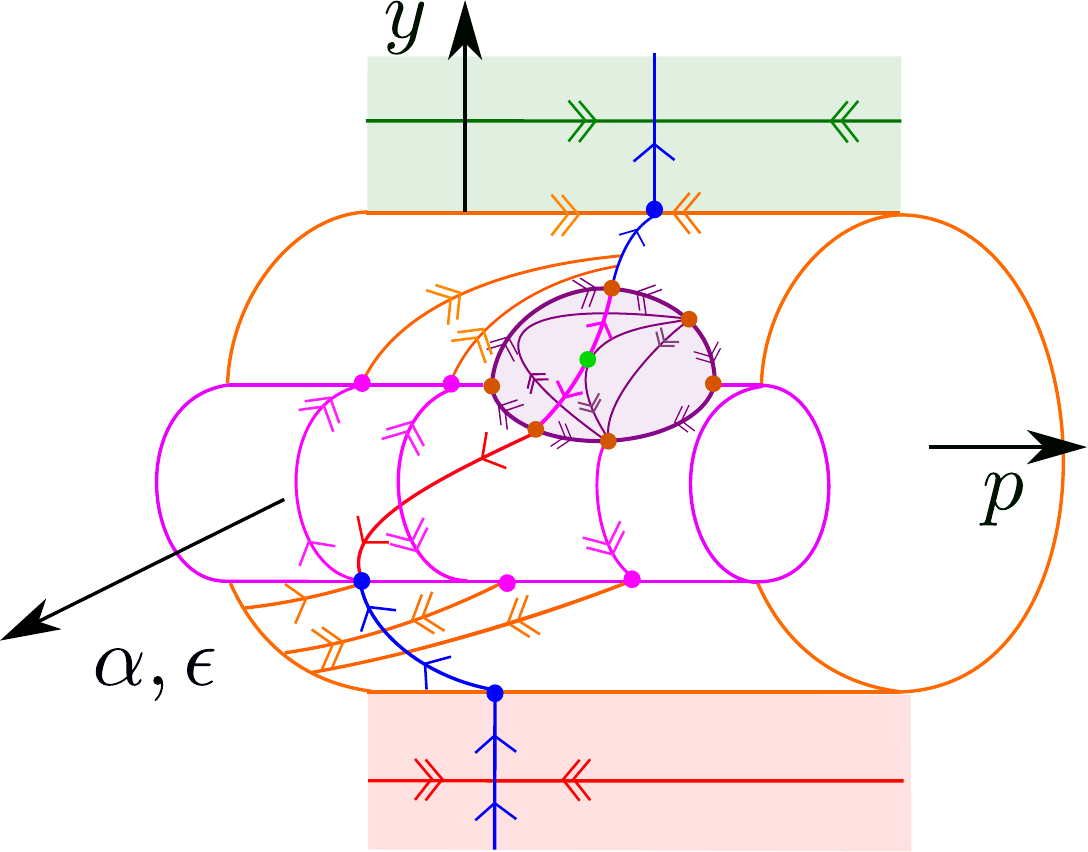}
 \end{center} \caption{Blowup dynamics in the case of crossing upwards where $Y_+(x,0)>0,Y_-(x,0)>0$ (corresponding to $x>0$ for the visible fold in \figref{visible0}). In this case, the flow on $C$ (blue) moves upwards and there is no equilibrium of the reduced flow on $M$ (moving downwards) when assumption \ref{assumption:3} holds.}
 \figlab{pws_hysteresis_final2}
              \end{figure}

On the other hand, in order to prove \thmref{main2} (\ref{item2}) in the regime $(\epsilon,\alpha)\in W_2$ we have to follow dynamics that becomes unbounded in the chart $(\bar y=1)_1$. In \secref{Qblowupvf}, we will specifically work on the blowup of $Q$. Here we will study the reduced problem on the critical manifold $R_{213}$ in the $(\bar \alpha=1,\bar{\bar y}=1,\bar \nu_{21}\bar \epsilon_{21}=1)_{213}$-chart for $x\response{\approx} 0$ using a separate blowup transformation. This gives rise to a folded saddle singularity \response{\cite{szmolyan_canards_2001}} for $(\epsilon,\alpha)\in W_2$ and \response{an associated} canard orbit  \response{along which (extended versions of) the slow manifolds $S_{\epsilon,\alpha}$ and $N_{\epsilon,\alpha}$, obtained as perturbations of $C$ and $M$ in chart $(\bar y=1)_1$ and $(\bar \alpha=1,\bar \epsilon=1)_{22}$, respectively, intersect transversally (see \propref{canard})}. This provides the main horseshoe-like mechanism  for the chaotic dynamics in \thmref{main2} (\ref{item2}). In fact, the geometric construction is similar to \cite{kristiansen2021a}, which (inspired by the work of \cite{haiduc2009a} on the forced van der Pol) proved existence of chaos in a friction oscillator in the presence of slow-fast and nonsmooth effects. We therefore complete the proof of \thmref{main2} (\ref{item2}) in \secref{main22} by exploiting this connection. 

In the proof of \thmref{main2}, we will therefore again try to strike the balance between including a complete, rigorous and self-contained analysis, while at the same time avoiding too many details, that can be found elsewhere (\cite{2019arXiv190806781U} for item \ref{item1} and \cite{kristiansen2021a,haiduc2009a} for item \ref{item2}) in similar contexts. 

\response{Finally, we should emphasize that the mechanism we find for the chaotic dynamics in case (ii) is very similar in nature to the one used in \cite{rev2021a} to prove existence of chaos in the case of hysteresis. This horseshoe-like mechanism occurs in an exponentially small regime (with respect to $\epsilon,\alpha\rightarrow 0$) and is therefore probably not troubling from
an engineering perspective. Moreover, any time-series of the chaotic dynamics would appear to be
periodic, with only very minor changes in the amplitudes at each oscillation. This has been referred to as micro-chaotic dynamics, see \cite{glendinning2010a} (for micro-chaotic dynamics in the context of hysteresis). }



\subsection{Analysis of the slow flow on $S_{\epsilon,\alpha}$ in the case of the visible fold}\seclab{grazingC1}
In this section, we work in the $(\bar y=1)_1$-chart and consider the reduced flow on $S_{\epsilon,1}$, recall \lemmaref{Seps1}, in the case of \eqref{Zpmvf}. For this, we use \eqref{xr1alpha1} with $Z_\pm$ as in \eqref{Zpmvf}:
\begin{equation}\eqlab{xr1alpha1vis}
\begin{aligned}
 \dot x &= r_1 (1+f(x,r_1))\left(1-\beta\epsilon^k\alpha_1^k + \mathcal O(\epsilon^{k+1}  \alpha_1^{k+1})\right),\\
 \dot r_1 &= r_1 \left[(2x+r_1g(x,r_1))\left(1-\beta\epsilon^k\alpha_1^k\right)+\beta\epsilon^k\alpha_1^k + \mathcal O(\epsilon^{k+1} \alpha_1^{k+1})\right],\\
 \dot \alpha_1 &= -\alpha_1 \left[(2x+r_1g(x,r_1))\left(1-\beta\epsilon^k\alpha_1^k\right)+\beta\epsilon^k\alpha_1^k + \mathcal O(\epsilon^{k+1} \alpha_1^{k+1})\right].
\end{aligned}
\end{equation}
The dynamics of this system within the invariant subspaces $\alpha_1=0$ and $r_1=0$ are illustrated in \figref{pws_hysteresis_visiblexr1a1}. Notice that $r_1=x=0,\alpha_1\ge 0$ is a line of degenerate singularities for $\epsilon=0$. We will again need to perform consecutive blowup transformations to resolve the degeneracy stemming from the terms of the form $\epsilon^k\alpha_1^k$. For this, we first blowup with respect to $\alpha_1$ and then subsequently blowup with respect to $\epsilon$. In further details, we first apply the transformation \response{(also used in \cite{2019arXiv190806781U})}:
\begin{align}
 \sigma\ge 0,(\bar x,\bar r_1,\bar \alpha_1)\in S^2\mapsto \begin{cases}
                                                     x &=\sigma^k \bar x,\\
                                                     r_1 &=\sigma^{2k} \bar r_1,\\
                                                     \alpha_1 &= \sigma \bar \alpha_1.
                                                    \end{cases}\eqlab{blowupvf1}
\end{align}
\response{Notice that weights on $x$ and $\alpha_1$ are so that the terms $2x$ and $\beta \epsilon^k \alpha_1^k$ in the equations for $r_1$ and $\alpha_1$ in \eqref{xr1alpha1vis} balance up. At the same time, the weights on $x$ and $r_1$ are so that the quadratic tangency of the grazing orbit within $\alpha_1=0$ of the vector-field $Z_+$ with the $x$-axis (see \figref{pws_hysteresis_visiblexr1a1}) is ``broken''.}

Since $r_1, \alpha_1\ge 0$ we are only interested in the subset of $S^2$ where $\bar r_1, \bar \alpha_1\ge 0$. 
This gives a vector-field $\overline V_1$ on $\sigma\ge 0,(\bar x,\bar r_1,\bar \alpha_1)\in S^2$ by pull-back of \eqref{xr1alpha1vis}. It has $\sigma^k$ as common factor and it is therefore $\widehat V:=\sigma^{-k}\overline V$ that has improved hyperbolicity properties. 

 \begin{figure}
\begin{center}
\includegraphics[width=.65\textwidth]{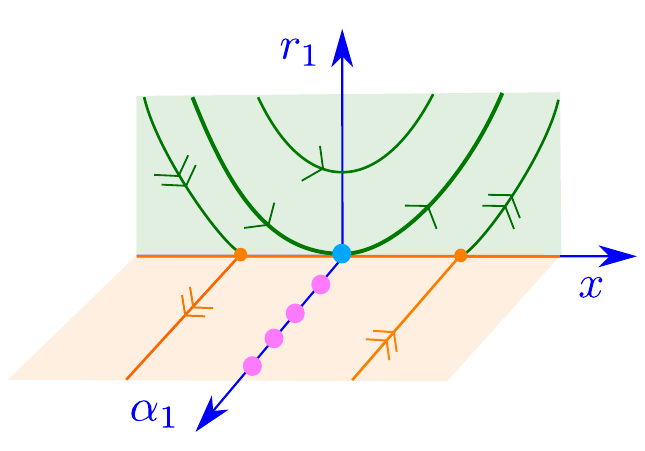}
 \end{center}\caption{Dynamics of the reduced problem on $S_{\epsilon,1}$, see \eqref{xr1alpha1vis}, within the invariant subspaces $\alpha_1=0$ and $r_1=0$.  }
 \figlab{pws_hysteresis_visiblexr1a1}
              \end{figure}

It is not difficult to analyze $\widehat V$  in the directional charts. In particular, in the chart defined by
\begin{align*}
 (\bar y=1,\bar r_1 =1)_{11}:\quad \begin{cases}
                                    x &=\sigma_{11}^k x_{11},\\
                                    r_1 &=\sigma_{11}^{2k},\\
                                    \alpha_1 &=\sigma_{11} \alpha_{11},
                                   \end{cases}
\end{align*}
where $\alpha = \sigma_{11}^{2k+1}\alpha_{11}$, we obtain the following equations
\begin{equation}\eqlab{eqns11vf}
\begin{aligned}
 \dot x_{11} &= (1+ \sigma_{11}^k f_{11}(x_{11},\sigma_{11}^k))(1+\mathcal O(\epsilon^k \sigma_{11}^k\alpha_{11}^k))
 -\frac12 x_{11}\left[\cdots \right],\\
 \dot \sigma_{11}&=\frac{1}{2k}\sigma_{11} \left[\cdots \right],\\
 \dot \alpha_{11}&=-\frac{2k+1}{2k}\alpha_{11} \left[\cdots\right],
\end{aligned}
\end{equation}
where
\begin{align*}
 \left[\cdots \right] = (2x_{11}+\sigma_{11}^k g_{11}(x_{11},\sigma_{11}^k))(1+\mathcal O(\epsilon^k \sigma_{11}^k \alpha_{11}^k))+\beta\epsilon^k \alpha_1^k + \mathcal O(\epsilon^{k} \sigma_{11}\alpha_{11}^k),
\end{align*}
with 
$\sigma_{11}^k f_{11}(x_{11},\sigma_{11}^k):=f(x,r_1)$, and $g_{11}(x_{11},\sigma_{11}^k):=g(x,r_1)$.
We find two hyperbolic equilibria:
\begin{align}
q_{11}^\pm :\quad (x_{11},\sigma_{11},\alpha_{11})=(\pm 1,0,0),\eqlab{q11pm}
\end{align} 
for any $\epsilon\ge 0$. The eigenvalues of the linearization around these points are 
\begin{align}
 -2x_{11},\frac{1}{k}x_{11},-\frac{2k+1}{k}x_{11},\eqlab{eigval1}
\end{align}
with $x_{11}=\pm 1$ at the two points $q_{11}^\pm$,
respectively. Whereas the point $q_{11}^-$ has a two-dimensional unstable manifold within $\sigma_{11}=0$, and a one-dimensional unstable manifold within $\alpha_{11}=0$ (corresponding to the grazing orbit of the PWS system \eqref{Zpmvf} within $x<0$), the point $q_{11}^+$ has a two-dimensional stable manifold within $\sigma_{11}=0$ and a one dimensional unstable manifold within $\alpha_{11}=0$ (corresponding to the grazing orbit of the PWS system \eqref{Zpmvf} within $x>0$).
Compare also with \figref{pws_hysteresis_visiblexr1a1} and \figref{pws_hysteresis_visibleblowup1}.

The dynamics on the sphere is given by 
\begin{equation}\eqlab{eqnsphere}
\begin{aligned}
 \dot x_{13} &=1,\\
 \dot r_{13}&=2x_{13},
\end{aligned}
 \end{equation}
upon using the coordinates $(x_{13},r_{13})$ defined by 
\begin{align*}
 (\bar y=1,\bar r_1 \bar \alpha_1=1)_{13}:\,\begin{cases}
                                             x&=\sigma_{13}^k x_{13},\\
                                             r_1 &=\sigma_{13}^{2k} r_{13},\\
                                             \alpha_1 &=\sigma_{13} r_{13}^{-1},
                                            \end{cases}
\end{align*}
where $\alpha=r_1\alpha_1 = \sigma_{13}^{2k+1}$. However, in the chart defined by 
\begin{align*}
 (\bar y=1, \bar \alpha_1=1)_{12}:\,\begin{cases}
                                             x&=\sigma_{12}^k x_{12},\\
                                             r_1 &=\sigma_{12}^{2k} r_{12},\\
                                             \alpha_1 &=\sigma_{12},
                                            \end{cases}
\end{align*}
we find that $x_{12}=r_{12}=0,\sigma_{12}\ge 0$ is a degenerate line for $\epsilon=0$. We summarize the findings in \figref{pws_hysteresis_visibleblowup1}.

             \begin{figure}
\begin{center}
\includegraphics[width=.7\textwidth]{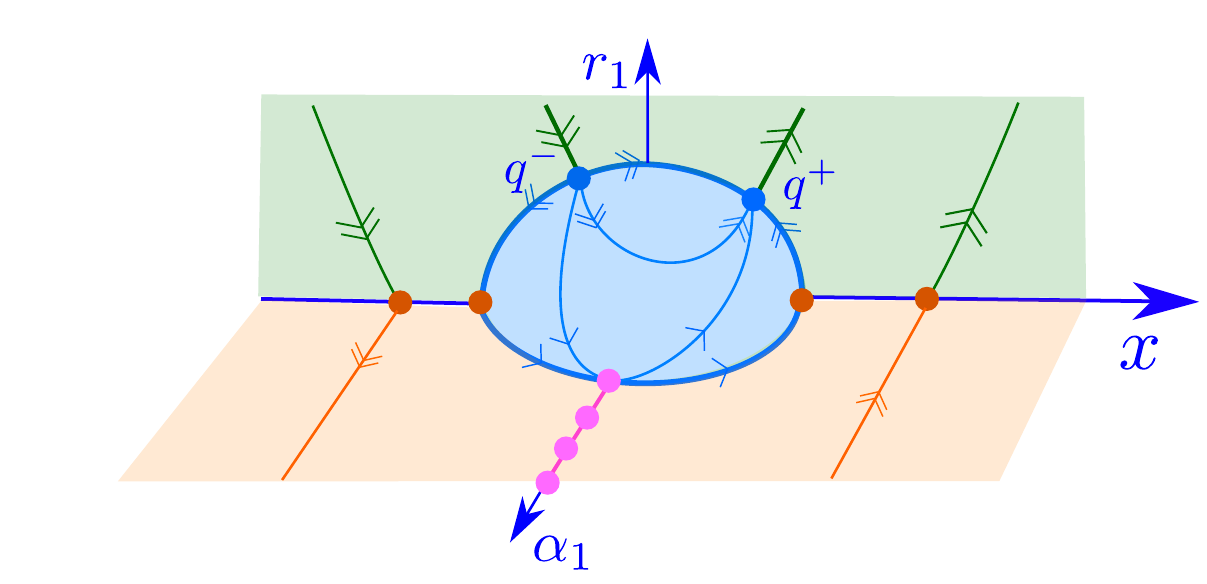}
 \end{center}\caption{The reduced flow on $S_{\epsilon,1}$ in the $(\bar \alpha=1)_1$-chart upon blowing up $\alpha_1=r_1=x=0$ to a sphere (in blue).  }
 \figlab{pws_hysteresis_visibleblowup1}
              \end{figure}

To gain hyperbolicity and resolve the dynamics near the degenerate line (pink in \figref{pws_hysteresis_visibleblowup1}), we proceed to augment $\dot \epsilon=0$ and then apply the following cylindrical blowup transformation:
\begin{align}
 \xi \ge 0,(\bar x_{12},\bar r_{12},\bar \epsilon)\in S^2\mapsto \begin{cases}
                                                     x_{12} &=\xi^k \bar x_{12},\\
                                                     r_{12} &=\xi^{2k} \bar r_{12},\\
                                                  \epsilon &= \xi \bar \epsilon,
                                                    \end{cases}\eqlab{blowup12cyl}
\end{align}
leaving $\sigma_{12}$ untouched. \response{(This transformation can be motivated in the same way as \eqref{blowupvf1}.)}
Let $\widehat V_{12}$ be the vector-field in the $(\bar y=1,\bar \alpha_{1}=1)_{12}$-chart with $\dot \epsilon=0$ augmented. The blowup transformation \eqref{blowup12cyl} then gives a vector-field $\overline{\widehat V}_{12}$ by pull-back of $\widehat V_{12}$. It has $\xi^{k}$ as a common factor and it is therefore $\widehat{\widehat V}_{12}:=\xi^{-k} \overline{\widehat V}_{12}$ that we shall study. 

To study $\widehat{\widehat V}_{12}$ and cover the relevant part of the sphere, we use two charts:
\begin{align}
 (\bar y=1,\bar \alpha_1=1,\bar r_{12}=1)_{121}:\quad \begin{cases} 
                                                       x_{12} &=\xi_{121}^k x_{121},\\
                                                       r_{12}&=\xi_{121}^{2k},\\
                                                       \epsilon &=\xi_{121}\epsilon_{121},
                                                      \end{cases}\eqlab{c121}\\
                                                    (\bar y=1,\bar \alpha_1=1,\bar \epsilon=1)_{122}:\quad \begin{cases} 
                                                       x_{12} &=\xi_{122}^k x_{122},\\
                                                       r_{12}&=\xi_{122}^{2k}r_{122},\\
                                                       \epsilon &=\xi_{122}.
                                                      \end{cases}\nonumber
\end{align}
The change of coordinates is given by the following expressions:
\begin{align*}
 \begin{cases}
\xi_{121} &= \xi_{122}r_{122}^{\frac{1}{2k}},\\
 x_{121} &= r_{122}^{-\frac12} x_{122},\\
 \epsilon_{121}&=r_{122}^{-\frac{1}{2k}}.
 \end{cases}
\end{align*}

 In the chart $(\bar y=1,\bar \alpha_1=1,\bar r_{12}=1)_{121}$, where 
 \begin{align}
  \alpha = \sigma_{12}^{2k+1}\xi_{121}^{2k},\quad \epsilon = \xi_{121}\epsilon_{121},\eqlab{alphaeps121}
 \end{align}
are conserved,
  we (again) find two hyperbolic equilibria at 
  \begin{align}\eqlab{eq121}
  z_{121}^\pm:\quad (x_{121},\sigma_{12},\xi_{121},\epsilon_{121})=(\pm 1,0,0,0).
  \end{align}
  The eigenvalues of the linearization around these points are given by
  \begin{align}
 -2x_{121},-2x_{121},-\frac{2k+1}{k}x_{121},\frac{2k+1}{k}x_{121},\eqlab{eigval2}
\end{align}
with $x_{121}=\pm 1$ at the two points $z_{121}^\pm$, respectively. The unstable manifold for $z_{121}^-$ is three-dimensional and contained within $\xi_{121}=0$. However, for $z_{121}^-$ it is the stable manifold that is three dimensional; in fact, $z_{121}^+$ will be the $\omega$-limit set of all points with $\xi_{121}=0,\epsilon_{121}\ne 0$. Notice, that since
\begin{align}
 \alpha \epsilon^{-2k} = \sigma_{12}^{2k+1} \epsilon_{121}^{-2k},\eqlab{aeps_2k}
\end{align}
see \eqref{alphaeps121},
each three-dimensional invariant manifold is foliated by constant values of $\sigma_{12}^{2k+1} \epsilon_{121}^{-2k} = \text{const}.$
  
  To describe the dynamics in further details, we focus on the cylinder $\xi=0$, $(\bar x_{12},\bar r_{12},\bar \epsilon)\in S^2$, $\sigma_{12}\ge 0$, and the two invariant subspace of $\widehat{\widehat V}_{12}\vert_{\xi=0}$ given by $\sigma_{12}=0$ and $\bar \epsilon=0$. The reason for doing so, is that these invariant spaces capture different scaling regimes of $\epsilon$ and $\alpha$. In particular, within the  $(\bar  y=1,\bar \alpha_1=1,\bar r_{12}=1)_{121}$-chart, \eqref{alphaeps121} holds and on $\sigma_{12}=\text{const.}$ we therefore have by \eqref{aeps_2k} that
\begin{align}
\epsilon^{2k}{\sim} \alpha \epsilon_{121}^{2k}.\eqlab{sim}
\end{align}
(Here we have used $\sim$ to indicate that two quantities differ by a constant that only depends upon the constant value of $\sigma_{12}$.)
Consequently, orbits lie close to $\bar \epsilon=0$ (i.e. $\epsilon_{121}=0$) provided that 
\begin{align}\eqlab{regime}
0<\epsilon^{2k}\ll \alpha\ll 1.
\end{align}
Notice also that on $\sigma_{12}=\text{const.}$ we have 
\begin{align}
 x\sim \sqrt{\alpha} x_{121},\eqlab{xorderoncyl}
\end{align}
upon eliminating $\xi_{121}$. This will be important later on. 

On the other hand, in the  $(\bar y=1,\bar \alpha_1=1,\bar \epsilon =1)_{122}$-chart, we have 
%
%
\begin{align*}
 \alpha =\sigma_{12}^{2k+1} \xi_{122}^{2k} r_{122},\quad   \epsilon = \xi_{122}.
\end{align*}
and
$r_{122}=\text{const}.$ therefore corresponds to 
\begin{align*}
 \alpha \sim \sigma_{12}^{2k+1} \epsilon^{2k}.
\end{align*}
Consequently, orbits follow $\sigma_{12}=0$ provided that $0<\alpha\ll \epsilon^{2k}\ll 1$. 

  We study each of these invariant subspaces in the following using the two charts $(\bar y=1,\bar \alpha_1=1,\bar r_{12}=1)_{121}$ and $(\bar y=1,\bar \alpha_1=1,\bar \epsilon =1)_{122}$. 
    \subsubsection*{Dynamics of $\widehat{\widehat V}_{12}\vert_{\xi=0}$ in the invariant subspace $\bar \epsilon=0$}
  In the $(\bar y=1,\bar \alpha_1=1,\bar r_{12}=1)_{121}$-chart, we obtain the following local form of $\widehat{\widehat V}_{12}$ within $\xi_{121}=\epsilon_{121}=0$:
  \begin{equation}\eqlab{boring121}
  \begin{aligned}
   \dot x_{121} &= 1-x_{121}^2,\\
   \dot \sigma_{12}&=2 \sigma_{12}x_{121}.
  \end{aligned}
  \end{equation}
The dynamics of this system are illustrate in \figref{pws_hysteresis_visible121}. Notice in particular, that there are two invariant lines 
\begin{align}\eqlab{L121pm}
L^\pm_{121}:\, x_{121}=\pm 1,
\end{align}along which we have $\dot \sigma_{12}>0$ and $\dot \sigma_{12}<0$ for $\sigma_{12}\ne 0$. These sets therefore belong to the stable and unstable manifolds of the points $z_{121}^\pm$, given by \eqref{eq121}, respectively. Notice also that the dynamics within $\xi=\bar \epsilon=0$ is unbounded (only bounded on one side of $L_{121}^-$). 
\begin{lemma}\lemmalab{reflection}
Consider any $\tilde c^3>0$ and let $\widetilde{\mathcal P}_{121x}^3$ denote the $x_{121}$-component of the transition map of \eqref{boring121} from 
 \begin{align*}
  \widetilde \Pi_{121}^3:\quad \sigma_{12}=\tilde c^3>0,\,x_{121}<0,
 \end{align*}
to 
\begin{align*}
  \widetilde \Pi_{121}^4:\quad \sigma_{12}=\tilde c^3>0,\,x_{121}>0.
 \end{align*}
 Then $\widetilde{\mathcal P}_{121x}^3$ is only well-defined for $x_{121}\in (-1,0)$ and here it is given by the reflection around $x_{121}=0$:
 \begin{align}
  \widetilde{\mathcal P}_{121x}^3(x_{121})=-x_{121},\quad x_{121}\in (-1,0).\eqlab{reflection}
 \end{align}

\end{lemma}
\begin{proof}
 Direct calculation. Notice in particular, that if $t_{121}$ denotes the time in \eqref{boring121}, then this system is reversible with respect to $(x_{121},\sigma_{12},t_{121})\mapsto (-x_{121},\sigma_{12},-t_{121})$. From this \eqref{reflection} follows. 
\end{proof}

  \subsubsection*{Dynamics of $\widehat{\widehat V}_{12}\vert_{\xi=0}$ in the invariant subspace $\sigma_{12}=0$}
  In the $(\bar y=1,\bar \alpha_1=1,\bar \epsilon =1)_{122}$-chart, we obtain the following local form of $\widehat{\widehat V}_{12}$ within $\xi_{122}=\sigma_{12}=0$:
  \begin{equation}\eqlab{eqns122}
  \begin{aligned}
   \dot x_{122} &= kx_{122}\left[\beta +2x_{122}\right]+r_{122},\\
   \dot r_{122}&=(2k+1)r_{122} [\beta+2x_{122}].
  \end{aligned}
  \end{equation}
Within $r_{122}=0$ we find two equilibria, one given by $x_{122}=0$ and another given by $x_{122} = -\frac{\beta}{2}$. The first point is hyperbolic and repelling for \eqref{eqns122} whereas the second one is partially hyperbolic, the linearization having a single nonzero and negative eigenvalue. A simple calculation, reveals the following:
\begin{lemma}
 There exists a unique, attracting center manifold $G_{122}$ for \eqref{eqns122} of the point $(x_{122},r_{122}) =( -\frac{\beta}{2},0)$. $G_{122}$ is its (nonhyperbolic) unstable manifold, along which $r_{122}$ is increasing.
\end{lemma}

Upon using that $[\cdots]$ occurs in both equations of \eqref{eqns122}, it is a direct calculation to show that the transformation:
\begin{align}
 (x_{122},r_{122})\mapsto \begin{cases}
                           u &= \left(\beta r_{122}^{-\frac{1}{2}}\right)^{-\frac{1}{2k+1}} x_{122},\\
                           v &=\left(\beta r_{122}^{-\frac{1}{2}}\right)^{-\frac{2}{2k+1}},
                          \end{cases}\eqlab{chinimap}
\end{align}
for $r_{122}>0$, brings \eqref{eqns122} into the Chini-equation \cite{2019arXiv190806781U,trifonov2011a}:
\begin{equation}\eqlab{chini}
\begin{aligned}
 \dot u &= 1,\\
 \dot v &=2u+v^{-k}.
\end{aligned}
\end{equation}
This equation also appeared in the blowup analysis of the grazing bifurcation for regularization by smoothing in \cite{2019arXiv190806781U}. In particular, from this reference we obtain the following result (see \figref{pws_hysteresis_visible122} for an illustration).
\begin{lemma}\lemmalab{chini}
Consider any $c^3>0$ and
 let $x_{122}\mapsto \mathcal P^3_{122x}(x_{122})$ denote the $x$-component of the transition map of \eqref{eqns122} from $$\Pi_{122}^3:\quad r_{122}=c^3,\,x_{122}<-\frac12 \beta,$$ to $$\Pi_{122}^4:\quad r_{122}=c^3,\,x_{122}>-\frac{1}{2}\beta.$$ Then 
 \begin{align}
  (\mathcal P^3_{122x})'(x_{122})\in (-1,0),\quad (\mathcal P^3_{122x})''(x_{122})<0,\eqlab{P3122xest}
 \end{align}
and
\begin{align*}
 \lim_{x_{122}\rightarrow -\frac12 \beta^-} (\mathcal P^3_{122x})'(x_{122}) = -1,\quad \lim_{x_{122}\rightarrow -\infty} (\mathcal P^3_{122x})'(x_{122}) = 0. 
\end{align*}

\end{lemma}
\begin{proof}
See 
\cite[Lemma 3.12]{2019arXiv190806781U} \response{(and \cite{kristiansen2023a})} describing a similar transition map for the Chini equation. By inverting \eqref{chinimap}, we obtain the desired result.
\end{proof}


\begin{remark}
Within $\xi_{122}=r_{122}=0$ we have the following
 \begin{equation}\nonumber
  \begin{aligned}
   \dot x_{122} &= kx_{122}\left(\beta +2x_{122}\right),\\
   \dot \sigma_{12}&=-\sigma_{12}\left(\beta +2x_{122}\right),
  \end{aligned}
  \end{equation}
  and hence $\sigma_{12}\ge 0$, $x_{12}=\xi_{122}=r_{122}=0$ is contained within the stable manifold of $(x_{122},r_{122},\sigma_{12},\xi_{122})=0$. 
Moreover, $x_{122}=-\frac{\beta}{2}$, $\sigma_{12}\ge 0$, $\xi_{122}=r_{122}=0$ is a normally hyperbolic critical manifold $H_{122}$. Through desingularization (by division by $r_{122}$) it is possible to show that $\sigma_{12}$ is monotonically decreasing on $H_{122}$. 
\end{remark}
We summarize the findings in the two charts in \figref{pws_hysteresis_visible2}.
 \begin{figure}
\begin{center}
\includegraphics[width=.57\textwidth]{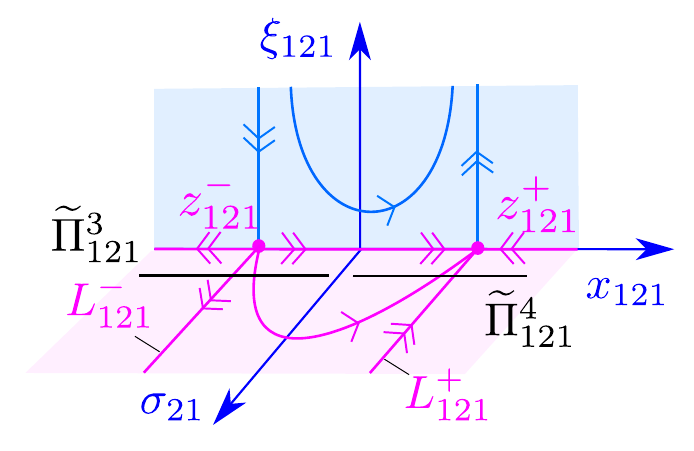}
 \end{center}\caption{Dynamics in the $(\bar y=1,\bar \alpha_1=1,\bar r_{12} =1)_{122}$-chart on $\bar \epsilon=0$. }
 \figlab{pws_hysteresis_visible121}
              \end{figure}
              \begin{figure}
\begin{center}
\includegraphics[width=.65\textwidth]{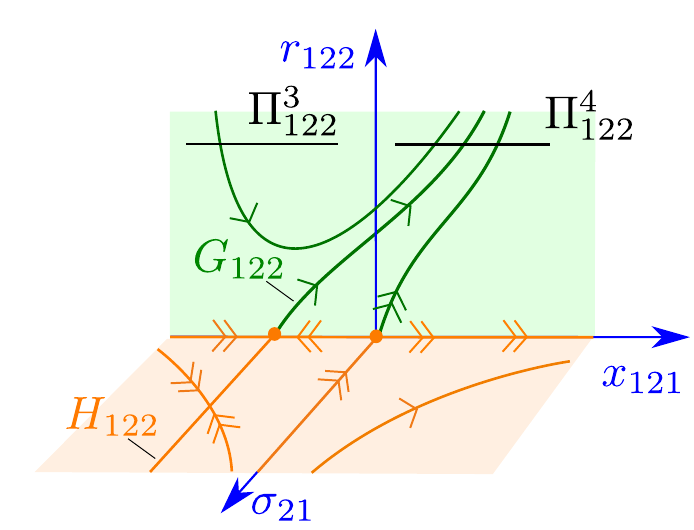}
 \end{center}\caption{Dynamics in the $(\bar y=1,\bar \alpha_1=1,\bar \epsilon =1)_{122}$-chart. In this chart we find a unique center manifold $G_{122}$ within $\sigma_{21}=0$. The mapping from $\Pi_{122}^0\rightarrow \Pi_{122}^1$ is described by the Chini-equation and it is contractive and concave as a function of $x$ on $\Pi_{122}^1$, see \lemmaref{chini}. This property is essential in the proof of \thmref{main2} (\ref{item2}). }
 \figlab{pws_hysteresis_visible122}
              \end{figure}
             \begin{figure}
\begin{center}
\includegraphics[width=.7\textwidth]{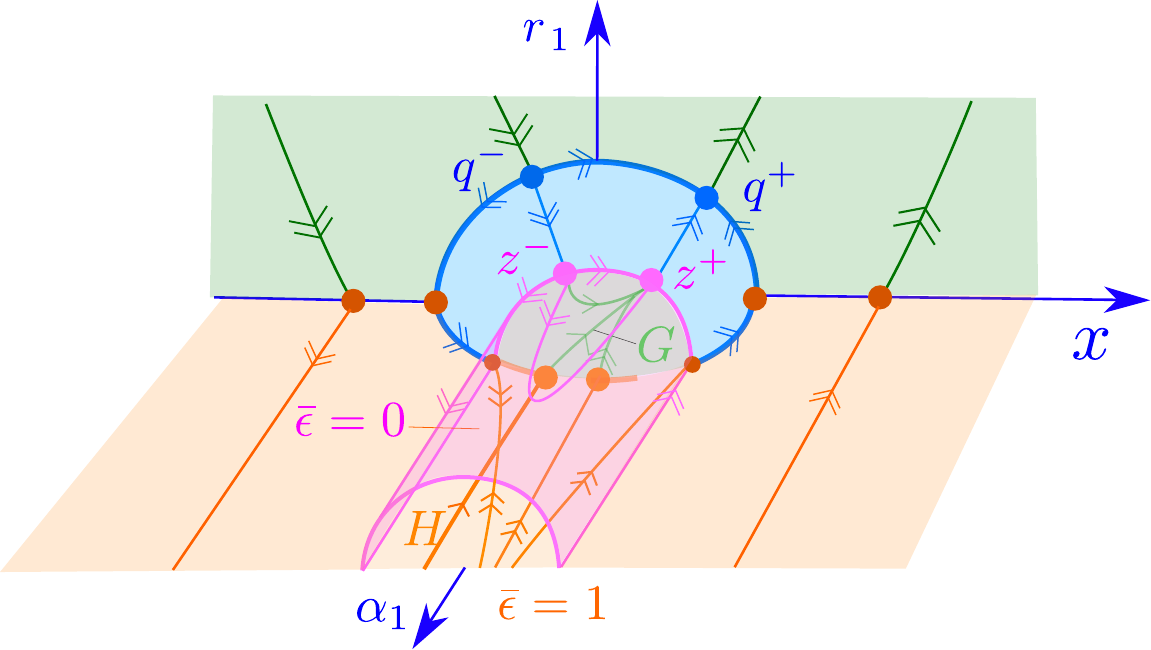}
 \end{center}\caption{The reduced flow on $C_1$ in the $(\bar \alpha=1)_1$-chart upon two consecutive blowup transformations of the degenerate set $\alpha_1\ge 0$, $r_1=x=\epsilon=0$. The dynamics on the cylinder obtained by the blowup transformation \eqref{blowup12cyl} (its boundary $\bar \epsilon=0$ being indicated in pink) breaks up into different regimes, depending on the ratio of $\epsilon$ and $\alpha$. For example, whenever $(\epsilon,\alpha)\in W_2$ then the dynamics near $\bar \epsilon=0$ (pink) becomes relevant, whereas within $(\epsilon,\alpha)\in W_1$ the green region where $\bar \epsilon>0$, described by the Chini-equation \eqref{chini}, becomes relevant. In this region, which is more visible in \figref{pws_hysteresis_visible2item1}, the attracting center manifold $G$ produces a contraction -- which is absent for $(\epsilon,\alpha)\in W_2$, see \lemmaref{reflection} -- of the return map $\mathcal P_{loc}$, see \lemmaref{chini}. It is the balance of this contraction and the expansion along $\gamma_0$ that gives rise to the saddle-node bifurcation in \thmref{main2} (\ref{item1}).}
 \figlab{pws_hysteresis_visible2}
              \end{figure}
\subsection{Proof of \thmref{main2} (\ref{item1})} \seclab{main21}

For the proof \thmref{main2} (\ref{item1}) we work on the slow manifold $S_{\epsilon,\alpha}$ that has been extended, through the blowup approach in \secref{main1}, to the first blowup cylinder. On this manifold, using the $(x,y)$-coordinates and the system \eqref{Zpmvf} locally near $(x,y)=0$, we then consider the return map $\mathcal P$ on a section $\Pi_{in}=\{(x,y): y=c_{in},x\in I_{in}\}$, for some appropriate closed interval $I_{in}\subset (-\infty,0)$ so that $\Pi_{in}$ is transverse to $\gamma_0$. We then decompose $\mathcal P$ into a local transition map $\mathcal P_{loc}:\Pi_{in}\rightarrow \Pi_{out}$,  with $\Pi_{out}=\{(x,y,p): y=c_{out},x\in I_{out}\}$, see \figref{visible0}, and a global map $\mathcal P_{glo}:\Pi_{in}\rightarrow \Pi_{out}$. The latter is regular on the attracting slow manifold, and we therefore turn our attention to $\mathcal P_{loc}$.

In order to describe $\mathcal P_{loc}$, we use the chart $(\bar y=1)_1$ and the blowup transformations \eqref{blowupvf1} and \eqref{blowup12cyl}, that resolve the degeneracy of $x=r_1=0$, $\alpha_1\ge 0$ for $\epsilon=0$, and chop the mapping into separate transition maps, see \figref{pws_hysteresis_visible2item1}: $\mathcal P^0:\,\Pi^0\rightarrow \Pi^1$ near $q^-$, a regular map $\mathcal P^1:\Pi^1\rightarrow \Pi^2$ being a regular perturbation of \eqref{eqnsphere}, $\mathcal P^2:\Pi^2\rightarrow \Pi^3$ near $z_-$, a regular map $\mathcal P^3:\Pi^3 \rightarrow \Pi^4$ being a regular perturbation of the map in \lemmaref{chini}, $\mathcal P^4:\Pi^4\rightarrow \Pi^5$ near $z^+$, a regular map $\mathcal P^5:\Pi^5\rightarrow \Pi^6$ being a regular perturbation of \eqref{eqnsphere}, and finally $\mathcal P^6:\Pi^6\rightarrow \Pi^7$ near $q^+$. 

 \begin{figure}
\begin{center}
\includegraphics[width=.7\textwidth]{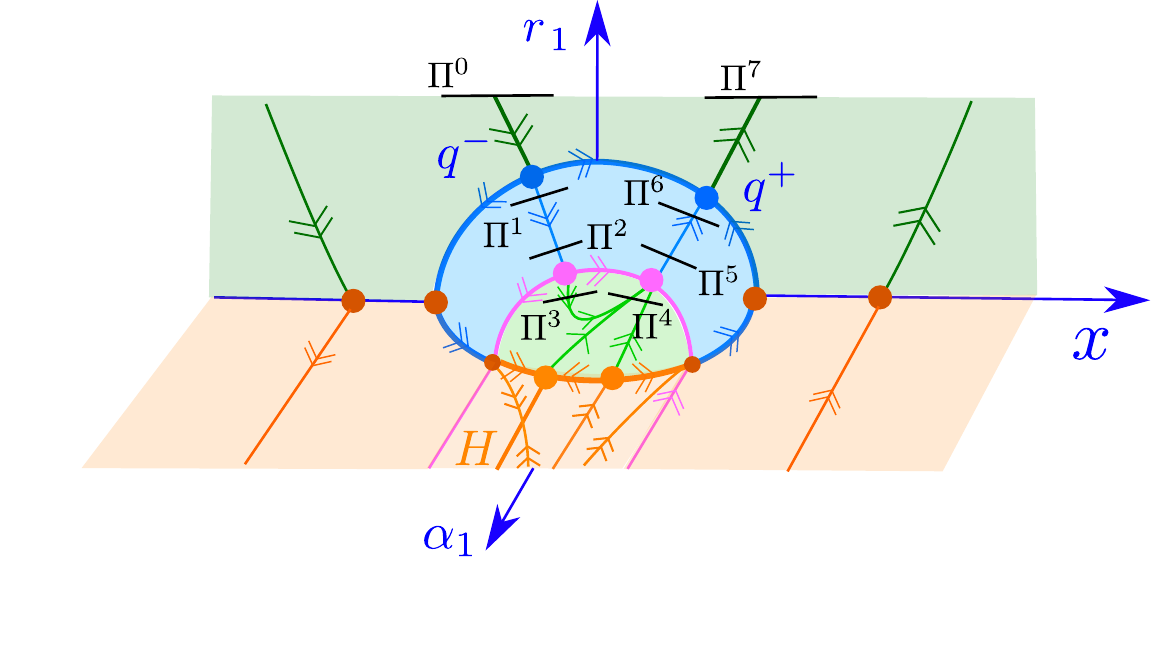}
 \end{center}\caption{Illustration of the sections $\Pi^{0-7}$ relevant in the proof of \thmref{main2} (\ref{item1}). In comparison with \figref{pws_hysteresis_visible2}, we leave out the dynamics on the cylinder $\bar \epsilon=0$, since this regime is not relevant for the proof of \thmref{main2} (\ref{item1}). }
 \figlab{pws_hysteresis_visible2item1}
              \end{figure}
Although the eigenvalues near the points $q^\pm,z^\pm$ are resonant, it is possible, following \cite{2019arXiv190806781U}, to achieve a (suitable) linearization near each of this points. We will only present the details near $q^-$ and $z^-$. 
\subsubsection*{Local transition map near $q^-$}
Consider \eqref{eqns11vf} and divide the right hand side $-\frac12 [\cdots]$, which is $\approx -x_{11}$ and therefore positive near $q_{11}^-$. This gives
\begin{equation}\eqlab{eqns11vf2}
\begin{aligned}
 \dot x_{11} &= x_{11}-\frac{2(1+\sigma_{11}^k f_{11}(x_{11},\sigma_{11}^k))}{2x_{11}+\sigma_{11}^kg_{11}(x_{11},\sigma_{11}^k)}+\epsilon^k \alpha_{11}^k A_{11}(x_{11},\sigma_{11},\alpha_{11},\epsilon),\\
 \dot \sigma_{11}&=-\frac{1}{k}\sigma_{11},\\
 \dot \alpha_{11}&=\frac{2k+1}{k}\alpha_{11} ,
\end{aligned}
\end{equation}
for $A_{11}$ smooth. 
\begin{lemma}\lemmalab{qminusmap}
 There exists a smooth diffeomorphism of the form
 \begin{align*}
  (\tilde x_{11},\tilde \sigma_{11},\tilde \alpha_{11})\mapsto \begin{cases}
                                                                x_{11} &=\mathcal X_{11}(\tilde x_{11},\tilde \sigma_{11}^k,\tilde \alpha_{11},\epsilon),\\
                                                                \sigma_{11} &=\tilde \sigma_{11} \mathcal S_{11}(\tilde x_{11},\tilde \sigma_{11}^k,\tilde \alpha_{11},\epsilon),\\
                                 \alpha_{11} &=\tilde \alpha_{11} \mathcal S_{11}(\tilde x_{11},\tilde \sigma_{11}^k,\tilde \alpha_{11},\epsilon)^{-2k-1},                        \end{cases}
 \end{align*}
 as well as a regular transformation of time, such that \eqref{eqns11vf2} becomes
 \begin{equation}\eqlab{eqns11vf3}
\begin{aligned}
 \dot{\tilde x}_{11} &= 2\tilde x_{11}+\epsilon^k \tilde \alpha_{11}^k \widetilde A_{11}(\tilde x_{11},\tilde \sigma_{11},\tilde \alpha_{11},\epsilon),\\
 \dot{\tilde \sigma}_{11}&=-\frac{1}{k}\tilde \sigma_{11},\\
 \dot{\tilde \alpha}_{11}&=\frac{2k+1}{k}\tilde \alpha_{11} .
\end{aligned}
\end{equation}
 Here $\mathcal X_{11}$, $\mathcal S_{11}$ and $\widetilde A_{11}$ are all smooth and satisfy $\mathcal X_{11}(0,0,0,0)=-1$, $\mathcal S_{11}(0,0,0,0)=1$ and
\begin{align*}
\widetilde A_{11} (x_{11},\sigma_{11},\alpha_{11},\epsilon)=\mathcal O(\sigma_{11}\alpha_{11}, \sigma_{11}^k),
\end{align*}
respectively.
\end{lemma}
\begin{proof}
The proof can be found in \cite{2019arXiv190806781U}, see Lemma 3.5 and Lemma 3.6 in this reference, but essentially we use that the $\alpha_{11}=0$ subsystem is equivalent to $z'=Z_+(z)$ which is regular. This enables a linearization within $\alpha_{11}=0$ through the flow box theorem. Subsequently, we linearize the non-resonant system within $\sigma_{11}=0$.
\end{proof}
Consider \eqref{eqns11vf3} and notice that $\alpha=\tilde \sigma_{11}^{2k+1}\tilde \alpha_{11}$ is still conserved in the tilde variables. We therefore drop the tildes and describe the transition map $\mathcal P_{11}^0$ from $\Pi_{11}^0:\,\sigma_{11}=c_{in}$ to $\Pi^1_{11}:\,\alpha_{11}=c_{out}$ by integrating these equations. This produces the following result.
\begin{lemma}\lemmalab{P110}
 $\mathcal P_{11}^0$ is well-defined for $x_{11} \in \left[-c\left(\alpha_{11} c_{out}^{-1}\right)^{\frac{2k}{2k+1}},c\left(\alpha_{11} c_{out}^{-1}\right)^{\frac{2k}{2k+1}}\right]$ with $c>0$ fixed small enough and 
 given by  $(x_{11},c_{in},\alpha_{11})\mapsto (\mathcal P_{11x}^0,c_{in}(\alpha_{11}c_{out}^{-1})^{\frac{1}{2k+1}},c_{out})$ 
 with 
 \begin{align*}
 \mathcal P_{11x}^0(x_{11},\alpha_{11},\epsilon) = \left(\alpha_{11} c_{out}^{-1}\right)^{-\frac{2k}{2k+1}} x_{11}+\mathcal O(\epsilon^{k}\alpha_{11}^{\frac{1}{2k+1}}).
\end{align*}
The order of the remainder terms does not change upon differentiation with respect to $x_{11}$.
\end{lemma}
\begin{proof}
 Simple calculation.
\end{proof}

The analysis near $q^+$ is almost identical. In particular, although the local mapping near $q^-$ is expanding, the local mapping near $q^+$ contracts by the same order. 
\subsubsection*{Local transition map near $z^-$}
We work in the $(\bar y=1,\bar \alpha_1=1,\bar r_{12}=1)_{121}$-chart. Here we have the following equations
\begin{align*}
\dot x_{121}&=(1+\sigma_{12}^k \xi_{121}^k f_{121}(x_{121},\sigma_{12}^k \xi_{121}^k))(1+\mathcal O(\xi_{121}^k \epsilon_{121}^k \alpha_{12}^k))-\frac12 x_{121}\left[\cdots \right],\\
 \dot \xi_{121}&=\frac{2k+1}{2k}\xi_{121} \left[\cdots\right],\\
 \dot \sigma_{12}&=-\sigma_{12} \left[\cdots\right],\\
 \dot \epsilon_{121}&=-\frac{2k+1}{2k}\epsilon_{121}\left[\cdots\right],
\end{align*}
where
\begin{align*}
 \left[\cdots \right] &= 2x_{121}+\sigma_{12}^k\xi_{121}^{k(2k-1)} g_{121}(x_{121},\sigma_{12}^k \xi_{121}^k)(1-\beta \xi_{121}^k \epsilon_{121}^k\sigma_{12}^k)\\
 &+\beta \epsilon_{121}^k \sigma_{12}^k + \mathcal O(\xi_{121} \epsilon_{121}^{k+1} \sigma_{12})
\end{align*}
Moreover, $$\sigma_{12}^k \xi_{121}^k f_{121}(x_{121},\sigma_{12}^k \xi_{121}^k):=f(\sigma_{12}^k \xi_{121}^k x_{121},\sigma_{12}^{2k} \xi_{121}^{2k}),$$ which is well-defined since $f(0,0)=0$, and $$g_{121}(x_{121},\sigma_{12}^k \xi_{121}^k):=g(\sigma_{12}^k \xi_{121}^k x_{121},\sigma_{12}^{2k} \xi_{121}^{2k}).$$
Working near $z_{121}^-$ where $x_{121}=-1$, we divide the right hand side by $-\frac12 \left[\cdots\right]\approx 1$. This gives the following equivalent system
\begin{equation}\eqlab{eqs121}
\begin{aligned}
\dot x_{121}&=x_{121}-\frac{2(1+\sigma_{12}^k \xi_{121}^k f_{121}(x_{121},\sigma_{12}^k \xi_{121}^k))}{2x_{121}+\sigma_{12}^k\xi_{121}^{k(2k-1)} g_{121}(x_{121},\sigma_{12}^k \xi_{121}^k)}+ A_{121}(x_{121},\xi_{121},\sigma_{12},\epsilon_{121}),\\
 \dot \xi_{121}&=-\frac{2k+1}{k}\xi_{121},\\
 \dot \sigma_{12}&=2\sigma_{12},\\
 \dot \epsilon_{121}&=\frac{2k+1}{k}\epsilon_{121},
\end{aligned}
\end{equation}
with $A_{121}(x_{121},\xi_{121},\sigma_{12},\epsilon_{121})=\mathcal O(\epsilon_{121}^k\sigma_{12}^k,\xi_{121}\epsilon_{121}^{k+1}\sigma_{12}).$
\begin{lemma}
 There exists a smooth diffeomorphism of the form
 \begin{align*}
  (\tilde x_{121},\tilde \xi_{121},\tilde \sigma_{12},\tilde \epsilon_{121})\mapsto \begin{cases}
                                                                x_{121} &=\mathcal X_{121}(\tilde x_{121},\tilde \xi_{121},\tilde \sigma_{12},\tilde \epsilon_{121}),\\
                                                                \xi_{121} &=\tilde \xi_{121} \mathcal S_{121}(\tilde x_{121},\tilde \xi_{121},\tilde \sigma_{12},\tilde \epsilon_{121}),\\
                                 \sigma_{12} &=\tilde \sigma_{12} \mathcal S_{121}(\tilde x_{121},\tilde \xi_{121},\tilde \sigma_{12},\tilde \epsilon_{121})^{-\frac{2k}{2k+1}},               \\
                                 \epsilon_{121} &= \tilde \epsilon_{121}\mathcal S_{121}(\tilde x_{121},\tilde \xi_{121},\tilde \sigma_{12},\tilde \epsilon_{121})^{-1}\end{cases}
 \end{align*}
  as well as a regular transformation of time, such that \eqref{eqs121} becomes
 \begin{equation}\eqlab{eqns121vf3}
\begin{aligned}
 \dot{\tilde x}_{121} &= 2\tilde x_{121} +\widetilde A_{121}(\tilde x_{121},\tilde \xi_{121},\tilde \sigma_{12},\tilde \epsilon_{121}),\\
 \dot{\tilde \xi}_{121}&=-\frac{2k+1}{k}\tilde \xi_{121},\\
 \dot{\tilde \sigma}_{12}&=2\tilde \sigma_{12},\\
 \dot{\tilde \epsilon}_{121}&=\frac{2k+1}{k}\tilde \epsilon_{121}. 
\end{aligned}
\end{equation}
Here $\mathcal X_{121}$, $\mathcal S_{121}$ and $\widetilde A_{121}$ are all smooth and satisfy $\mathcal X_{121}(0,0,0,0)=-1$, $\mathcal S_{121}(0,0,0,0)=1$ and 
\begin{align*}
\widetilde A_{121}(\tilde x_{121},\tilde \xi_{121},\tilde \sigma_{12},\tilde \epsilon_{121})=\mathcal O(\tilde \xi_{121}\tilde \epsilon_{121}^k\tilde \sigma_{12})
\end{align*}
respectively.
\end{lemma}
\begin{proof}
 The proof follows the proof of \lemmaref{qminusmap}, with only minor modifications.
\end{proof}
Consider \eqref{eqns121vf3}  and notice that $\alpha=\tilde \sigma_{12}^{2k+1}\tilde \xi_{121}^{2k}$ and $\epsilon =\tilde \xi_{121}\tilde \epsilon_{121}$ are still conserved in the tilde variables. We therefore drop the tildes and describe the transition map $\mathcal P_{121}^2$ from $\Pi_{121}^2:\,\xi_{121}=c_{in}$ to $\Pi^3_{121}:\,\epsilon_{121}=c_{out}$ by integrating these equations. 
\begin{lemma}\lemmalab{P1212}
 The transition map $\mathcal P_{121}^2$ is well-defined for 
 \begin{align}
 0\le \sigma_{12} \le \epsilon_{121}^{\frac{2k}{2k+1}}\alpha_0^{\frac{1}{2k+1}},\eqlab{sigma12cond}
\end{align}
 and $x_{121} \in \left[-c\left(\epsilon_{121} c_{out}^{-1}\right)^{\frac{2k}{2k+1}},c\left(\epsilon_{121} c_{out}^{-1}\right)^{\frac{2k}{2k+1}}\right]$ with $c>0$ and $\alpha_0$ small enough and given by  
  $(x_{121},c_{in},\sigma_{12},\epsilon_{121})\mapsto (\mathcal P_{121x}^2,(\epsilon_{121}c_{out}^{-1}) c_{in},\mathcal P_{12\sigma}^2,c_{out})$ 
  with 
  \begin{align*}
 \mathcal P_{12\sigma}^2(x_{121},\sigma_{12},\epsilon_{121}) &= \left(\epsilon_{121} c_{out}^{-1}\right)^{-\frac{2k}{2k+1}} \sigma_{12}\\
 \mathcal P_{121x}^2(x_{121},\sigma_{12},\epsilon_{121}) &= \left(\epsilon_{121} c_{out}^{-1}\right)^{-\frac{2k}{2k+1}} x_{121}+\mathcal O(\epsilon_{121}\mathcal P_{12\sigma}^2(x_{121},\sigma_{12},\epsilon_{121})). 
\end{align*}
The order of the remainder terms does not change upon differentiation with respect to $x_{121}$. Moreover, by \eqref{sigma12cond}
\begin{align}
  \mathcal P_{12\sigma}^2(x_{121},\sigma_{12},\epsilon_{121})\in (0,c_{out}^{\frac{2k}{2k+1}} \alpha_0^{\frac{1}{2k+1}}).\eqlab{sigmaest}
\end{align}
\end{lemma}
\begin{proof}
 Simple calculation.
\end{proof}
The analysis near $z^+$ is almost identical. In particular, although the local mapping near $z^-$ is expanding, the local mapping near $z^+$ contracts by the same order. 

\subsubsection*{The local map $\mathcal P_{loc}$}
Let $x_{in}$ denote the value of $x$ on $\Pi^0=\Pi_{in}$ of the grazing orbit of $Z_+$. Similarly, let $x_{out}$ be the corresponding value on $\Pi^7=\Pi_{out}$. 
$(\epsilon,\alpha)\in W_1(\epsilon_0,\alpha_0)$ implies that 
\begin{align*}
 0<\sigma_{12} \le \epsilon_{121}^{\frac{2k}{2k+1}}\alpha_0^{\frac{1}{2k+1}},
\end{align*}
in the $(\bar y=1,\bar \alpha_1=1,\bar r_{12}=1)_{121}$-chart and it is therefore consistent with \eqref{sigma12cond}. Consequently, by \lemmaref{P110} and \lemmaref{P1212}, we consider any $(\epsilon,\alpha)\in W_1(\epsilon_0,\alpha_0)$ with $\alpha_0>0$ small enough and $x$ in a small neighborhood of $x_{in}$:
\begin{align*}
 x-x_{in}\in \left[-c\epsilon^{\frac{2k}{2k+1}}\alpha^{\frac{2k}{2k+1}},c\epsilon^{\frac{2k}{2k+1}}\alpha^{\frac{2k}{2k+1}}\right],
\end{align*}
for some $c>0$. This leads to the following.
\begin{lemma}
 Let $x\mapsto \mathcal P_{loc,x}(x)$ denote the $x$-component of the map $\mathcal P_{loc}$ from $\Pi^0\rightarrow \Pi^7$. For any $x_2\in [-c,c]$, we then have
 \begin{align}
  \epsilon^{-\frac{2k}{2k+1}}\alpha^{-\frac{2k}{2k+1}} \left(\mathcal P_{loc,x}(x_{in} +\epsilon^{\frac{2k}{2k+1}}\alpha^{\frac{2k}{2k+1}}x_2)-x_{out}\right) =  \widehat{\mathcal P}_{122x}^3(x_2) +o(1),\eqlab{expandPloc}
 \end{align}
with $\widehat{\mathcal P}_{122x}^3=\psi_+  \circ \mathcal P_{122x}^3\circ \psi_-$ for some diffeomorphisms $\psi_\pm$, for $(\epsilon,\alpha)\in W_1(\epsilon_0,\alpha_0)$ with $\alpha_0,\epsilon_0>0$ sufficiently small.

The following can be said about $\psi_\pm$: \responsenew{For any $\delta>0$ and any $n\in \mathbb N$, there are constants $c_{in},c_{out},c>0$ such that $\vert \psi_\pm'-1\vert\le \delta$, $\vert \psi_\pm ^{(k)}\vert \le \delta$ for all $k=2,\ldots,n$}.

Moreover, the remainder term $o(1)$ is bounded by a constant $c_m(\alpha_0)\rightarrow 0$ for $\alpha_0\rightarrow 0$ in $C^m$, $m\in \mathbb N$ fixed. 
\end{lemma}
\begin{proof}
The proof is similar to \cite[Lemma 4.3]{2019arXiv190806781U}. In particular, we write $\mathcal P_{loc}$ as the composition of the maps $\mathcal P^{0-6}$ and the result then follows from \lemmaref{chini} and \lemmaref{P110} and \lemmaref{P1212}, near $q^-$ and $z^-$, along with similar results (these maps are basically the inverses (to leading order) of those in \lemmaref{P110} and \lemmaref{P1212}) near $q^+$ and $z^+$. The fact that the remainder term can be bounded by a constant $c_m(\alpha_0)$ follows from \eqref{sigmaest}. 
\end{proof}
From this lemma, it follows that $\widehat{\mathcal P}_{122x}^3$ also satisfies the estimates \eqref{P3122xest} on $x_2\in [-c,c]$. In fact, one can show (see \cite[Theorem 1.3]{2019arXiv190806781U} and \cite{kristiansen2023a}) that for any $l\in (0,1)$, there exists constants, including $c>0$, such that $(\widehat{\mathcal P}_{122x}^3)(x_2)$ can be extended in such a way that \eqref{expandPloc} holds and such that $(\widehat{\mathcal P}_{122x}^3)'(x_2)$ attains all values in $[-1+l,-l]$ while $(\widehat{\mathcal P}_{122x}^3)''(x_2)<0$. To do this one just extends $\mathcal P_{loc}$ through a redefinition of $\Pi^3$ and $\Pi^4$. Specifically, in the $(\bar y=1,\bar \alpha_1=1,\bar r_{12}=1)_{121}$-chart, we would consider $\Pi_{121}^3:\,x_{121}=-1\pm c_{out}$. 
%

We now write the regular map $\mathcal P_{glo}$ in a similar way. In fact, we focus on $\mathcal P_{glo}^{-1}$. Let $\mathcal P_{glox}^{-1}(x,\mu)$ be the $x$-component of $\mathcal P_{glo}^{-1}$. Since it is regular it depends smoothly on $x$ and on the unfolding parameter $\mu$. By assumption \ref{assumption:5}, we have that $\mathcal P_{glox}^{-1}(x_{in},0)=x_{out}$. Consequently, we obtain the following expansion: There exists $\nu_0\in (-1,0)$ and $\nu_1>0$ such that 
\begin{align}
 \epsilon^{-\frac{2k}{2k+1}} \alpha^{-\frac{2k}{2k+1}} \left(\mathcal P_{glox}^{-1}(x,\mu )-x_{out}\right) =\upsilon_0 x_2 +  \upsilon_1 \mu_2+\mathcal O(\epsilon^{\frac{2k}{2k+1}} \alpha^{\frac{2k}{2k+1}}),\eqlab{Pglo}
\end{align}
for
\begin{align}
 x=x_{in}+\epsilon^{\frac{2k}{2k+1}} \alpha^{\frac{2k}{2k+1}}x_2,\quad \mu = \epsilon^{\frac{2k}{2k+1}} \alpha^{\frac{2k}{2k+1}}\mu_2,\eqlab{x2mu2scaling}
\end{align}
The fact that $\nu_0\in (-1,0)$ follows from the fact that $\gamma_0$ is repelling, see \cite[Lemma 1.6]{2019arXiv190806781U}. Moreover, $\nu_1>0$ follows by assumption \ref{assumption:6}. 

To solve the fixed-point equation $\mathcal P_x(x,\mu)=x$, we therefore solve $\mathcal P_{locx}= \mathcal P_{glox}^{-1}$. By \eqref{expandPloc} and \eqref{Pglo} this gives
\begin{align*}
 \nu_0 x_2 + \nu_1 \mu_2 = \widehat{\mathcal P}_{122x}^3 (x_2)+o(1),
\end{align*}
setting $x$ and $\mu$ equal to the expressions in \eqref{x2mu2scaling}.
Seeing that $(\widehat{\mathcal P}_{122x}^3)' (x_2)$ attains all values in $[-1+l,-l]$ with $0<l<1+\nu_0<1$, we obtain a (locally unique) saddle-node of the fixed point by applying the implicit function theorem, see \cite[Lemma 4.5]{2019arXiv190806781U}. The proof in the present case is identical. In this way, we have completed the proof of \thmref{main2} (\ref{item1}).

\begin{figure}
\begin{center}
\includegraphics[width=.7\textwidth]{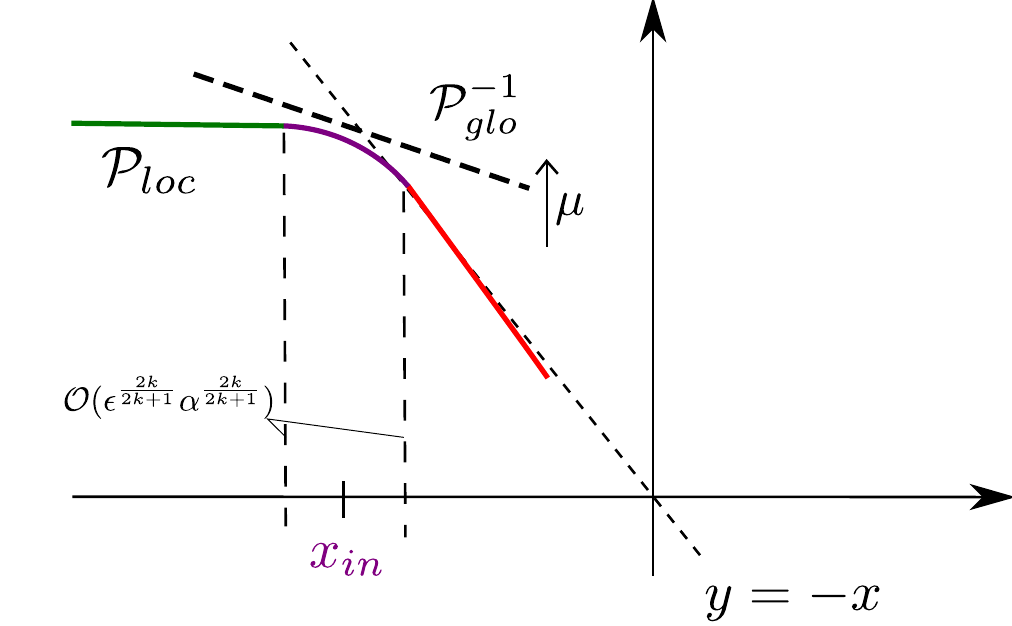}
  \end{center}
\caption{Illustration of the maps $\mathcal P_{loc}$ and $\mathcal P_{glo}^{-1}$ restricted to the slow manifold in 
  the case $(\epsilon,\alpha)\in W_1$, see \thmref{main2} (\ref{item1}). Here $x_{in}$ is the $x$-values of the orbit of $Z_+$ that grazes $\Sigma$ on $\Pi^0$. In the parameter regime $(\epsilon,\alpha)\in W_1$, the mapping $\mathcal P_{loc}$ is then dominated by the attraction towards the attracting center manifold $G_{122}$ on one side (green) $x\lesssim x_{in}$ and the dynamics of $Z_+$ (which itself is as close to $x\mapsto -x$ as desired upon adjusting the domains) on the other side $x\gtrsim x_{in}$. The transition inbetween (in purple), which extends over a $\mathcal O(\epsilon^{\frac{2k}{2k+1}}\alpha^{\frac{2k}{2k+1}})$-neighborhood of $x_{in}$, is described by the Chini-equation, see \eqref{chini}, and it is concave cf. \lemmaref{chini}. On the other hand, since $\gamma_0$ is repelling, it follows that $\mathcal P_{glo}$ is expanding. In particular, $\mathcal P_{glo}$ moves with nonzero speed for $\mu\response{\approx} 0$ by assumption \ref{assumption:6} and this therefore gives the saddle-node bifurcation of limit cycles as solutions of $\mathcal P_{glo}^{-1}=\mathcal P_{loc}$ when the two graphs are tangent at a point.
  }
 \figlab{QR}
              \end{figure}
%
 
\subsection{Dynamics on the blowup of $Q$}\seclab{Qblowupvf}
To prove \thmref{main2} (\ref{item2}), we consider the regime $(\epsilon,\alpha)\in W_2(\epsilon_0,\epsilon_1,\alpha_0)$ where $0< \alpha^{\frac{k+1}{k}}\epsilon_0\le \epsilon\le \alpha^{\frac{k+1}{k}}\epsilon_1$. In this case, the dynamics within $\bar \epsilon=0$ becomes relevant, recall \eqref{regime}. We therefore decompose $\mathcal P_{loc}$ in a different way, replacing $\Pi^3$ and $\Pi^4$ with $\widetilde \Pi^3$ and $\widetilde \Pi^4$, respectively, see \lemmaref{reflection}. In this way, since the mapping from $\widetilde \Pi^3$ and $\widetilde \Pi^4$ within $\epsilon_{121}=0$ is completely ``neutral'' with no contraction, see \eqref{reflection}, it follows that for all $x\in I_{in}$ and $\alpha_0$ small enough, so that the dynamics is uniformly bounded in the $(\bar y=1)_1$-chart, then $\mathcal P_{loc}$ is as close as desired (upon adjusting the domains) to a reflection $x\mapsto -x$. Consequently, there can be no saddle-node bifurcations of limit cycles in this chart within this parameter regime. 

In order to prove \thmref{main2} (\ref{item2}) and describe the chaotic dynamics, we have to follow the set $L^-$. Recall that this set is unbounded in the $(\bar y=1,\bar \alpha_1=1,\bar r_{12}=1)_{121}$-chart, see \eqref{L121pm}, so we follow it across the first blowup cylinder and towards the blowup of the point $Q$. 

In the following, we focus on the $(\bar \alpha=1,\bar \epsilon=1,\bar \nu_1\bar \epsilon_{1}=1)_{213}$-chart of the blowup of $Q$ and the equations \eqref{eqs213}, repeated here for with $Z_\pm$ as given in \eqref{Zpmvf}:
\begin{equation}\eqlab{eqs213new}
 \begin{aligned}
  \dot x &=\rho_{213}^{k+1}\nu_{213} \alpha \left[1+\mathcal O(x,\alpha)\right] ,\\
  \dot \nu_{213} &= \nu_{213}\left(\rho_{213}  Y_{213}(x,\nu_{213},p_{213},\rho_{213},\alpha) -\phi_+(\rho_{213}\nu_{213}^{-1}) \nu_{213}^{-k}-p_{213}\right),\\
  \dot p_{213} &=-\nu_{213} \left(\phi_+(\rho_{213}\nu_{213}^{-1}) \nu_{213}^{-k}+p_{213}\right) ,
 \end{aligned}
 \end{equation}
 and $\rho_{213}=\epsilon^{\frac{1}{k+1}}$,
 where 
 \begin{align}
  Y_{213}(x,\nu_{213},p_{213},\rho_{213},\alpha) = (2x +y g(x,y))p+1-p,\eqlab{Y213}
 \end{align}
 using assumption \ref{assumption:3},
for 
\begin{align*}
 y &=-\alpha(1+\rho_{213}^k p_{213})+\alpha \rho_{213}^{2k+1} \nu_{213},\\
 p &=1+\rho_{213}^k p_{213},
\end{align*}
on the right hand side of \eqref{Y213}. \response{Therefore for $\rho_{213}=0$, we find the critical manifold $R_{213}$ as a graph $p_{213}=-\beta \nu_{213}^{-k}$ over $\nu_{213}>0$. $R_{213}$} divides into an attracting part $R_{213,a}$ for \response{$\nu_{213}>\nu_{213,f}$} and a repelling part $R_{213,r}$ for $\nu_{213}<\nu_{213,f}$. 

For $x=0$ so that $Y_+=0$ on $y=0$, \response{see \eqref{Zpmvf}}, the fold curve $J_{213}$ given by $R_{213}\cap \{\nu_{213}=\nu_{213,f}\}$ no longer consist purely of jump points. \response{In particular, we will now show that it also includes folded singularities/canard points \cite{szmolyan_canards_2001}}: 

The system \eqref{eqs213new} is slow-fast (in nonstandard form) with respect to $\rho_{213}=0$ (which corresponds to $\epsilon=0$, recall \eqref{rho3eps}). 
The reduced problem on $R_{213}$ is given in \eqref{R213red} for $\rho_{213},\alpha\rightarrow 0$, repeated here for \response{convenience}:
\begin{equation}\eqlab{R213rednew}
\begin{aligned}
 x' &=0,\\
 \nu_{213}' &=2x\frac{\nu_{213}^{2}}{\nu_{213}-k\beta\nu_{213}^{-k}}.
\end{aligned}
\end{equation}
Consequently, for $\rho_{213}=\alpha=0$ the set $x=0$ is completely degenerate. We therefore proceed to blowup $x=\alpha=\rho_{213}=0$. We will only need one chart: Let $\alpha_{213}$, $x_{213}$ be defined by 
\begin{align}\eqlab{scalingalphax213}
 \begin{cases}
  \alpha &=\rho_{213}^{k}\alpha_{213},\\
  x &=\rho_{213}^k x_{213}.
 \end{cases}
\end{align}
Seeing that $\epsilon=\rho_{213}^{k+1}$, the scaling of $\alpha$ can be written as 
$\alpha = \epsilon^{\frac{k}{k+1}}\alpha_{213}$ which is therefore consistent with the regime $W_{2}(\epsilon_0, \epsilon_1,\alpha_0)$. In particular, $\epsilon_1>0$ sufficiently small in $W_2(\epsilon_0,\epsilon_1,\alpha_0)$ implies that $\alpha_{213}>0$ is large enough. 
Upon using \eqref{scalingalphax213}, we then obtain the following equations for the reduced problem:
\begin{equation}\eqlab{x213nu213red}
\begin{aligned}
\dot x_{213} &= \nu_{213}\alpha_{213},\\
\dot \nu_{213}&= \left[2x_{213} -\alpha_{213} g_0 +\beta \nu_{213}^{-k}\right]\frac{\nu_{213}^{2}}{\nu_{213}-k\beta\nu_{213}^{-k}},
\end{aligned}
\end{equation}
after having desingularized through division of the right hand side by $\rho_{213}^k$.  Here we have introduced $g_0:=g(0)$, see \eqref{Zpmvf}. Recall that $R_{213,a}$ corresponds to $\nu_{213}>\nu_{213,f}$ whereas $R_{213,r}$ corresponds to \response{$\nu_{213}<\nu_{213,f}$}. $\nu_3=\nu_{213,f}$, where the denominator of the right hand side of \eqref{x213nu213red} vanishes, is the degenerate set $J_{213}$. To analyze this situation, we proceed as usual \cite{szmolyan_canards_2001} by considering the desingularized system, obtained by multiplying the right hand side by $1-k\beta\nu_{213}^{-k-1}$:
\begin{equation}\eqlab{x213nu213redds}
\begin{aligned}
\dot x_{213} &= \alpha_{213} \left({\nu_{213}-k\beta\nu_{213}^{-k}}\right) ,\\
\dot \nu_{213}&=\left[2x_{213} -\alpha_{213} g_0 +\beta \nu_{213}^{-k}\right]\nu_{213}.
\end{aligned}
\end{equation}
On $R_{213,a}$ this multiplication corresponds to a time reparametrization, whereas on $R_{213,r}$ the direction of orbits of \eqref{x213nu213redds} have to be reversed to agree with \eqref{x213nu213red}. The dynamics of \eqref{x213nu213redds} is easy to study: For each $\alpha_{213}>0$, there exists a unique equilibrium at 
\begin{align}\eqlab{fs}
(x_{213,f},\nu_{213,f}),\quad x_{213,f}:=\frac12 \alpha_{213} g_0-\frac12 \beta\nu_{213,f}^{-k}.
\end{align}
It is a saddle; the linearization having the following eigenvalues eigenvalues:
\begin{align*}
 -\frac12v_{213,f} \pm \frac12\sqrt{8(k+1) \nu_{213,f} \alpha_{213}+\nu_{213}^2}.
\end{align*}
These eigenvalues are clearly real and of opposite sign for any $\alpha_{213}>0$. 
See \figref{pws_hysteresis_Rvf}.
\begin{figure}
\begin{center}
\includegraphics[width=.56\textwidth]{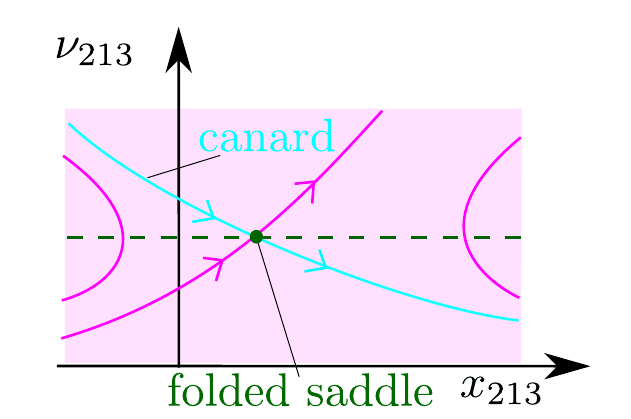}
 \end{center} \caption{Reduced dynamics on the critical manifold $R_{213}$ within the scaling regime defined by \eqref{scalingalphax213}. $R_{213}$ is attracting for $\nu_{213}>\nu_{213,f}$ and repelling for \response{$\nu_{213}<\nu_{213,f}$}. For any $\alpha_{213}>0$, there exists a (singular) canard \response{(for $\rho_{213}=0$, in cyan)} of the folded saddle \eqref{fs}, which perturbs \response{for all $0<\rho_{213}\ll 1$ within $(\epsilon,\alpha)\in W_2$ by slow-fast theory, see \propref{canard}}. The \response{(singular)} canard is a stable manifold of \response{the (folded) saddle of} the desingularized \response{reduced problem on $R_{213}$}. }
 \figlab{pws_hysteresis_Rvf}
              \end{figure}
In terms of the slow-fast system obtained from \eqref{eqs213new}, with $x$ and $\alpha$ scaled according to \eqref{scalingalphax213} \response{and $\rho_{213}>0$ being the small time scale separation parameter}:
\begin{equation}\eqlab{eqs213new2}
 \begin{aligned}
  \dot x_{213} &=\rho_{213}^{2k+1}\nu_{213} \alpha_{213} \left[1+\mathcal O(\rho_3^k)\right] ,\\
  \dot \nu_{213} &= \nu_{213}\left(\rho_{213}  Y_{213}(\rho_{213}^k x_{213},\nu_{213},p_{213},\rho_{213},\rho_{213}^k\alpha_{213}) -\phi_+(\rho_{213}\nu_{213}^{-1}) \nu_{213}^{-k}-p_{213}\right),\\
  \dot p_{213} &=-\nu_{213} \left(\phi_+(\rho_{213}\nu_{213}^{-1}) \nu_{213}^{-k}+p_{213}\right),
 \end{aligned}
 \end{equation}
the point $(x_{213},\nu_{213},p_{213})= (x_{213,f},\nu_{213,f},-\beta\nu_{213,f}^{-k})$ is \response{therefore a \textit{folded saddle}} \cite{szmolyan_canards_2001}. In particular, by \cite[Theorem 4.1]{szmolyan_canards_2001} we have the following:

 Consider the slow-fast system \eqref{eqs213new2}, having $R_{213,a}$ and $R_{213,r}$ as attracting and repelling (but noncompact) normally hyperbolic critical manifolds. \response{Fix appropriate compact submanifolds of $R_{213,a}$ and $R_{213,r}$; basically these sets have to contain an open subset of the singular canard in their enterior. Then by extending the resulting Fenichel slow manifolds (obtained as perturbations of these compacts sets) by the forward and backward flow, respectively, we obtain the \textnormal{extended attracting and repelling slow manifolds}.}
\begin{proposition}\proplab{canard}
 \response{Fix a compact interval $K\subset (0,\infty)$. 
Then there exists a $\rho_{2130}>0$ sufficiently small, such that for any $\alpha_{213}\in K$, $0\le \rho_{213}<\rho_{2130}$ there exists a canard trajectory as a transverse intersection of the extended attracting and repelling slow manifolds. The canard trajectory is an $\mathcal O(\sqrt{\rho_{213}})$-perturbation of the stable manifold of the (folded) saddle (cyan in \figref{pws_hysteresis_Rvf}).}
\end{proposition}

In fact, by working in separate charts, we can fix the Fenichel slow manifolds as extended versions of the slow manifolds $S_{\epsilon,\alpha}$ and $N_{\epsilon,\alpha}$, by applying the forward and backward flow to these manifolds. In this way, we can therefore extend the canard in \propref{canard} near $M$ on the second cylinder, see \figref{pws_hysteresis_visible2item2}. The canard has an unstable foliation along $N_{\epsilon,\alpha}$. By following this foliation back  towards $C$ on the $\bar {\bar y}$-positive side of $M$, see \figref{pws_hysteresis_visible2item2} (black orbits), we obtain a foliation of points on $S_{\epsilon,\alpha}$, specifically on $C$ for $\epsilon=\rho_{213}^{k+1},\alpha=\rho_{213}^k \alpha_{213}\rightarrow 0$, with $\alpha_{213}>0$ fixed. In fact, these points form a curve which is a graph over $y_2={\bar y}/{\bar \alpha} \in (0,1)$ in chart $(\bar \alpha=1)_2$, or equivalent a graph over $\alpha_1=y_{2}^{-1}\in (1,\infty)$ in chart $(\bar y=1)_1$, recall \eqref{cc12}. We focus on a compact subset $F_{\epsilon,\alpha}$ of this curve given by $\alpha_1\in [c_1,c_2]$ in the chart $(\bar y=1)_1$ with 
\begin{align}\eqlab{c1c2}
1<c_1<c_2,
\end{align}
fixed. For simplicity, we will frequently suppress $\epsilon$ and $\alpha$ and write $F_{\epsilon,\alpha}$ as $F$. 

We have the following regarding $F$: By applying the scaling \eqref{scalingalphax213} with $\epsilon=\rho_{213}^{k+1}$ to the system \eqref{reducedN22}, and upon using assumption \ref{assumption:3}, we obtain a desingularized flow on the manifold $M_{22}$ in the $(\bar \alpha=1,\bar \epsilon=1)_{22}$-chart:
\begin{equation}\eqlab{x213y22}
\begin{aligned}
  \dot x_{213} &=  \alpha_{213} \phi(y_{22}),\\
  \dot y_{22} &= -\frac{1-\phi(y_{22})}{\phi'(y_{22})},
\end{aligned}
\end{equation}
for $\rho_{213}\rightarrow 0$. Consequently, along the canard orbit following $M_{22}$, $x_{213}$ changes by an $\mathcal O(1)$-amount. Seeing that $x=\alpha\alpha_{213}^{-1} x_{213}$, we can therefore write the curve $F$ in the $(\bar y=1,\bar \alpha_1=1,\bar r_{12}=1)_{121}$-chart using the coordinates $(x_{121},\xi_{121},\epsilon_{121},\sigma_{12})$ on $C_1$, see \eqref{c121}, with $x=\sqrt{\alpha}x_{121}$, as follows 
\begin{align*}
 F_{121}:\,\sigma_{12}\in [c_1,c_2],\, x_{121}=0,\,r_{121}=\epsilon_{121}=0,
\end{align*}
 recall \eqref{xorderoncyl},
for $\rho_{213}\rightarrow 0$. In particular, we use that $x_{121}\sim \sqrt{\alpha}x_{213}\rightarrow 0$ for $\alpha\rightarrow 0$. See  \figref{pws_hysteresis_visible2item2}.
%

\begin{remark}\remlab{ass3main2}
\eqref{x213y22} is the only place in the proof of \thmref{main2}, where we use assumption \ref{assumption:3}. This assumption could easily be relaxed; we only need that the slow flow of $x_{213},y_{22}$ is well-defined on $M_{22}$ with $y_{22}$ decreasing.
\end{remark}

\begin{figure}
\begin{center}
\includegraphics[width=.9\textwidth]{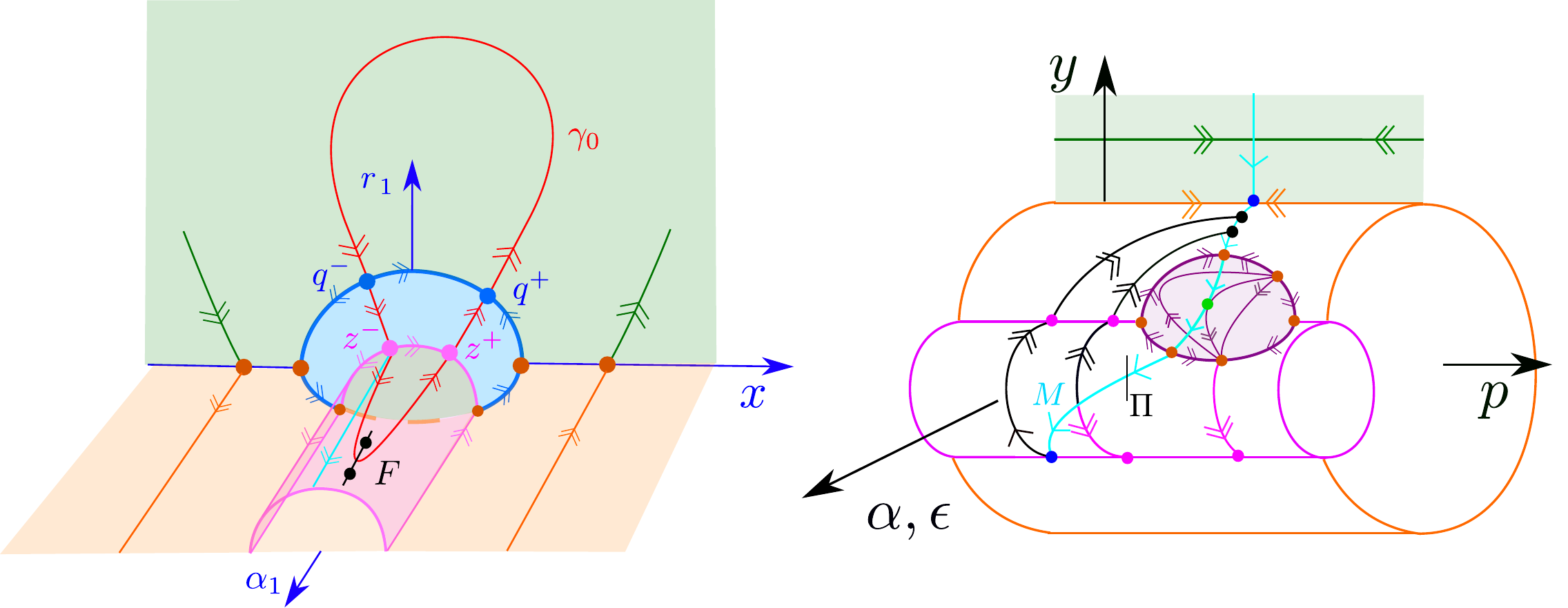}
 \end{center} \caption{Blowup dynamics for $(\epsilon,\alpha)\in W_2$. On the left, we show the reduced, desingularized dynamics on $C_1$ in the $(\bar y=1)_1$-chart upon application of the two consecutive blowup transformations, see \eqref{blowupvf1} and \eqref{blowup12cyl}. In red we have indicated the repelling limit cycle $\gamma_0$. It extends onto the cylinder $\bar \epsilon=0$, due to the blowup \eqref{blowup12cyl}, in the singular limit $\epsilon,\alpha\rightarrow 0$, with the limit understood within the parameter regime $W_2$. In comparison with \figref{pws_hysteresis_visible2}, we leave out the dynamics along $\bar \epsilon>0$ since this is not relevant for the regime $(\epsilon,\alpha)\in W_2$. At the same time, we also indicate the canard in cyan, see also \figref{pws_hysteresis_Rvf}, and the set $F$ (black) which is the set of base points on $C_1$, obtained by following the unstable foliation of the canard along $M$. On the right, we illustrate the dynamics in the projection also used in \figref{pws_hysteresis_final}, where the fast dynamics are also visible. Here we specifically indicate how the canard (cyan) extends across the two cylinders following $C$ on top and $M$ below. The section $\Pi$, transverse to $M$ and the canard, is used in the proof of \thmref{main2} (\ref{item2}).  }
 \figlab{pws_hysteresis_visible2item2}
              \end{figure}

%
%
\subsection{Completing the proof of \thmref{main2} (\ref{item2})}\seclab{main22}

Our strategy for completing the proof of \thmref{main2} is as follows:
Let $\mu\response{\approx} 0$ and consider $\epsilon_1>0$ small enough, so that the system has a repelling limit cycle, that when written in the $(\bar y=1,\bar \alpha_1=1,\bar r_{12}=1)_{121}$-chart intersects $\sigma_{12}=1$ transversally for each $(\epsilon,\alpha)\in W_2(0,\epsilon_1,\alpha_0)$ provided that $\alpha_0>0$ is small enough. The fact that this is possible follows from the analysis above in the $(\bar y=1,\bar \alpha_1=1,\bar r_{12}=1)_{121}$-chart, see the start of \secref{Qblowupvf}. Next, by decreasing $\alpha_0>0$ if necessary there exists an $0<\epsilon_0<\epsilon_1$ such that there is a canard trajectory for each $(\epsilon,\alpha)\in W_2$. In fact, the canard has an unstable foliation on the repelling side, which -- when carried across $N_{\epsilon,\alpha}$ near $M$ -- gives a twist-like return to the slow manifold $S_{\epsilon,\alpha}$. This induces the foliation of point on $S_{\epsilon,\alpha}$ given by the curve $F$. 

At the same time, since the limit cycle is repelling, we can track the canard backwards on $S_{\epsilon,\alpha}$, and conclude that the limit cycle is the $\alpha$-limit set of the canard. Upon increasing the interval $[c_1,c_2]\subset (1,\infty)$, recall \eqref{c1c2}, we can therefore ensure that the canard  transversally intersects the curve $F_{\epsilon,\alpha}$ on $S_{\epsilon,\alpha}$ in at least $n\in \mathbb N$ points for all $\rho_{213}>0$ small enough. The proof of the theorem then follows \response{\cite[Theorem 4.1]{kristiansen2021a}}, which is inspired by \response{\cite[Theorem 1]{haiduc2009a}} in a similar setting. In particular, we define a return map in the $(\bar \alpha=1,\bar \epsilon=1)_{22}$-chart using the scaling \eqref{scalingalphax213}, with $\epsilon=\rho_{213}^{k+1}$, defined on a section $\Pi_{22}$ transverse to $M_{22}$ and the canard. Since the expansion along $M$ is greater than the contraction along $C$, recall \remref{lambda}, we will study this mapping in backward time (so that $M$ becomes attracting and $C$ repelling). By flowing $N_{\epsilon,\alpha}\cap \Pi_{22}$ backwards near the canard, we obtain -- due to the transverse intersection of $S_{\epsilon,\alpha}$ and $N_{\epsilon,\alpha}$ along the canard -- a stable foliation of the canard on the $S_{\epsilon,\alpha}$-side. For each transverse intersection $i=1,\ldots,n$ of the canard with $F_{\epsilon,\alpha}$ on $S_{\epsilon,\alpha}$, we then further obtain a small subset of this foliation which, upon extension by the backward flow, eventually returns to $\Pi_{22}$ in a ``horizontal'' curve $H_i$ that extends an $\mathcal O(1)$ distance in the direction tangent to $N_{\epsilon,\alpha}\cap \Pi_{22}$ at the canard. At the same time, $H_i$ is exponentially close to $N_{\epsilon,\alpha}\cap \Pi_{22}$. This gives $n$ disjoint horizontal curves $H_{1},\ldots,H_n$, whose preimages are $n$ disjoint exponentially small intervals $I_1,\ldots,I_n$ on $N_{\epsilon,\alpha}\cap \Pi_{22}$. By the unstable foliation of $N_{\epsilon,\alpha}$, we obtain $n$ ``vertical strips'' $V_1,\ldots,V_n$ over $I_1,\ldots,I_n$. These strips get mapped to  horizontal strips that contain the curves $H_1,\ldots,H_n$, respectively. We call these thickened (although exponentially small) versions by the same symbols. 

This gives the basics of the horseshoe, with $n$ disjoint horizontal strips $H_1,\ldots,H_n$ and $n$ disjoint vertical strips $V_1,\ldots,V_n$ that intersect in $n\times n$ exponentially small squares. \thmref{main2} (\ref{item2}) therefore follows from the Conley-Moser theorem, see e.g. \cite[Theorem 25.2.1]{wiggins2003a}. In particular, the verification of the cone-properties of this theorem can be done in the exact same way as in the proof of \cite[Theorem 4.1]{kristiansen2021a}, \response{see \cite[p. 2387]{kristiansen2021a}}, using the foliations of the slow manifolds and the transverse intersection of $S_{\epsilon,\alpha}$ and $N_{\epsilon,\alpha}$ along the canard. \response{A similar verification (in the context of the forced van der Pol) can be found in \cite[Chapter 14.5]{kuehn2015}, and we therefore leave out further details.} %


\section{Discussion}
In this paper, we have described the dynamics of a new model \eqref{xypmodel} of hysteresis based upon singular perturbations. We \response{focussed upon} $\alpha>0$, \response{as this case corresponds to hysteresis}, and studied two scenarios where the associated PWS system (1) has stable sliding, see \thmref{main1}, and (2) has a repelling limit cycle grazing $\Sigma$ in the plane, see \thmref{main2}. In particular, in \thmref{main2} we identified two parameter regimes in the $(\epsilon,\alpha)$-plane, where the dynamics of \eqref{xypmodel} resembles regularization by smoothing and regularization by hysteresis, respectively. 

In future work, it would be interesting to perform the same analysis for $\alpha<0$, but also, in the case of the grazing bifurcation, to explore the transition between the two regimes of \thmref{main2}. Presumably there is an actual curve in the $(\epsilon,\alpha)$-plane along which  saddle-node limit cycles ``touch'' or ``grazes'' the foliation of points, described by $F$ in the singular limit and bounded by $\alpha_1=1$ from above, due to the twist and return to $S_{\epsilon,\alpha}$ away from the canard. An analysis of such a bifurcation scenario is interesting in its own right and in future work we aim to describe this in a simpler setting. 

%
%

%




\newpage
\bibliography{refs}

\begin{thebibliography}{10}

\bibitem{reves_regularization_2014}
C.~Bonet and T.~M. Seara.
\newblock {Regularization of sliding global bifurcations derived from the local
  fold singularity of Filippov systems}.
\newblock {\em Discrete and Continuous Dynamical Systems}, 36(7):3545--3601,
  2016.

\bibitem{rev2021a}
C.~Bonet and T.~M. Seara.
\newblock Two regularizations of the grazing-sliding bifurcation giving non
  equivalent dynamics.
\newblock {\em Journal of Differential Equations}, 332:219--277, 2022.

\bibitem{bonet2017a}
C.~Bonet, T.~M. Seara, E.~Fossas, and M.~R. Jeffrey.
\newblock A unified approach to explain contrary effects of hysteresis and
  smoothing in nonsmooth systems.
\newblock {\em Communications in Nonlinear Science and Numerical Simulation},
  50:142--168, 2017.

\bibitem{bossolini2017a}
E.~Bossolini, M.~Br{\o}ns, and K.~U. Kristiansen.
\newblock Canards in stiction: on solutions of a friction oscillator by
  regularization.
\newblock {\em {SIAM Journal on Applied Dynamical Systems}}, 16(4):2233--2258,
  2017.

\bibitem{de2016a}
P.~De~Maesschalck and S.~Schecter.
\newblock The entry-exit function and geometric singular perturbation theory.
\newblock {\em Journal of Differential Equations}, 260(8):6697--6715, 2016.

\bibitem{dumortier2006a}
F.~Dumortier, J.~Llibre, and J.~C. Art\'es.
\newblock {\em Qualitative theory of planar differential systems}.
\newblock Springer Berlin Heidelberg, 2006.

\bibitem{dumortier1996a}
F.~Dumortier and R.~Roussarie.
\newblock Canard cycles and center manifolds.
\newblock {\em Memoirs of the American Mathematical Society}, 121(577):1--96,
  1996.

\bibitem{fen3}
N.~Fenichel.
\newblock Geometric singular perturbation theory for ordinary differential
  equations.
\newblock {\em J. Diff. Eq.}, 31:53--98, 1979.

\bibitem{filippov1988differential}
A.F. Filippov.
\newblock {\em Differential Equations with Discontinuous Righthand Sides}.
\newblock Mathematics and its Applications. Kluwer Academic Publishers, 1988.

\bibitem{glendinning2010a}
P.~Glendinning and P.~Kowalczyk.
\newblock Micro-chaotic dynamics due to digital sampling in hybrid systems of
  filippov type.
\newblock {\em Physica D: Nonlinear Phenomena}, 239(1-2):58--71, 2010.

\bibitem{haiduc2009a}
R.~Haiduc.
\newblock Horseshoes in the forced van der pol system.
\newblock {\em Nonlinearity}, 22(1):213--237, 2009.

\bibitem{jeffrey_geometry_2011}
M.~R. Jeffrey and S.~J. Hogan.
\newblock The geometry of generic sliding bifurcations.
\newblock {\em {SIAM} Review}, 53(3):505--525, January 2011.

\bibitem{jelbart2021a}
S.~Jelbart, K.~U. Kristiansen, P.~Szmolyan, and M.~Wechselberger.
\newblock Singularly perturbed oscillators with exponential nonlinearities.
\newblock {\em Journal of Dynamics and Differential Equations}, pages 1--53,
  2021.

\bibitem{jelbart2021c}
S.~Jelbart, K.~U. Kristiansen, and M.~Wechselberger.
\newblock Singularly perturbed boundary-equilibrium bifurcations.
\newblock {\em Nonlinearity}, 34(11):7371--7314, 2021.

\bibitem{jelbart2021b}
S.~Jelbart, K.~U. Kristiansen, and M.~Wechselberger.
\newblock Singularly perturbed boundary-focus bifurcations.
\newblock {\em Journal of Differential Equations}, 296:412--492, 2021.

\bibitem{jones_1995}
C.K.R.T. Jones.
\newblock {\em Geometric Singular Perturbation Theory, Lecture Notes in
  Mathematics, Dynamical Systems (Montecatini Terme)}.
\newblock Springer, Berlin, 1995.

\bibitem{Gucwa2009783}
I.~Kosiuk and P.~Szmolyan.
\newblock Geometric singular perturbation analysis of an autocatalator model.
\newblock {\em Discrete and Continuous Dynamical Systems - Series S},
  2(4):783--806, 2009.

\bibitem{kristiansen2017a}
K.~U. Kristiansen.
\newblock {Blowup for flat slow manifolds}.
\newblock {\em Nonlinearity}, 30(5):2138--2184, 2017.

\bibitem{2019arXiv190806781U}
K.~U. Kristiansen.
\newblock The regularized visible fold revisited.
\newblock {\em Journal of Nonlinear Science}, pages 1--49, 2020.

\bibitem{kristiansen2021a}
K.~U. Kristiansen.
\newblock A stiction oscillator under slowly varying forcing: Uncovering small
  scale phenomena using blowup.
\newblock {\em {SIAM Journal on Applied Dynamical Systems}}, 20(4):2359--2390,
  2021.

\bibitem{kristiansen2018a}
K.~U. Kristiansen and S.~J. Hogan.
\newblock {Resolution of the piecewise smooth visible-invisible two-fold
  singularity in R3 using regularization and blowup}.
\newblock {\em Journal of Nonlinear Science}, 29(2):723--787, 2018.

\bibitem{uldall2021a}
K.~U. Kristiansen and P.~Szmolyan.
\newblock Relaxation oscillations in substrate-depletion oscillators close to
  the nonsmooth limit.
\newblock {\em Nonlinearity}, 34(2):1030--1083, 2021.

\bibitem{kristiansen2023a}
K.~Uldall Kristiansen.
\newblock Correction: The regularized visible fold revisited.
\newblock {\em Journal of Nonlinear Science}, 33(4), 2023.

\bibitem{krupa_extending_2001}
M.~Krupa and P.~Szmolyan.
\newblock Extending geometric singular perturbation theory to nonhyperbolic
  points - fold and canard points in two dimensions.
\newblock {\em {SIAM} Journal on Mathematical Analysis}, 33(2):286--314, 2001.

\bibitem{kuehn2015}
C.~Kuehn.
\newblock {\em {Multiple Time Scale Dynamics}}.
\newblock Springer-Verlag, Berlin, 2015.

\bibitem{Kuznetsov2003}
{\mbox{Yu.~A}}.~Kuznetsov, S.~Rinaldi, and A.~Gragnani.
\newblock One parameter bifurcations in planar {F}ilippov systems.
\newblock {\em Int. J. Bif. Chaos}, 13:2157--2188, 2003.

\bibitem{Llibre09}
J.~Llibre, P.~R. da~Silva, and M.~A. Teixeira.
\newblock {Study of singularities in nonsmooth dynamical systems via singular
  perturbation}.
\newblock {\em {SIAM} Journal on Applied Dynamical Systems},
  {8}({1}):{508--526}, {2009}.

\bibitem{trifonov2011a}
F.~W.~J. Olver, D.~W. Lozier, R.~F. Boisvert, and C.~W. Clark.
\newblock {NIST Handbook of Mathematical Functions }.
\newblock {\em Journal of Geometry and Symmetry in Physics}, pages 99--104,
  2011.

\bibitem{Sotomayor96}
J.~Sotomayor and M.~A. Teixeira.
\newblock Regularization of discontinuous vector fields.
\newblock In {\em Proceedings of the International Conference on Differential
  Equations, Lisboa}, pages 207--223, 1996.

\bibitem{szmolyan_canards_2001}
P.~Szmolyan and M.~Wechselberger.
\newblock Canards in $\mathbb{R}^3$.
\newblock {\em J. Diff. Eq.}, 177(2):419--453, December 2001.

\bibitem{szmolyan2004a}
P.~Szmolyan and M.~Wechselberger.
\newblock {Relaxation oscillation in R3}.
\newblock {\em Journal of Differential Equations}, 200(1):69--104, 2004.

\bibitem{wechselberger2020a}
M.~Wechselberger.
\newblock {\em Geometric singular perturbation theory beyond the standard
  form}.
\newblock Springer Nature Switzerland, 2020.

\bibitem{wiggins2003a}
S.~Wiggins.
\newblock {\em Introduction to Applied Nonlinear Dynamical Systems and Chaos},
  volume~2.
\newblock Springer New York, 2003.

\end{thebibliography}
\bibliographystyle{plain}
 \end{document}